\documentclass{article}
\usepackage{amsfonts, amsthm, amsmath, mathtools, amssymb, graphicx, comment, xcolor, enumitem, mathabx,xfrac, mathrsfs, indentfirst, esint, bm, tensor, setspace}
\usepackage{latexsym}
\usepackage{amssymb}
\usepackage{amsxtra}
\usepackage{geometry}

\newtheorem{lemma}{Lemma}[section]
\newtheorem{thm}{Theorem}[section]

\newtheorem{defn}{Definition}[section]
\newtheorem{cor}{Corollary}[section]
\newtheorem{rmk}{Remark}[section]
\newcommand{\ZZ}{\mathbb{Z}}
\newcommand{\chiw}{\chi_{\w}}

\newcommand{\R}{\mathbb{R}}
\newcommand{\C}{\mathbb{C}}

\newcommand{\nab}{\nabla}

\newcommand{\W}{\Omega}
\newcommand{\Kb}{\mathbb{K}}
\newcommand{\w}{\omega}
\newcommand{\be}{\beta}
\newcommand{\Pol}{\mathfrak{P}}
\newcommand{\FS}{\mathcal{S}}
\newcommand{\B}{\mathcal{B}}
\newcommand{\Cyl}{\mathcal{C}}
\newcommand{\Obs}{\mathcal{O}}

\newcommand{\m}{\mathbf{m}}
\newcommand{\Comp}{\mathfrak{C}}

\newcommand{\We}{\mathrm{We}}

\newcommand{\pex}{p_{\mathrm{ext}}}
\newcommand{\pext}[1]{p_{\mathrm{ext},#1}}
\newcommand{\BR}{\mathbf{BR}}
\newcommand{\F}{\mathfrak{F}}
\newcommand{\Tta}{\Theta}

\newcommand{\Y}{\mathbf{Y}}
\newcommand{\Z}{\mathbf{Z}}
\newcommand{\vel}{\mathbf{u}}
\newcommand{\Wt}{\mathbf{W}}
\newcommand{\Wo}{\widetilde{\Wt}}
\newcommand{\V}{\mathbf{V}_0}

\newcommand{\yb}{\mathbf{br}}
\newcommand{\X}{\mathfrak{X}}
\newcommand{\E}{\mathcal{E}}
\newcommand{\Edamp}{\mathcal{E}_{\mathrm{damped}}}
\newcommand{\al}{\alpha}

\newcommand{\alp}{\al^\prime}
\newcommand{\eb}{\mathfrak{e}}
\newcommand{\La}{\Lambda}
\newcommand{\y}{\gamma}

\newcommand{\pr}{\prime}
\newcommand{\ca}{\mathfrak{c}}

\newcommand{\pv}{\ \mathrm{pv}}
\newcommand{\ut}{\mathbf{\hat{t}}}
\newcommand{\un}{\mathbf{\hat{n}}}

\newcommand{\ttanorm}{\norm{\tta}_{H^s}}
\newcommand{\ynorm}{\norm{\y}_{H^{\sHalf}}}

\newcommand{\Vzeronorm}{\abs{V_0}}
\newcommand{\benorm}{\norm{\be}_{H^1}}
\newcommand{\vphi}{\varphi}
\newcommand{\vphic}{\varphi_{\mathrm{cyl}}}

\newcommand{\half}{\sfrac{1}{2}}
\newcommand{\sHalf}{{s-\sfrac{1}{2}}}
\newcommand{\D}{\Delta}
\newcommand{\de}{\delta}
\newcommand{\delt}{\tilde{\de}}

\newcommand{\ee}{\varepsilon}
\newcommand{\g}{\mathbf{g}}

\newcommand{\z}{\zeta}
\newcommand{\tta}{\theta}

\newcommand{\Ut}{\z_t}
\newcommand{\p}{\partial}

\newcommand{\abs}[1]{\left\lvert#1\right\rvert}
\newcommand{\sgn}{\operatorname{sgn}}

\newcommand{\set}[1]{\left\{#1\right\}}

\newcommand{\T}{\mathbb{T}}
\newcommand{\N}{\mathbb{N}}
\newcommand{\Nonloc}{N}
\newcommand{\Loc}{F}
\newcommand{\h}{\mathfrak{h}}

\newcommand{\norm}[1]{\left\|#1\right\|}
\newcommand{\comm}[2]{\left[#1,#2\right]}

\newcommand{\Test}{\mathcal{D}}
\newcommand{\Dist}{\mathcal{D}^\pr}

\newcommand{\hnorm}[2]{\norm{#1}_{\dot{H}^{#2}}}
\newcommand{\pnorm}[2]{\norm{#1}_{\Lp{#2}}}

\def\avint{\mathop{\mathchoice{\,\rlap{-}\!\!\int}
                              {\rlap{\raise.15em{\scriptstyle -}}\kern-.2em\int}
                              {\rlap{\raise.09em{\scriptscriptstyle -}}\!\int}
                              {\rlap{-}\!\int}}\nolimits}
                              
\newcommand{\Norm}[2]{\norm{#1}_{H^{#2}}}

\newcommand{\Lip}{\mathrm{Lip}}

\newcommand{\Int}{\int_0^{2\pi}}                           

\newcommand{\Kop}{K}

\DeclareMathOperator{\Rea}{\mathfrak{Re}}

\DeclareMathOperator{\HT}{\mathcal{H}}
\DeclareMathOperator{\FT}{\mathcal{F}}
\DeclareMathOperator{\K}{\mathscr{K}}

\DeclareMathOperator{\bigo}{\mathcal{O}}

\DeclareMathOperator{\id}{\mathrm{id}}
\DeclareMathOperator{\J}{\mathcal{J}_\de}
\DeclareMathOperator{\Jt}{\mathcal{J}_{\delt}}
\DeclareMathOperator{\diver}{\mathrm{div}}
\DeclareMathOperator{\curl}{\mathrm{curl}}

\newcommand{\Lp}[1]{{L^{#1}}}
\newcommand{\wnorm}{\norm{\w}_{H^1}}

\newcommand{\Ef}{\E_{\mathrm{flow}}}

\let\originalleft\left
\let\originalright\right
\renewcommand{\left}{\mathopen{}\mathclose\bgroup\originalleft}
\renewcommand{\right}{\aftergroup\egroup\originalright}

\numberwithin{equation}{section}

\title{Local Well-Posedness of the Gravity-Capillary Water Waves System in the Presence of Geometry and Damping}
\author{Gary Moon}
\date{}

\begin{document}

\pagebreak

\maketitle

\begin{abstract}
We consider the gravity-capillary water waves problem in a domain $\W_t \subset \T \times \R$ with substantial geometric features. Namely, we consider a variable bottom, smooth obstacles in the flow and a constant background current. We utilize a vortex sheet model introduced by Ambrose, et.\@ al.\@ in \cite{AMEA}, which is an extension of the vortex sheet model studied in \cite{A1, AMas1}. We show that the water waves problem is locally-in-time well-posed in this geometric setting and study the lifespan of solutions. We then add a damping term and derive evolution equations that account for the damper. Ultimately, we show that the same well-posedness and lifespan results apply to the damped system. We primarily utilize energy methods; particularly our approach here closely follows the approach taken in \cite{A1}.
\end{abstract}

\section{Introduction}

\let\thefootnote\relax\footnote{AMS Subject Classifications: 76B03, 76B15, 76B45, 35Q31, 35Q35.}

\let\thefootnote\relax\footnote{The author was partially supported by NSF CAREER Grant DMS-1352353 and NSF Applied Math Grant DMS-1909035.}

The gravity-capillary water waves problem concerns the evolution of the velocity field $\vel$ and the pressure $p$ of an inviscid, incompressible, irrotational fluid, as well as the fluid-vacuum interface $\FS_t$ under the influence of gravity and surface tension. The ambient setting is $d$-dimensional Euclidean space, with the physically relevant dimensions being $d=2$ and $d=3$. We shall restrict ourselves to consideration of the $2d$ problem.

We shall take the fluid domain $\W_t$ to be a subset of $\T \times \R$, where $\T \coloneqq \R/2\pi\ZZ$. The dynamics of the flow are governed by the irrotational free-surface Euler equations; that is, the incompressible, irrotational Euler equations, coupled with two boundary conditions on the interface (the so-called kinematic and dynamic boundary conditions):
\begin{equation}
\label{eqn:Free-SurfaceEulerEqns}
\begin{dcases}
\p_t\vel + (\vel\cdot\nab)\vel = -\nab\frac{p}{\rho_0} - \g & \text{in } \W_t\\
\diver\vel = 0 &\text{in } \W_t\\
\curl\vel = 0 &\text{in } \W_t\\
\p_t + \vel\cdot\nab \text{ is tangent to } \bigcup_t \FS_t \times \set{t} \subset \T_x\times\R_y\times\R_t\\
p = -\tau \kappa(\z) & \text{on } \FS_t
\end{dcases}.
\end{equation}
In \eqref{eqn:Free-SurfaceEulerEqns}, $\rho_0$ is the constant fluid density, $\g \coloneqq (0,g)$ with $g$ being acceleration due to gravity, $\tau$ is the coefficient of surface tension and $\kappa(\z)$ is the curvature of the interface. We can, by rescaling, assume the fluid has unit density $\rho_0 = 1$ and we shall henceforth make this assumption. In the setting considered here, the interface is described parametrically by $\z(\al,t) = \xi(\al,t) + i\eta(\al,t)$. The curvature is then given by
\begin{equation}
\label{eqn:CurvatureDef}
\kappa(\z) = \frac{\abs{\xi_\al\eta_{\al\al} - \eta_\al\xi_{\al\al}}}{(\xi_\al^2 + \eta_\al^2)^\frac{3}{2}}.
\end{equation}
Notice that we have taken the density of the fluid to be $\rho \equiv 1$. We impose free-slip boundary conditions on the remaining portions of $\p\W_t$:
\begin{equation}
\label{eqn:Free-SlipBC}
\vel\cdot\un = 0 \text{ on } \p\W_t\setminus\FS_t.
\end{equation}
The boundary condition \eqref{eqn:Free-SlipBC} is also commonly referred to as a solid-wall or no-penetration boundary condition.

Given the assumption of irrotationality, the free-surface Euler equations can be reduced to a system on the free surface and, beginning from \eqref{eqn:Free-SurfaceEulerEqns}, there are many ways to formulate the water waves problem. These include the vortex sheet formulation (e.g., \cite{CCG1}), the Zakharov-Craig-Sulem formulation (e.g., \cite{ABZ1}), holomorphic coordinates and the conformal method (e.g., \cite{HIT1}), other Lagrangian formulations (e.g., \cite{CouShk}), a coordinate-free geometric formulation (e.g., \cite{EbinMar1}) and various other formulations (e.g., \cite{AFM1}). See Chapter 1 of \cite{Lan1} for an overview of formulations of the water waves problem. We shall utilize a vortex sheet formulation. Vortex sheet formulations are a popular choice for numerical modeling of water waves \cite{BMO, HLS1, HLS2, BHL2, B3, HZ1}. For example, the representation of the Dirichlet-Neumann map via layer potentials is well adapted to the needs of numerical computation \cite{WV}.

Though we present analytical results here, this paper is substantially motivated by numerical work. The formulation which we use here is a vortex sheet model for water waves in the presence of geometry proposed by Ambrose, et.\@ al.\@ in \cite{AMEA}. The objective of the authors in \cite{AMEA} was to obtain accurate and efficient algorithms for numerically solving the two-dimensional, free-surface Euler equations in a geometric setting. The model allows for variable topography, smooth obstacles in the fluid flow and a (constant) background current. Utilizing this formulation, our first objective is to show that the water waves system is locally well-posed in this more geometric setting and to study the lifespan of solutions. We note briefly that when we say lifespan, we simply mean a timescale on which the energy of the solutions remains bounded and so the solutions persist. We do not claim that the solution is of any particular size or small to any given order (e.g., that the size of the solution remains of the same order as the initial data).

Another important concept in the numerical simulation of water waves is that of damping. It is often of interest to study water waves on an (effectively) unbounded domain, such as on the open ocean. However, when carrying out numerical experiments one is forced to truncate the domain, introducing an artificial boundary, and this can create problems. In particular, one wants to ensure that waves do not reflect off of the artificial boundary, propagate back into the domain and create interference. There are a number of approaches designed to achieve this outcome. One popular approach is to add a damping term to the system. The damping term is designed to dissipate energy in a neighborhood of the boundary, which causes outgoing waves to decay.

The form of damping we shall consider, which we call Clamond damping, was first introduced in the numerical work of Clamond, et.\@ al.\@ in the setting of $3d$ water waves \cite{ClaEtAl}. Clamond damping is a type of modified sponge-layer, which is effected via the application of an external pressure at a portion of the interface:
\begin{equation}
\label{eqn:ClamondDamper}
\pex \coloneqq \p_x^{-1}(\chi_\w\p_x\vphi).
\end{equation}
In the above, $\w \subset [0,2\pi)$ is the connected interval on which we damp the fluid, $\chi_\w$ is a smooth, non-negative cut-off function supported on $\w$ and $\vphi$ is the velocity potential. Equation \eqref{eqn:ClamondDamper} is simply the $2d$ analogue of the $3d$ damper given in \cite{ClaEtAl}:
\begin{equation}
\label{eqn:3dClamondDamper}
\pext{3d} \coloneqq \nab^{-1}\cdot(\chi_\w\nab\vphi) - b(t),
\end{equation}
where $b$ is a Bernoulli constant. Technically, we should also have a Bernoulli constant in \eqref{eqn:ClamondDamper}, however we have chosen to ignore this term. We are able to do so because, as a function of time alone, the Bernoulli constant $b$ will have no effect on the energy estimates which will be the focus of our analysis. This should in no way be taken to mean that the Bernoulli constant is generically unimportant. On the contrary, the Bernoulli constant can be quite important computationally. Ultimately, from a numerical perspective, the importance of the Bernoulli constant and how one treats it will depend on what method one uses to resolve the equations. For further details, see \cite{ClaEtAl}.

Numerical experiments have shown Clamond damping to be remarkably effective \cite{ClaEtAl}. However, Clamond damping is a linear phenomenon and the question of why it performs so well for the full (nonlinear) water waves system is still open. For example, there is no proof that Clamond damping dissipates energy. Given that Clamond damping is so highly effective numerically, it is our belief that a more thorough understanding of this damping mechanism is important. Our second objective is to attempt to initiate this process of better understanding Clamond damping. In particular, we shall show, again using the vortex sheet formulation described above, that the water waves system subject to Clamond damping is locally-in-time well-posed and to consider the timescales on which solutions persist.

\subsection{A Brief History of the Water Waves Problem}

Before proceeding to discuss the results of this paper, we give a brief overview of prior results on the water waves problem (focusing primarily on well-posedness and the lifespan of solutions), vortex methods and the damping of water waves. We begin by reviewing the literature on the well-posedness of the water waves problem. Given the breadth and depth of the literature on the mathematical study of water waves, we give only a (proper) subset of the existing results. Given that we consider the water waves system with surface tension here, we shall, when making choices about results to discuss, be biased towards results that consider surface tension.

The water waves problem belongs to the class of problems known as free boundary problems, which are notoriously challenging to analyze. The earliest well-posedness results made strong assumptions on the Cauchy data and the geometry of the domain. Broadly, they considered analytic data and analytic geometry, or perturbative data in Sobolev spaces and perturbative geometry, including infinite depth. By perturbative, we mean small perturbations of flat, so a perturbative assumption would usually involve assuming that the initial configuration of the free boundary is a small perturbation of still water and the bottom, if present, is a small perturbation of flat. As an example of the former, Kano-Nishida proved well-posedness of the gravity water waves problem with analytic Cauchy data and a flat bottom in \cite{KanNis}. An example of work in the latter group (i.e., those working in Sobolev spaces and utilizing perturbative assumptions) would be the groundbreaking work of Nalimov \cite{Nal}, which, to the author's knowledge, represents the earliest well-posedness result on the full water waves system. One notable benefit of the smallness assumption in the case of gravity waves is that it implies that the Taylor sign condition holds:
\begin{equation}
\label{eqn:TaylorSign}
-\p_\un p \geq c_0 > 0 \text{ on } \FS_t,
\end{equation}
where $\un$ is the outer unit normal on $\FS_t$. The condition \eqref{eqn:TaylorSign} is critical for the well-posedness of the gravity water waves problem. In fact, it is known that the gravity water waves problem may be ill-posed if \eqref{eqn:TaylorSign} fails \cite{Ebin1}.

The need for a smallness assumption was first overcome for infinite-depth water waves. In her seminal work, Wu utilized Lagrangian coordinates and the conformal method to show that the gravity water waves problem is well-posed by proving that \eqref{eqn:TaylorSign} always holds as long as the free surface is non-self-intersecting \cite{Wu2} (this analysis was extended to $3d$ via the use of Clifford analysis \cite{Wu3}). An alternative proof, utilizing a vortex sheet framework, is given by Ambrose-Masmoudi in \cite{AMas1, AMas3}. On the other hand, Beyer-G\"{u}nther showed well-posedness of the Cauchy problem for a capillary drop noting that their methods extend to the well-posedness of capillary waves over an infinite-depth fluid \cite{BeyGun}. Iguchi and Ambrose independently provided proofs, via distinct approaches, of the well-posedness of the two-dimensional gravity-capillary water waves problem \cite{A1, Igu1}. Ambrose-Masmoudi prove a similar result in the case $d=3$ \cite{AMas2}. These results have been extended to allow for vorticity and rough Cauchy data (e.g., see \cite{Lin1, CouShk, ShaZen1, ShaZen3} for rotational water waves and \cite{HIT1,AIT1} for rough Cauchy data).

The aforementioned work of Iguchi actually proved that the two-dimensional gravity-capillary water waves problem is well-posed in the finite depth setting with variable bathymetry \cite{Igu1}. Well-posedness of the gravity water waves problem in the presence of topography was shown by Lannes in \cite{Lan2}, utilizing Eulerian coordinates. The work of Lannes was extended by Ming-Zhang to account for the effects of surface tension \cite{MingZhang}. The work of Alazard-Burq-Zuily extended this work by allowing for low-regularity initial data and very rough topography (in fact, the only restriction on the geometry was a non-cavitation assumption) \cite{ABZ1, ABZ2}. As was the case for the infinite-depth theory, the above finite-depth results have been extended in numerous directions: non-zero vorticity \cite{CasLan1}, emerging bottom \cite{MingWang1, MingWang2, MingWang3}, rougher Cauchy Data \cite{H-GIT1}, Coriolis forcing \cite{Mel1} and so on. In addition to the question of local-in-time well-posedness, there are myriad interesting questions related to the water waves problem which are the focus of active research.

Another related question regards the lifespan of solutions to the water waves problem, usually in the small-data setting. Here some interesting results are provided by Hunter-Ifrim-Tataru and their collaborators, who have applied their ``modified energy method'' to the water waves system. The modified energy method has been applied to infinite-depth gravity waves \cite{HIT1}, infinite-depth capillary waves \cite{IT1}, infinite-depth gravity waves with constant (non-zero) vorticity \cite{IT3} and finite-depth gravity waves \cite{H-GIT1}. The main idea of the ``modified energy method'', as applied to a quadratically nonlinear equation, is to use a normal form transformation to construct a modified energy functional which satisfies cubically nonlinear estimates. As such, when considering a quadratically nonlinear, quasilinear equation, the modified energy estimates can be used to prove local well-posedness with a cubic lifespan. Normal form methods can also be applied more directly to obtain long-time existence of solutions to the water waves system (e.g., see \cite{BFF1}).

While we are primarily concerned with lifespan as a function of the size of the initial data, in the small data regime, there is another collection of long-time existence results. These results measure the lifespan in terms of various dimensionless parameters used to characterize the flow (e.g., the shallowness parameter $\mu \coloneqq \frac{H^2}{\lambda^2}$, where $H$ is the characteristic water depth and $\lambda$ is the characteristic wavelength in the longitudinal direction) and are used to provide rigorous justification of various simplified models in asymptotic regimes. M\'{e}sognon-Gireau proved that the gravity-capillary water waves problem is well-posed on a large timescale in the presence of large variations in topography \cite{BMG1}, which extended earlier work of Alvarez-Samaniego and Lannes on large-time existence of gravity waves \cite{AlvLan1}.

Intimately related to the question of lifespan of solutions, there is the question of global or almost-global regularity of solutions, under the assumption of small, localized, smooth initial data. Additionally, to the best of the author's knowledge, all almost-global and global well-posedness results require the assumption of vanishing vorticity in the bulk of the fluid domain. Most of these results are in the setting of infinite depth, however global regularity in the finite-depth setting has been considered very recently. Further, such results tend to be easier to obtain in $3d$ as opposed to $2d$ due to better rates of decay in higher dimension. In addition, if one has a global solution to the water waves system, it is desirable to understand its long-time asymptotic behavior, such as whether the solution scatters to a linear solution as $t \to +\infty$.

In three dimensions, the global regularity problem has been resolved for the gravity, capillary and gravity-capillary water waves problems. For example, in \cite{GMS1}, Germain-Masmoudi-Shatah used their method of space-time resonances to prove global regularity of the capillary water waves problem in $3d$, where the authors also prove that the global solution scatters to a solution of the linearized problem. Deng-Ionescu-Pausader-Pusateri utilized the paradifferential framework to obtain a global solution to the $3d$ gravity-capillary water waves system and show that this solution scatters in \cite{DIPP1}. In dimension two, the global regularity problems for gravity and capillary waves have been resolved. The interested reader can consult \cite{AD2} for gravity waves and \cite{IT1} for capillary waves (both prove modified scattering results); these papers contain references to other presentations of the corresponding results. To the author's knowledge, the best result for the $2d$ gravity-capillary water waves system is the almost-global well-posedness result of Berti-Delort \cite{BD1}. Some of the above infinite-depth results have been extended to hold in the context of flat geometry in $3d$ (e.g., see \cite{Wang2} for a proof of the existence of a global solution to the capillary water waves system).

The water waves problem is a highly active area of research and the above outlined questions are far from the only questions which one can ask about the water waves problem. For example, there is the question of providing rigorous mathematical justifications for the various models used to describe the dynamics of water waves in different asymptotic regimes (e\@.g\@., KdV, Green-Naghdi and the cubic NLS). Further, there are questions of the existence of soliton solutions and the properties of steady waves (e.g., the famous Stokes conjecture). However, given that we are primarily concerned here with issues of local-in-time well-posedness and lifespan, we shall not go into further detail about these other issues.

\subsection{Previous Results on Vortex Sheets and the Vortex Sheet Formulation of the Water Waves Problem}

As discussed above, there are numerous ways to formulate the water waves problem (various coordinate systems, parameterizations of the interface and so on). The model which we consider utilizes the vortex sheet framework. The classical vortex sheet problem (also called the Kelvin-Helmholtz problem) considers the interface between two incompressible, inviscid, irrotational, density-matched fluids moving past each other in two dimensions, neglecting the effects of surface tension. In such a scenario the vorticity is concentrated entirely along the interface due to the jump in tangential velocity (while the normal velocity is continuous).

It has long been known that the Kelvin-Helmholtz problem is ill-posed in the usual sense due to the well-known Kelvin-Helmholtz instability (see, e.g., \cite{CafOre1}), however it is worth noting that the Kelvin-Helmholtz problem is nevertheless well-posed in analytic function spaces \cite{SSBF1}. Importantly, these ill-posedness results neglect the effects of surface tension, which exhibits a smoothing effect. When surface tension is incorporated, high-frequency Fourier modes remain bounded in the linearization. Building off of this, Beale-Hou-Lowengrub showed that the linearized two-dimensional vortex sheet problem with surface tension is well-posed, even far from equilibrium \cite{BHL1} (see \cite{HTZ1} for the corresponding result in three dimensions). It was proven by Iguchi-Tanaka-Tani in \cite{ITT1} that the (nonlinear) vortex sheet problem with surface tension is well-posed subject to a perturbative hypothesis. This smallness assumption was removed by Ambrose who showed that the vortex sheet problem with surface tension is well-posed, at least in the infinite-depth setting \cite{A1}. This local well-posedness result also holds in dimension $d=3$ \cite{AMas2}.

In spite of the classical vortex sheet problem assuming that the upper and lower fluids are density-matched, this assumption is not necessary and vortex sheet formulations have been widely used to study water waves and other phenomena in fluid dynamics. This approach (i.e., using the vortex sheet formulation to model phenomena in fluid dynamics) belongs to the broader class of tools known as vortex methods. The seminal work on vortex sheet formulations is that of Baker-Meiron-Orszag, which considered two-dimensional water waves \cite{BMO}. Vortex sheet formulations have also been applied to study other phenomena in fluid dynamics (e.g., gas bubbles in liquids \cite{Hua1}).

A particularly beneficial framework for vortex sheet formulations was developed by Hou-Lowengrub-Shelley (HLS) in their beautiful paper \cite{HLS1} (see also \cite{HLS2}). This framework was developed from a numerical perspective to create a non-stiff algorithm for modeling $2d$ interfacial flow under the influence of surface tension. The HLS framework rests on two key ideas. The first, influenced by earlier work of Mullins on ``curve shortening'' in the context of grain boundaries \cite{Mul1}, is to select a special frame of reference by choosing particular geometric coordinates (as opposed to Cartesian coordinates). The second is to pick a favorable, renormalized arclength parameterization of the interface. A third important component of the HLS framework that is primarily relevant for numerical work is the use of a small-scale decomposition (SSD). That is, terms which are unstable at small spatial scales are identified so that they can be computed implicitly, whereas the remaining terms are computed explicitly. It is worth noting that the terms showing up in the SSD also require care when studying the equations analytically, however there are additional terms that require similar care that do not appear in the HLS SSD (see \cite{A1} for further discussion). We shall discuss the HLS framework further in the sequel, but one particular benefit, following from the first key idea, is that one obtains a highly simplified expression for the curvature of the interface $\kappa(\z)$, which is relevant when considering surface tension due to the Laplace-Young condition at the interface.

The HLS framework is powerful and, in addition to classical vortex sheets, has been used to study water waves \cite{AMas1, CHS1, CCG1, Dull1}, Darcy flows \cite{A4}, hydroelastic waves \cite{ASieg1} and flame fronts \cite{AkeA1}. Moreover, although the HLS framework is necessarily two-dimensional, the main insights have been extended to study $3d$ flows. In the case of three-dimensional flows, isothermal coordinates take the place of the arclength parameterization. Examples of numerical and analytical work using this framework can be found in \cite{AMas2, CCG3, HZ1}.

The well-posedness theory of the vortex sheet formulation of the water waves problem has been developed by several authors. All of the results of which we are aware deal with the infinite-depth setting. Ambrose proved in \cite{A1} that the vortex sheet formulation of the two-dimensional gravity-capillary water waves problem is well-posed and this model was shown to be well-posed in the zero surface tension limit by Ambrose-Masmoudi in \cite{AMas1}. Ambrose-Masmoudi prove analogous results in three dimensions in \cite{AMas2, AMas3}. Christianson-Hur-Staffilani utilize a vortex sheet formulation to prove Strichartz estimates (with loss) and local smoothing for the two-dimensional (infinite-depth) water waves system with surface tension \cite{CHS1}.

The above-cited references represent only a small fraction of the literature on vortex sheets and vortex methods. For example, we have not addressed the celebrated work of Delort \cite{Del1} on the global existence of a weak solution to the Euler equations with vortex sheet Cauchy data under the assumption that the vortex sheet strength does not change sign, which built off of important work on reduced defect measures and concentration-cancellation and has been extended in numerous ways. The survey article \cite{BL} by Bardos-Lannes is also well worth reading and covers the Kelvin-Helmholtz problem, the Rayleigh-Taylor problem and the vortex sheet formulation of the water waves problem. Nevertheless, we believe that the results we have discussed should provide sufficient background to place the results of this work into the proper context.

\subsection{Existing Results on Damped Water Waves}

When we refer to damping water waves, we are referring to the application of a sponge layer or numerical beach; that is, an artificial, dissipative term supported near the boundary that removes energy from the system. However, in the literature, there are other systems which are referred to as models for damped water waves. We will briefly give an overview of some of this material and ultimately discuss how it differs from the damping we consider here.

We first mention the damped Euler equations:
\begin{equation}
\label{eqn:DampedEulerEqns}
\begin{dcases}
\p_t\vel + (\vel\cdot\nab)\vel + a\vel = -\nab\frac{p}{\rho_0} + \mathbf{f}\\
\diver\vel = 0
\end{dcases},
\end{equation}
where $a > 0$ is the damping coefficient and $\mathbf{f}$ denotes any body forces acting on the flow. Then, one approach to studying damped water waves is to study the free boundary problem corresponding to the damped Euler equations \eqref{eqn:DampedEulerEqns}. We note that the kinematic and dynamic boundary conditions will be the same for the damped Euler equations. The gravity-capillary water waves problem for the damped Euler equations is globally well-posed and solutions decay to equilibrium exponentially in time \cite{Lian1}. Thus, we see that the term damped is justified.

It is known that viscosity is physically dissipative. Of course, this can be seen mathematically by comparing the Euler and Navier-Stokes equations. It is therefore reasonable to think that incorporating viscosity into the water waves system should have a damping effect. There is, however, an obstacle to adding viscosity to the water waves system: viscosity is, in general, not compatible with potential flow and the existence of a velocity potential is critical for many formulations of the water waves problem. So, if one wants to retain the existence of a velocity potential, then any viscosity incorporated into the problem will be, in some sense, artificial. A well-studied model for viscous potential flow is the Dias-Dyachenko-Zakharov (DDZ) model, which, for $2d$ gravity-capillary water waves over a flat bottom is
\begin{equation}
\label{eqn:DDZ}
\begin{dcases}
\D\vphi = 0 &\text{in } \W_t\\
\p_t\eta = \p_y\vphi + 2\nu\p_x^2\vphi - \p_x\eta\p_x\vphi &\text{on } \FS_t\\
\p_t\vphi = -\frac{1}{2}\abs{\nab\vphi}^2 - 2\nu\p_y^2\vphi - g\eta + \tau H(\eta) &\text{on } \FS_t\\
\p_y\vphi = 0 &\text{on } y = -h
\end{dcases},
\end{equation}
where $\eta$ is a function describing the location of the free surface, $\set{y=-h}$ is the flat bottom and $H(\eta)$ is the mean curvature of $\eta$:
\begin{equation}
\label{eqn:MeanCurvature}
H(\eta) = \p_x\left( \frac{\p_x\eta}{\sqrt{1+(\p_x\eta)^2}} \right).
\end{equation}
The DDZ model was first formulated for gravity waves over infinite depth in \cite{DDZ1}. Moreover, gravity-capillary water waves problem \eqref{eqn:DDZ} is known to be globally well-posed with solutions decaying to equilibrium exponentially in time \cite{NgNi1}. So, again, we see that it is appropriate to refer to the DDZ model as a model for damped water waves.

Recall that we are primarily interested in damping that can be applied to the numerical study of water waves (e.g., damping water waves in a numerical wave tank). As such, we would like the waves to propagate freely in as much of the domain as possible and only be attenuated near the artificial boundary in order to avoid undesirable reflections. The above models, at least as written, are not well-adapted to this task for they damp the fluid on the entirety of the domain. This could, at least in principle, be fixed by localizing the effect of the damping, however one would then need to investigate the performance of the resulting damper. Indeed this gives rise to a number of fascinating questions for future research. For example, if we take $\nu = \nu(x)$ in \eqref{eqn:DDZ} and localize by requiring $\nu$ be supported near the boundary, will it stabilize the water waves system? What will be the rate of decay? Of course, we could ask similar questions about the damped Euler equations \eqref{eqn:DampedEulerEqns}. As interesting as such questions may be, at least to the author, they are beyond the scope of this work.

The numerical literature on damped water waves is quite substantial. The interested reader may begin by consulting \cite{BMO2, Bonn1, ClaEtAl, Clem1, Ducr1, IsrOrs1, JenEA1} as well as the references therein. While numerical experiments are important in their own right, obtaining an analytical understanding of damped water waves is also important, however the literature here is much more sparse. We reiterate that, by damped, we mean an artificial dissipative term whose effect is localized.

An important exception would be Alazard's wonderful papers on the stabilization of the water waves system \cite{Ala1, Ala2}. In \cite{Ala1}, the popular damper 
\begin{equation}
\label{eqn:HamiltonianDamper}
\pext{1} \coloneqq \lambda\chi_1\p_t\eta
\end{equation}
is considered, where $\lambda$ is a positive constant and $\chi_1$ is a cut-off function. Notice that we could rewrite $\pext{1} = \lambda\chi_1G(\eta)\psi$, where $G(\eta)$ is the normalized Dirichlet-Neumann map and $\psi$ is the trace of the velocity potential along the free surface. The damper \eqref{eqn:HamiltonianDamper} is a natural choice from the Hamiltonian perspective. If $\FS_t$ is the graph of a function $\eta$, then the water waves system can be written as a Hamiltonian system with Hamiltonian energy
\begin{equation}
\label{eqn:HamiltonianEnergy}
\mathscr{H} = \frac{g}{2}\int_0^{2\pi} \eta^2 \ dx + \tau \int_0^{2\pi} \sqrt{1 + \eta_x^2} - 1 \ dx + \frac{1}{2}\int_0^{2\pi}\int_{-h}^{\eta(x,t)} \abs{\nab\vphi}^2 \ dydx,
\end{equation}
where $\set{y=-h}$ is the (flat) bottom of the fluid domain. Then, one has the Hamiltonian equations
\begin{equation}
\label{eqn:HamiltonianEqns}
\frac{\p\eta}{\p t} = \frac{\de\mathscr{H}}{\de\psi}, \ \frac{\p\psi}{\p t} = -\frac{\de\mathscr{H}}{\de\eta} - \pext{1}
\end{equation}
and, using \eqref{eqn:HamiltonianEqns}, one can deduce that
\begin{equation}
\label{eqn:HamiltonianDamperDissipation}
\frac{d\mathscr{H}}{dt} = \Int \frac{\de\mathscr{H}}{\de\eta}\p_t\eta + \frac{\de\mathscr{H}}{\de\psi}\p_t\psi \ dx = -\Int \p_t\eta\pext{1} \ dx = -\lambda\Int \chi_1(\p_t\eta)^2 \ dx \leq 0.
\end{equation}
Thus, it is easily seen that $\pext{1}$ induces dissipation of the energy. The real achievement of \cite{Ala1} is to show that $\pext{1}$ stabilizes the water waves system with the rate of convergence being exponential in time.

An analogous result is obtained in \cite{Ala2} for the $2d$ gravity water waves system. In the gravity case, the pneumatic damper is taken to satisfy
\begin{equation}
\label{eqn:AlazardDamper2}
\pext{2}(x,t) = \p_x^{-1}\left(\chi_2(x)\int_{-h}^{\eta(x,t)} \vphi_x(x,y,t) \ dy\right).
\end{equation}
The reason that the Hamiltonian damper \eqref{eqn:HamiltonianDamper} is not considered is due to difficulties in showing that the Cauchy problem is well-posed. A similar, though slightly more involved, argument shows that \eqref{eqn:AlazardDamper2} causes the Hamiltonian energy to decay. The main result of \cite{Ala2} is that $\pext{2}$, which satisfies \eqref{eqn:AlazardDamper2}, stabilizes the water waves system with the energy decaying to zero exponentially in time.

The question of stabilizability of the water waves equations belongs to the broader field of control theory for water waves. Within control theory, the problems of stabilizability, controllability and observability are closely related. These questions are likewise important for the numerical simulation of water waves. For example, the question of controllability relates to the generation of waves via a wave maker.

The first results on the controllability of the full (nonlinear) water waves system were obtained in the masterful work \cite{ABH1} by Alazard, Baldi and Han-Kwan, which considered control via an external pressure (i.e., a pneumatic wave maker). The authors prove that the periodic $2d$ gravity-capillary water waves system is locally exactly controllable in arbitrarily short time subject to a smallness constraint. The smallness assumptions of \cite{ABH1} are rather restrictive, but the stabilization result of \cite{Ala1}, which imposed a milder smallness assumption, can be combined with the small-data control result of \cite{ABH1} to yield a less restrictive control result due to a strategy of Dehman-Lebeau-Zuazua \cite{DLZ1}, which exploits the time-reversibility of the water waves system. The controllability result of \cite{ABH1} was extended to higher dimensions in \cite{Zhu1} subject to the requirement that the control domain $\w\subset\T^d$ satisfies the geometric control condition (GCC) of Rauch-Taylor (see \cite{RT1} or \cite{BLR1}). The GCC is a natural requirement for control problems (and stabilization problems when considering non-dissipative equations). Furthermore, we note that the GCC was implicit in the result of \cite{ABH1} as any $\w \subset \T$ satisfies the GCC. For observability, Alazard proves a result on the boundary observability of the gravity water waves system in \cite{Ala3} (both $2d$ and $3d$ waves are considered); namely, it is shown that, considering a fluid in a rectangular tank bounded by a flat bottom, vertical walls and a free surface, that one can estimate the energy of the system via observations at the boundary (i.e., where the free surface meets the vertical walls).

\subsection{Plan of the Paper}

We consider a vortex sheet model for water waves with a (constant) background current over obstacles and topography proposed by Ambrose, et.\@ al.\@ in \cite{AMEA}. For simplicity of presentation, we limit ourselves to the case of a single obstacle, however our techniques apply to the case of any finite number of obstacles. The velocity is given by the gradient of a scalar potential $\vphi$, which is represented via layer potentials on the different components of the boundary. The variables which we evolve are $\tta$, the tangent angle formed by the interface with the horizontal; $\y$, the vortex sheet strength; $\w$, the density of the layer potential on the bottom and $\be$, the density of the layer potential on the obstacle. We note that $\y \coloneqq \mu_\al$, where $\mu$ is the density of the layer potential on the free surface. 

The system of equations which we consider is nonlocal and, in particular, is of the form
\begin{equation}
\label{eqn:WaterWavesSystemForm}
\begin{cases} (\id + \K[\Tta])\Tta_t = \F(\Tta)\\ \Tta(t=0) = \Tta_0 \end{cases},
\end{equation}
where $\Tta \coloneqq (\tta,\y,\w,\be)^t$ and $\K[\Tta]$ is a compact operator. We introduce the parameter $B$ to denote the size of the initial data:
\begin{equation}
\label{eqn:BDef}
B \coloneqq \norm{\Tta_0}_X, 
\end{equation}
where $X$ is the energy space. We will obtain our main lifespan results in the context of small data and, in this setting, we take
\begin{equation}
\label{eqn:SmallDataDef}
B = \ee \ll 1.
\end{equation}

Our first main objective will be to show that the model proposed in \cite{AMEA} is well-posed and that solutions persist on a timescale of order $\bigo(\log\frac{1}{\ee})$ (resp\@. $\bigo(1)$) in the presence of zero (resp\@. non-zero) background current. Our approach will be to first consider the model problem
\begin{equation}
\label{eqn:ModelProblem}
\begin{cases} \p_t\Tta = \F(\Tta)\\ \Tta(t=0) = \Tta_0\end{cases},
\end{equation}
beginning by proving the desired results about this model problem via energy estimates. Then, we will deduce mapping properties of $(\id + \K)^{-1}$ that imply that the results proved for the model problem \eqref{eqn:ModelProblem} are also true of the water waves system \eqref{eqn:WaterWavesSystemForm}.

Our next primary objective will be to modify the system \eqref{eqn:WaterWavesSystemForm} to incorporate the Clamond damper and show that the same results hold for the damped system. We do so by following the same approach as for the non-damped system (i.e., first consider the model problem for damped water waves and then use mapping properties of $(\id + \K)^{-1}$ to obtain the desired result). As noted above, we primarily utilize energy estimates and, in particular, we largely follow the approach of \cite{A1}. The existence time obtained here is certainly not sharp, particularly being less than the $\bigo(\frac{1}{\ee})$ lifespan suggested by the nonlinearity (this follows from the classical local well-posedness theory for quasilinear hyperbolic equations; e.g., see \cite{Kato1, Kato2, Maj2}). However, obtaining the sharper existence time requires a more detailed study of the system (e.g., via paradifferential analysis). As such, we have decided to leave this to a follow-up paper and here simply focus on results obtainable by straightforward energy methods.

The plan of this paper is as follows. In Section 2, we give an overview of the main results. We then proceed to give a brief overview of the model which we utilize in Section 3. Next, in Section 4, we determine the appropriate right-hand side $\F(\Tta)$ for our model problem \eqref{eqn:ModelProblem}. Moving on, in Section 5, we prove the first main result which is a uniform energy estimate for our model problem. In Sections 6 and 7, we complete the proof of the local-in-time well-posedness of the (undamped) model system. Section 8 begins by discussing how to extend results on the model problem to the water waves system and then goes on to study the lifespan of solutions to the system. Section 9 considers the damped problem and here we show that the results of all previous sections apply to the damped problem. We also include two appendices. The first appendix is a collection of results which we utilize frequently. The second appendix considers the solvability of the integral equations arising in the system, which gives an alternative approach to the one given in \cite{AMEA}. One of the reasons we include this proof is that it can be more readily extended to $3d$ than the proof given in \cite{AMEA}.

\section{Main Results}

Here we will state the main results of the paper. As outlined above, our first main result is to show that that this system is well-posed locally in time and to obtain a lower bound on the lifespan of solutions. Next, we consider a damped version of the system and show that all of the results obtained for the non-damped system apply to the damped system.

To simplify notation, we shall omit the domain from spaces of functions or distributions when the domain is the torus $\T$. That is, we write $H^r$, $L^p$, $\Dist$ and so on, instead of $H^r(\T)$, $L^p(\T)$, $\Dist(\T)$, etc. Letting $\V \coloneqq (V_0,0)$ denote the background current, our first main result is then the following:

\begin{thm}
\label{MainTheorem1}
Let $s$ be sufficiently large. The system \eqref{eqn:WaterWavesSystemForm} is locally well-posed (in the sense of Hadamard). Namely, there exists a unique solution $\Tta \in C([0,T(B,\Vzeronorm)];H^s \times H^{\sHalf} \times H^1 \times H^1)$ to the system \eqref{eqn:WaterWavesSystemForm} and the flow map $\Tta_0 \mapsto \Tta$ is continuous. In the case of small Cauchy data and zero background current (i.e., $V_0 = 0$), we have
\begin{equation}
\label{eqn:LogLifespanForm}
T(\ee) \gtrsim \log\frac{1}{\ee}.
\end{equation}
On the other hand, for large Cauchy data, we have
\begin{equation}
\label{eqn:ConstantLifespanForm}
T(B,\Vzeronorm) \gtrsim \begin{cases} B^{1-N} & V_0 = 0\\ \min\left( (1+\Vzeronorm)^{-2}, B^{1-N} \right) & V_0 \neq 0\end{cases},
\end{equation}
where $N$ is a parameter given in equation \eqref{eqn:PolDef}.
\end{thm}

\begin{rmk}
We again note that the solution is not guaranteed to remain of size $\bigo(\ee)$ on the given lifespans. Rather, all that is assured is that the energy remains bounded, and thus the solutions persist, on the stated timescales. The $\bigo(\log\frac{1}{\ee})$ lifespan when $V_0=0$ is certainly not sharp. In fact, the quadratic nonlinearity exhibited by the system \eqref{eqn:WaterWavesSystemForm} suggests an $\bigo(\frac{1}{\ee})$ lifespan. However, actually proving that solutions exist on an $\bigo(\frac{1}{\ee})$ timescale is not a trivial matter and will require more delicate analysis \cite{AlvLan1, BMG1}. On the other hand, proving that solutions persist on an $\bigo(\log\frac{1}{\ee})$ timescale can be done using only energy estimates and a Gr\"{o}nwall argument. As such, in this paper, which is largely based on energy methods, we simply prove the $\bigo(\log\frac{1}{\ee})$ lifespan. We are presently working on a follow-up paper in which we prove the $\bigo(\frac{1}{\ee})$ lifespan.
\end{rmk}

\begin{rmk}
The existence time of $\bigo((1+\Vzeronorm)^{-2}\wedge B^{1-N})$ when $V_0\neq0$ may not be sharp, however substantial improvements are not possible. In fact, when $V_0 \neq 0$, numerical simulations have shown splash singularities to occur in $\bigo(1)$ time, even beginning from still water \cite{AMEA}.
\end{rmk}

We next consider a damped version of the system. As noted above, we implement a modified sponge layer damper, which we call Clamond damping, first introduced in \cite{ClaEtAl}. Recall that Clamond damping utilizes a pneumatic damper with the external pressure given by \eqref{eqn:ClamondDamper} (i.e., $\pex \coloneqq \p_x^{-1}(\chi_\w\p_x\vphi)$). Though we use the same notation $\w$ for the damping region and the density of the single layer potential on the bottom, this will cause no confusion as context will always make clear what $\w$ represents.

We derive evolution equations which account for the Clamond damping and we denote the new right-hand side by $\F_D$. We then arrive at the damped water waves system:
\begin{equation}
\label{eqn:DampedWaterWaveSystemForm}
\begin{cases} (\id + \K[\Tta])\Tta_t = \F_D(\Tta)\\ \Tta(t=0) = \Tta_0 \end{cases}.
\end{equation}
Our second main result is as follows:

\begin{thm}
\label{MainTheorem2}
All of the results of Theorem \ref{MainTheorem1} apply to the damped system. In particular, take $s$ to be sufficiently large. Then, \eqref{eqn:DampedWaterWaveSystemForm} is locally-in-time well-posed with the flow map $\Tta_0 \mapsto \Tta$ being continuous and the solution $\Tta$ belonging to $C([0,T(B,\Vzeronorm)]; H^s \times H^{\sHalf} \times H^1 \times H^1)$. For small Cauchy data and zero background flow, the lifespan $T(\ee)$ satisfies \eqref{eqn:LogLifespanForm} and, in the case of large Cauchy data, we again have \eqref{eqn:ConstantLifespanForm}.
\end{thm}

\begin{rmk}
In \cite{AMEA}, Ambrose, et.\@ al.\@ actually present two formulations of the water waves problem. Namely, in addition to the vortex sheet formulation we consider here, they propose a dual formulation via Cauchy integrals. The energy methods employed here would yield results analogous to those of Theorem \ref{MainTheorem1} for the Cauchy integral formulation. Further, Clamond damping can be implemented in the Cauchy integral formulation and the results of Theorem \ref{MainTheorem2} can similarly be obtained for the Cauchy integral formulation via the energy arguments utilized here.
\end{rmk}

\section{A Brief Overview of the Model}

Our objective here is to give a brief overview of the model which we utilize. We will discuss the domain as well the relevant variables and parameters with which we work. Finally, we will write down the evolution equations which govern the system. For full details on the model, the reader should consult \cite{AMEA}.

\subsection{The Domain}

At time $t$, the fluid is contained in a domain $\W_t \subset \T \times \R$ of finite vertical extent. The fluid domain is bounded above by a free surface $\FS_t$ and below by a fixed, solid boundary $\B$. We assume $\W_t$ is multiconnected and $\p\W_t\setminus(\FS_t \cup \B)$ is composed of smooth Jordan curves. We describe the location of the free surface via a parameterized curve, $\FS_t: (\xi(\al,t),\eta(\al,t))$, where $t$ denotes time and $\al$ is the parameter along $\FS_t$. Here, $\xi(\al) - \al$ and $\eta$ are both periodic with period $2\pi$. The bottom is fixed (i.e., time-independent) and also described by a parameterized curve $\B: (\xi_1(\al),\eta_1(\al))$ with the same periodicity. Additionally, the multiconnectedness of $\W_t$ corresponds to one or more obstacles in the flow. For simplicity of notation and presentation we utilize a single obstacle $\Obs$ (i.e., $\complement\W_t = \Obs \cup \mathcal{U}_t$, where $\mathcal{U}_t$ is unbounded). However, we note that the extension to an arbitrary, finite number of obstacles is immediate and all of our results apply to this case. We denote $\Cyl \coloneqq \p\Obs = \p\W_t\setminus(\FS_t \cup \B)$. We assume that the obstacle is fixed and that its boundary is given by a closed parameterized curve $\Cyl: (\xi_2(\al),\eta_2(\al))$ with $\xi_2$ and $\eta_2$ being $2\pi$-periodic.

It will frequently be beneficial to utilize a complexified description of the domain and to this end define
\begin{equation}
\label{eqn:z_jDef}
\z \coloneqq \xi+i\eta \text{ and } \z_j \coloneqq \xi_j + i\eta_j.
\end{equation}
Regarding orientation, we parameterize the boundary of the fluid domain so that the normal on $\FS_t$ points into the vacuum region, the normal on $\B$ points into the fluid region and the normal on $\Cyl$ points into the fluid region. We denote the length of one period of the free surface by $L=L(t)$, the length of one period of $\B$ by $L_1$ and the length of $\Cyl$ by $L_2$.

For technical reasons, we shall want the interface to be free of self-intersections. In order to ensure that this is so, we impose the chord-arc condition on $\z$:
\begin{equation}
\label{eqn:ChordArc}
\exists \ca > 0 : \ \abs{\frac{\z(\al) - \z(\alp)}{\al-\alp}} > \ca \qquad (\forall \al \neq \alp).
\end{equation}
This condition will rule out self-intersections (e.g., splash and splat singularities) as well as cusps.

We shall also assume that the depth of the water is bounded away from zero, as is the distance from the free surface to the boundary of the obstacle. Namely, there exist positive constants $\h$ and $\tilde{\h}$ so that
\begin{align}
\eta - \eta_1 &\geq \h, \label{eqn:MinimalWaterDepth1}\\
\eta - \eta_2 &\geq \tilde{\h}. \label{eqn:MinimalWaterDepth2}
\end{align}
These assumptions mean that neither the bottom nor the obstacle go dry and are critical for our analysis. The question of removing these assumptions is quite fascinating and much work is yet to be done. Some progress has been made in considering the water waves system in a simply connected domain in the absence of assumption \eqref{eqn:MinimalWaterDepth1}. For more on this fascinating problem, the interested reader can consult \cite{deP1, MingWang1, MingWang2, MingWang3}. Asymptotic models for water waves are studied in this context in, for example, \cite{LanMet1}.


Finally, we introduce the notation $\z_d$, which we define by
\begin{equation}
\label{eqn:zdDef}
\z_d(\al,t) \coloneqq \z(\al,t) - \z(0,t).
\end{equation}
The value of $\z(0,t)$ is, in general, unimportant. It is worth noting that $\p_\al\z_d = \p_\al\z$.

\subsection{The Dynamics of the Free Surface}

We will now briefly discuss how the evolution of the free surface is described in this model. At each point on $\FS_t$, there is a unit tangent vector $\ut = \abs{(\xi_\al,\eta_\al)}^{-1}(\xi_\al,\eta_\al)$ and a unit normal vector $\un = \abs{(\xi_\al,\eta_\al)}^{-1}(-\eta_\al,\xi_\al)$, where the subscript $\al$ denotes differentiation with respect to the parameter $\al$. We let $U$ denote the normal velocity and $V$ the tangential velocity:
\begin{equation}
\label{eqn:GenEvoEqn}
\p_t (\xi,\eta) = U\un + V\ut.
\end{equation}
A key observation underlying the HLS framework is that the shape of the free boundary $\FS_t$ is solely determined by the normal velocity $U$, while changes in the tangential velocity $V$ serve only to reparameterize the interface \cite{HLS1}. The tangential velocity $V$ will be chosen so as to enforce a renormalized arclength parameterization of $\FS_t$.

The HLS framework utilizes a geometric frame of reference to describe the location of the free surface, as opposed to the usual Cartesian coordinates $(\xi,\eta)$ \cite{HLS1}. The first of the geometric coordinates is $\tta = \tta(\al,t)$, which denotes the tangent angle formed by $\FS_t$ with the horizontal:
\begin{equation}
\label{eqn:ThetaDef}
\tta \coloneqq \arctan\frac{\eta_\al}{\xi_\al}.
\end{equation}
Using this new variable, we can write $\ut = (\cos\tta,\sin\tta)$ and $\un = (-\sin\tta,\cos\tta)$. In addition, we have
\begin{equation}
\label{eqn:xForm}
\xi(\al) = \al + \p_\al^{-1}(s_\al\cos\tta(\al)),
\end{equation}
where, in this case, $\p_\al^{-1}$ denotes the mean-zero antiderivative (for more details, see section 2.2 of \cite{AMEA}).

The other geometric coordinate is the arclength element $s_\al = s_\al(\al,t)$ given by $s_\al \coloneqq \sqrt{\xi_\al^2 + \eta_\al^2}$. It is straightforward to see that
\begin{equation}
\label{eqn:sEvo}
\p_ts_\al = V_\al - \tta_\al U.
\end{equation}
We further note that $L$ is given by $L(t) = \int_0^{2\pi} s_\al(\al,t) \ d\al$. We may again differentiate with respect to time and use equation \eqref{eqn:GenEvoEqn} to infer the evolution equation for $L$:
\begin{equation}
\label{eqn:LEvo}
\p_t L = -\int_0^{2\pi} \tta_\al U \ d\al.
\end{equation}
In fact, one can either take $s_\al$ or $L$ to be the other independent variable describing $\FS_t$.

The tangential velocity $V$ is selected to enforce that $s_\al$ be independent of the spatial variable, which yields a renormalized arclength parameterization of $\FS_t$. Considering the equations for $\p_t s_\al$ and $\p_t L$ leads to the choice
\begin{equation}
\label{eqn:VDef}
V \coloneqq \p_\al^{-1}\left( \tta_\al U - \frac{1}{2\pi}\int_0^{2\pi} \tta_\al U \ d\al \right).
\end{equation}
Implicit in \eqref{eqn:VDef} is a constant of integration, which we are free to choose. Reasonable choices include taking the constant of integration so as to force (i) $V$ to have mean zero, (ii) $V(0,t) = 0$ or (iii) $\xi(0,t) = 0$. It is straightforward to check that such a choice of $V$ leads to $L = 2\pi s_\al$ for all time (also see \cite{AMEA}).

Our next objective is to give a definition of the normal velocity $U$ along the free surface. We recall that the fluid velocity satisfies the (irrotational) free-surface Euler equations \eqref{eqn:Free-SurfaceEulerEqns}. In particular, the assumption that $\curl\vel = 0$ (irrotationality) implies that the velocity field is given by the gradient of a scalar potential $\vphi$. With this in mind, we shall write $\vel = \nab\vphi$ with $\vphi = \varphi_0 + \varphi_1 + \varphi_2 + \chi(a_0\nab\vphic + \V)$, noting that each of the $\varphi_j$'s corresponds to a different part of the boundary of the fluid region -- the interface $\FS_t$, the bottom $\B$ and the boundary of the obstacle $\Cyl$. The constant $a_0$ is a circulation parameter and $\vphic$ is given by
\begin{equation}
\label{eqn:Phi_cyl}
\vphic(z) \coloneqq \Rea\bigg\{\frac{1}{2}z - i\log\sin\frac{1}{2}(z - z_c)\bigg\},
\end{equation}
where $z_c \in \Obs$. We note that it is only necessary to introduce $\vphic$ in the case of a nonzero background flow, which is why we have introduced the coefficient
\begin{equation}
\label{eqn:chiDef}
\chi \coloneqq \begin{cases} 1 & V_0 \neq 0\\ 0 & V_0 = 0\end{cases}.
\end{equation}
As previously noted, $U$ must be determined by the physics and we have
\begin{equation}
\label{eqn:UPhys}
U \coloneqq \p_\un\vphi.
\end{equation}
We take the $\varphi_j$'s to be given by layer potentials (a double layer potential on the free surface and single layer potentials on the bottom as well as on the boundary of the obstacle).

The double layer potential corresponding to the free surface is given by
\begin{equation}
\label{eqn:DoubleLayerPotentialFreeSurface}
\vphi_0(x,y) = \Rea\bigg\{ \frac{1}{4\pi i}\int_0^{2\pi} \mu(\alp)\z_\al(\alp)\cot\frac{1}{2}((x+iy)-\z(\alp)) \ d\alp \bigg\},
\end{equation}
where $(x,y)$ is in the fluid region. Of course, the gradient of \eqref{eqn:DoubleLayerPotentialFreeSurface} will be singular on $\FS_t$, however we can take the limit as we approach the interface using the Plemelj formulae. This process yields
\begin{equation}
\label{eqn:GradPhi0Surface}
\lim_{(x,y) \to (\xi(\al),\eta(\al))} (\p_x - i\p_y)\vphi_0(x,y) = \frac{1}{4\pi i} \pv\int_0^{2\pi} \y(\alp)\cot\frac{1}{2}(\z(\al)-\z(\alp)) \ d\alp + \frac{\y(\al)\z_\al^*(\al)}{2s_\al^2},
\end{equation}
where the $\mathrm{pv}$ denotes a principal value integral, $\y \coloneqq \mu_\al$ is the vortex sheet strength and $(\cdot)^*$ denotes complex conjugation. Note that the integral in \eqref{eqn:GradPhi0Surface} is the (complex conjugate of the) complexified Birkhoff-Rott integral. We denote the real Birkhoff-Rott integral as $\BR = (BR_1,BR_2)$ and so
\begin{equation}
\label{eqn:WDef}
\Comp(\BR)^*(\al) = \frac{1}{4\pi i}\pv\int_0^{2\pi} \y(\al^\prime)\cot\frac{1}{2}(\z(\al) - \z(\al^\prime)) \ d\al^\prime,
\end{equation}
where $\Comp: (a,b) \mapsto a+ib$. We can rewrite \eqref{eqn:GradPhi0Surface} as
\begin{equation}
\label{eqn:GradPhi0Surface2}
\lim_{(x,y) \to (\xi(\al),\eta(\al))} \nab\vphi_0 = \BR + \frac{\y}{2s_\al}\ut.
\end{equation}
A key aspect of our methods that restricts us to considering the $2d$ case involves the simplification of the Birkhoff-Rott integral in \eqref{eqn:WDef}, namely summing over periodic images to obtain a complex cotangent kernel. It is also worthwhile to reinforce that the integral defining $\BR$ is a singular integral as this fact shall be important in the analysis to come.

We define $\Y \coloneqq \nab\vphi_1(\z)$ and a simple computation yields
\begin{equation}
\label{eqn:YDef}
\Comp(\Y)^*(\al) = \frac{1}{4\pi} \int_0^{2\pi} \w(\al^\prime)s_{1,\al}(\al^\prime)\cot\frac{1}{2}(\z(\al) - \z_1(\al^\prime)) \ d\al^\prime,
\end{equation}
where $s_{1,\al}$ is the arclength parameter on the bottom. Similary, for $\vphi_2$, we take $\Z \coloneqq \nab\vphi_2(\z)$ and have
\begin{equation}
\label{eqn:ZDef}
\Comp(\Z)^*(\al) = \frac{1}{4\pi} \int_0^{2\pi} \be(\al^\prime)s_{2,\al}(\alp) \cot\frac{1}{2}(\z(\al) - \z_2(\al^\prime)) \ d\al^\prime,
\end{equation}
where $s_{2,\al}$ denotes the arclength parameter on $\Cyl$. Notice that the integrals defining $\Y$ and $\Z$ are not singular.

It shall be convenient to introduce the notation $\Wt \coloneqq \BR + \Y + \Z + \chi\left(a_0\nab\vphic(\z) + \V \right)$. With this notation in place, utilizing \eqref{eqn:UPhys}, we can write $U$ along the interface as
\begin{equation}
\label{eqn:UDef}
U(\al) = \Wt(\al) \cdot \un(\al).
\end{equation}
We shall write $U = U_0 + U_1 + U_2 + \chi U_3$, where
\begin{equation*}
U_0 \coloneqq \BR \cdot \un, \ U_1 \coloneqq \Y \cdot \un, \ U_2 \coloneqq \Z \cdot \un, \ U_3 \coloneqq a_0\nab\vphic(\z) \cdot \un + \V \cdot \un.
\end{equation*}

Given the singular nature of $\BR$, it will be useful to decompose it, as well as $\BR_\al$, into a singular term and a smooth remainder. To this end, we shall utilize the following decompositions from \cite{A1}:
\begin{align}
\Comp(\BR)^* &= \frac{1}{2i}\HT\left(\frac{\y}{\z_\al}\right) + K[\z]\y, \label{eqn:WToHK}\\
\BR_\al &= \frac{1}{2s_\al}\HT(\y_\al)\un - \frac{1}{2s_\al}\HT(\y\tta_\al)\ut + \m, \label{eqn:WDeriv}
\end{align}
where $K[\cdot]$ is a smoothing operator (see Lemma 3.5 in \cite{A1} or Lemma \ref{KopEst} below) given by
\begin{equation}
\label{eqn:KDef}
K[\z]f(\al) \coloneqq \frac{1}{4\pi i}\int_{b-\pi}^{b+\pi} f(\alp) \left[ \cot\frac{1}{2}(\z_d(\al) - \z_d(\alp)) - \frac{1}{\z_\al(\alp)}\cot\frac{1}{2}(\al - \alp) \right] \ d\alp,
\end{equation}
where $b$ can be any real number. On the other hand, $\m$ is given by
\begin{equation}
\label{eqn:mDef}
\Comp(\m)^* \coloneqq \frac{\z_\al}{2i}\left[\HT,\z_\al^{-2}\right]\left(\y_\al - \frac{\y \z_{\al\al}}{\z_\al}\right) + \z_\al K[\z]\left( \frac{\y_\al}{\z_\al} - \frac{\y \z_{\al\al}}{\z_\al^2} \right).
\end{equation}
So, the singular parts of $\BR$ and $\BR_\al$ are given by Hilbert transforms, while the smooth part of $\Comp(\BR)^*$ is given by $K[\z]\y$ and the smooth part of $\BR_\al$ is given by $\m$. We shall occasionally write $\m = \mathbf{B} + \mathbf{R}$, where $\mathbf{B}$ is the commutator term and $\mathbf{R}$ is the term involving the operator $K[\z]$.

The terms in $\m$ arise upon approximating $\z(\al) - \z(\alp)$ to first order via Taylor expansion and then rewriting the remainder. The reader may turn to \cite{A1} for all of the details. The singular nature of $\BR$, as opposed to the other terms in $\Wt$, means that it will, at times, be useful to distinguish it from the remaining terms. To do so, we write $\Wt = \BR + \Wo$.

A quantity which shall appear frequently in the work to come is $\p_\al(V-\Wt\cdot\ut)$, which results from the choice of $V \neq \Wt \cdot \ut$. Using the geometric identity $\ut_\al = \tta_\al\un$, we can formulate a convenient expression for $\p_\al(V-\Wt\cdot\ut)$:
\begin{equation}
\label{eqn:VWDeriv}
\p_\al(V-\Wt\cdot\ut) = \tta_\al U + s_{\al t} - \Wt_\al \cdot\ut - \Wt \cdot (\tta_\al\un) = s_{\al t} - \Wt_\al\cdot\ut.
\end{equation}
We can obtain a useful representation of $\BR_\al\cdot\ut$ via equation \eqref{eqn:WDeriv}. The remaining terms in $\Wt_\al$ are regular and so are quite a bit simpler to grasp. We can simply compute them as follows, integrating by parts to retain the cotangent kernel:
\begin{align}
\p_\al\Comp(\Y)^*(\al) &= \frac{1}{4\pi}\int_0^{2\pi} \p_{\alp}\left(\frac{\w(\alp)s_{1,\al}(\alp)\z_\al(\al)}{\z_{1,\al}(\alp)}\right)\cot\frac{1}{2}(\z(\al) - \z_1(\alp)) \ d\alp, \label{eqn:YDeriv}\\
\p_\al\Comp(\Z)^*(\al) &= \frac{1}{4\pi}\int_0^{2\pi} \p_{\alp}\left(\frac{\be(\alp)s_{2,\al}(\alp)\z_\al(\al)}{\z_{2,\al}(\alp)}\right)\cot\frac{1}{2}(\z(\al) - \z_2(\alp)) \ d\alp, \label{eqn:ZDeriv}\\
\p_\al(\nab\vphic(\z(\al))) &= \nab^2\vphic(\z(\al))\z_\al(\al), \label{eqn:GradPhiCylDeriv}
\end{align}
where $\nab^2\vphic$ denotes the Hessian of $\vphic$.

\subsection{Evolution Equations}

Following the approach in \cite{AMEA}, we take our variables to be $\tta$, $\y$, $\w$ and $\be$. Notice that we do not explicitly evolve $s_\al$ or $L$. This will cause no trouble as after solving for the given variables, we can obtain $U$ and then easily solve for $s_\al$ and/or $L$. Here we wish to write out the system of evolution equations for $\tta$, $\y$, $\w$ and $\be$. Derivations of the evolution equations can be found in \cite{AMEA}.

Utilizing the definition of $\tta$, we can easily see that
\begin{equation}
\label{eqn:ThetaEvo1}
\tta_t = \frac{U_\al + \tta_\al V}{s_\al}.
\end{equation}
Using \eqref{eqn:WDeriv}, we can rewrite \eqref{eqn:ThetaEvo1} as
\begin{equation}
\label{eqn:ThetaEvolutionEquation}
\tta_t = \frac{1}{2s_\al^2}\HT(\y_\al) + \frac{\tta_\al}{s_\al}\left(V - \Wt \cdot \ut \right) + \frac{1}{s_\al}\Wo_\al \cdot \un + \frac{\m \cdot \un}{s_\al}.
\end{equation}

Recall that $\y$ is the vortex sheet strength and related to the velocity potential at the free surface by $\y \coloneqq \mu_\al$, where $\mu$, on the other hand, is the density of the double layer potential at the free surface. Hence, via standard layer potential theory (e\@.g\@., see \cite{F}), we know that $\mu$ represents the jump in $\vphi_0$ across the interface. The derivation of the evolution equation for $\y$ is substantially more involved than that for $\tta$. Roughly, one begins from \eqref{eqn:GradPhi0Surface} and rearranges to obtain an expression for $\y$, which is then differentiated with respect to time. One rewrites the resulting expression using the Bernoulli equation and then uses the Laplace-Young condition on the pressure at the interface. This is where we see one of the great strengths of the HLS framework. Namely, we have a highly simplified expression for the curvature of the free surface: $\kappa(\z) = \frac{\tta_\al}{s_\al}$. This process yields the following evolution equation for $\y$:
\begin{equation}
\label{eqn:yEvolutionEqn0}
\y_t = \p_\al\left( \frac{2\tau}{s_\al}\tta_\al + \frac{1}{s_\al}(V - \Wt\cdot\ut)\y - \frac{\y^2}{4s_\al^2} - 2g\eta \right) - 2s_\al\Wt_t\cdot\ut + 2(V - \Wt\cdot\ut)(\Wt_\al\cdot\ut).
\end{equation}
We can rewrite equation \eqref{eqn:yEvolutionEqn0} by expanding the derivative and applying \eqref{eqn:VWDeriv}. We then have
\begin{align}
\label{eqn:yEvolutionEqn}
\y_t &= \frac{2\tau}{s_\al}\tta_{\al\al} + \frac{\y}{2s_\al^2}\HT(\y\tta_\al) + \frac{\y_\al}{s_\al}\left(V - \Wt \cdot \ut\right) + \frac{\y}{s_\al}\left( s_{\al t} - \Wo_\al \cdot \ut - \m \cdot \ut \right) \nonumber\\ 
&\hspace{0.5cm} - 2s_\al\Wt_t \cdot \ut - \frac{\y\y_\al}{2s_\al^2} - 2g\eta_\al + 2\left(V - \Wt \cdot \ut\right)\Wt_\al \cdot \ut.
\end{align}
Observe that the $\y_t$ equation is nonlocal; more particularly, it is an integro-differential equation due to the presence of $\Wt_t \cdot \ut$ which involves integral operators acting on $\y_t$, $\w_t$ and $\be_t$.

Finally, we turn our attention to the evolution equations for $\w$ and $\be$. Recall that $\w$ is the density of the layer potential on the bottom and $\be$ is the density of the layer potential on the obstacle. In order to write the evolution equations (and later equations) more compactly, we introduce some notation for the integral kernels. These integral kernels, as well as the evolution equations for $\w$ and $\be$, arise from enforcing the homogeneous Neumann boundary conditions on the solid boundaries. On the free surface, we have
\begin{align}
\label{eqn:SurfaceIntegralKernels}
k_\FS^1(\al,\al^\prime) &= \Rea\bigg\{\frac{1}{2s_{1,\al}(\al)}\z_{1,\al}(\al)\cot\frac{1}{2}(\z_1(\al) - \z(\al^\prime))\bigg\}, \nonumber\\
k_\FS^2(\al,\al^\prime) &= \Rea\bigg\{\frac{1}{2s_{2,\al}(\al)}\z_{2,\al}(\al)\cot\frac{1}{2}(\z_2(\al) - \z(\al^\prime))\bigg\}.
\end{align}
Notice that the integral kernels in \eqref{eqn:SurfaceIntegralKernels} are time-dependent. The kernels on the bottom are
\begin{align}
\label{eqn:BottomIntegralKernels}
k_\B^1(\al,\al^\prime) &= \Rea\bigg\{\frac{is_{1,\al}(\al^\prime)}{2s_{1,\al}(\al)}\z_{1,\al}(\al)\cot\frac{1}{2}(\z_1(\al) - \z_1(\al^\prime))\bigg\}, \nonumber\\
k_\B^2(\al,\al^\prime) &= \Rea\bigg\{\frac{i s_{1,\al}(\al^\prime)}{2s_{2,\al}(\al)}\z_{2,\al}(\al)\cot\frac{1}{2}(\z_2(\al) - \z_1(\al^\prime))\bigg\}.
\end{align}
Finally, the kernels on the boundary of the obstacle are given by
\begin{align}
\label{eqn:ObstacleIntegralKernels}
k_\Cyl^1(\al,\al^\prime) &= \Rea\bigg\{\frac{is_{2,\al}(\alp)}{2s_{1,\al}(\al)}\z_{1,\al}(\al)\cot\frac{1}{2}(\z_1(\al) - \z_2(\al^\prime))\bigg\}, \nonumber\\
k_\Cyl^2(\al,\al^\prime) &= \Rea\bigg\{\frac{is_{2,\al}(\alp)}{2s_{2,\al}(\al)}\z_{2,\al}(\al)\cot\frac{1}{2}(\z_2(\al) - \z_2(\al^\prime))\bigg\}. 
\end{align}
Notice that, at first appearance, it seems that the kernels $k_\B^1$ and $k_\Cyl^2$ are also singular. However, they are in fact not singular (see \cite{AMEA} for details). We also note that the kernels in \eqref{eqn:BottomIntegralKernels} and \eqref{eqn:ObstacleIntegralKernels} are independent of time.

Utilizing this notation, the evolution equations for $\w$ and $\be$ are given by
\begin{align}
\label{eqn:wEvolutionEqnKer}
\left(\frac{1}{2}\w_t(\al) + \frac{1}{2\pi}\int_0^{2\pi} \w_t(\alp)k_\B^1(\al,\alp) \ d\alp \right) &= -\frac{1}{2\pi} \int_0^{2\pi} \y(\alp) k_{\FS,t}^1(\al,\alp) \ d\alp - \frac{1}{2\pi} \int_0^{2\pi} \y_t(\alp)k_\FS^1(\al,\alp) \ d\alp \nonumber\\
&\hspace{0.5cm} - \frac{1}{2\pi} \int_0^{2\pi} \be_t(\alp)k_\Cyl^1(\al,\alp) \ d\alp
\end{align}
and
\begin{align}
\label{eqn:beEvolutionEqnKer}
\left(\frac{1}{2}\be_t(\al) + \frac{1}{2\pi}\int_0^{2\pi} \be_t(\alp)k_\Cyl^2(\al,\alp) \ d\alp \right) &= -\frac{1}{2\pi}\int_0^{2\pi} \y(\alp) k_{\FS,t}^2(\al,\alp) \ d\alp - \frac{1}{2\pi}\int_0^{2\pi} \y_t(\alp) k_\FS^2(\al,\alp) \ d\alp \nonumber\\
&\hspace{0.5cm} - \frac{1}{2\pi}\int_0^{2\pi} \w_t(\alp) k_\B^2(\al,\alp) \ d\alp.
\end{align}
The equations for $\w_t$ and $\be_t$ are integro-differential equations and so, like the evolution equation for $\y$, are nonlocal.

Combining \eqref{eqn:ThetaEvolutionEquation}, \eqref{eqn:yEvolutionEqn}, \eqref{eqn:wEvolutionEqnKer} and \eqref{eqn:beEvolutionEqnKer}, we have the full water waves system which we shall study:
\begin{equation}
\label{eqn:WaterWavesSystem}
\begin{dcases}
\tta_t = \frac{1}{2s_\al^2}\HT(\y_\al) + \frac{\tta_\al}{s_\al}\left(V - \Wt \cdot \ut \right) + \frac{1}{s_\al}\Wo_\al \cdot \un + \frac{\m \cdot \un}{s_\al} \\
\y_t = \frac{2\tau}{s_\al}\tta_{\al\al} + \frac{\y}{2s_\al^2}\HT(\y\tta_\al) + \frac{\y_\al}{s_\al}\left(V - \Wt \cdot \ut\right) + \frac{\y}{s_\al}\left( s_{\al t} - \Wo_\al \cdot \ut - \m \cdot \ut \right) \\ 
\hspace{0.75cm} - 2s_\al\Wt_t \cdot \ut - \frac{\y\y_\al}{2s_\al^2} - 2g\eta_\al + 2\left(V - \Wt \cdot \ut\right)\Wt_\al \cdot \ut \\
\w_t = -\frac{1}{\pi}\int_0^{2\pi} \w_t(\alp)k_\B^1(\cdot,\alp) \ d\alp - \frac{1}{\pi} \int_0^{2\pi} \y(\alp) k_{\FS,t}^1(\cdot,\alp) \ d\alp \\
\hspace{0.75cm} - \frac{1}{\pi} \int_0^{2\pi} \y_t(\alp)k_\FS^1(\cdot,\alp) \ d\alp - \frac{1}{\pi} \int_0^{2\pi} \be_t(\alp)k_\Cyl^1(\cdot,\alp) \ d\alp \\
\be_t = -\frac{1}{\pi}\int_0^{2\pi} \be_t(\alp)k_\Cyl^2(\cdot,\alp) \ d\alp - \frac{1}{\pi}\int_0^{2\pi} \y(\alp) k_{\FS,t}^2(\cdot,\alp) \ d\alp \\
\hspace{0.75cm} - \frac{1}{\pi}\int_0^{2\pi} \y_t(\alp) k_\FS^2(\cdot,\alp) \ d\alp - \frac{1}{\pi}\int_0^{2\pi} \w_t(\alp) k_\B^2(\cdot,\alp) \ d\alp \\
\tta(t=0) = \tta_0, \ \y(t=0) = \y_0, \ \w(t=0) = \w_0, \ \be(t=0) = \be_0
\end{dcases}.
\end{equation}

\begin{rmk}
\label{AlgorithmPaperEqnsRmk}
\begin{enumerate}
\item Compare the integral kernels given above in equations \eqref{eqn:SurfaceIntegralKernels}-\eqref{eqn:ObstacleIntegralKernels} with the $K_{kj}$ and $G_{kj}$ in Table 1 in \cite{AMEA}. Note that there are superficial differences between the kernels we use and the kernels in \cite{AMEA} due to a minor difference of how the arclength terms $s_{k,\al}$ are handled, but they are otherwise the same.
\item The equations in \eqref{eqn:WaterWavesSystem} correspond to the the first equation in (2.10), equation (4.14) and the system (4.17) with $N=2$ in \cite{AMEA}. The equation we utilize for $\y_t$ in \eqref{eqn:WaterWavesSystem} more closely corresponds to the evolution equation obtained in Appendix D of \cite{AMEA}.
\end{enumerate}
\end{rmk}

\begin{rmk}
\label{NonlocalRmk}
As noted above, the evolution equations for $\y$, $\w$ and $\be$ are nonlocal. In fact, we can now clearly see that the system \eqref{eqn:WaterWavesSystem} is of the form \eqref{eqn:WaterWavesSystemForm}. We shall refer to $\F(\Tta)$ as the right-hand side of the system and write $\F = (\F_1,\F_2,\F_3,\F_4)^t$. Since $(\id + \K)$ is invertible (see \cite{AMEA} or Appendix B below), we have
\begin{equation*}
\p_t\Tta = (\id + \K[\Tta])^{-1}\F(\Tta).
\end{equation*}
This motivates the plan of attack outlined earlier:
\begin{enumerate}
\item Obtain energy estimates for the model problem \eqref{eqn:ModelProblem}.
\item Use mapping properties of $(\id + \K[\cdot])^{-1}$ to conclude that the estimates still hold for the full system \eqref{eqn:WaterWavesSystem}.
\end{enumerate}
\end{rmk}

\section{The Right-Hand Side $\F$}

In order to carry out the strategy outlined in Remark \ref{NonlocalRmk}, we will need to determine which terms belong to the right-hand side $\F(\Tta)$ and write down the model problem \eqref{eqn:ModelProblem} in a way that is amenable to carrying out the needed energy estimates. This will involve exploiting some subtle cancellation. We begin by decomposing the system \eqref{eqn:WaterWavesSystem} into terms that belong to $\K[\Tta]\Tta_t$ (i\@.e\@., those that involve a nonlocal operator acting on $\y_t$, $\w_t$ or $\be_t$; no equation involves nonlocal operators acting on $\tta_t$) and those that belong in the right-hand side $\F(\Tta)$ (all other terms). Noting that the evolution equation for $\tta$ contains no nonlocal terms, we write
\begin{align*}
\y_t &= \Loc_\y + \Nonloc_\y,\\
\w_t &= \Loc_\w + \Nonloc_\w,\\
\be_t &= \Loc_\be + \Nonloc_\be,
\end{align*}
where the $\Loc$ terms belong to the right-hand side and the $\Nonloc$ terms arise from $\K[\Tta]$ being applied to $\Tta_t$. This can be done immediately in the case of the $\w_t$ equation and the $\be_t$ equation. In particular, we have
\begin{align}
\Loc_\w &= -\frac{1}{\pi}\int_0^{2\pi} \y(\alp) k_{\FS,t}^1(\al,\alp) \ d\alp, \label{eqn:LocwDef}\\
\Loc_\be &= -\frac{1}{\pi}\Int \y(\alp) k_{\FS,t}^2(\al,\alp) \ d\alp.
\end{align}
Then, $\Nonloc_\w$ contains the remaining integrals in \eqref{eqn:wEvolutionEqnKer}, multiplied by $2$ to clear the factor of $\frac{1}{2}$ in front of $\w_t$, with $\Nonloc_\be$ defined analogously from equation \eqref{eqn:beEvolutionEqnKer}. 

For the $\y_t$ equation, we begin by noticing that the only terms in $N_\y$ will arise from $\Wt_t$; in particular, only $\BR_t$, $\Y_t$ and $\Z_t$ will contribute terms to $N_\y$. As such, we will write $\BR_t = F_\BR + N_\BR$, $\Y_t = F_\Y + N_\Y$ and $\Z_t = F_\Z + N_\Z$. We now compute the relevant pieces of $\Wt_t$, integrating by parts to retain the cotangent kernel:
\begin{align}
\p_t\Comp(\BR)^*(\al) &= \frac{1}{4\pi i} \pv \Int \y_t(\alp) \cot\frac{1}{2}(\z(\al) - \z(\alp)) \ d\alp \nonumber\\
& \hspace{0.5cm} + \frac{1}{4\pi i} \pv \Int \p_{\alp}\left( \frac{\y(\alp)(\Ut(\al) - \Ut(\alp))}{\z_\al(\alp)} \right)\cot\frac{1}{2}(\z(\al) - \z(\alp)) \ d\alp, \label{eqn:WTimeDeriv}\\
\p_t\Comp(\Y)^*(\al) &= \frac{1}{4\pi} \Int \w_t(\alp) s_{1,\al}(\alp)\cot\frac{1}{2}(\z(\al) - \z_1(\alp)) \ d\alp \nonumber\\
& \hspace{0.5cm} + \frac{1}{4\pi} \Int \p_{\alp}\left( \frac{\w(\alp)s_{1,\al}(\alp)\Ut(\al)}{\z_{1,\al}(\alp)} \right) \cot\frac{1}{2}(\z(\al) - \z_1(\alp)) \ d\alp, \label{eqn:YTimeDeriv}\\
\p_t\Comp(\Z)^*(\al) &= \frac{1}{4\pi} \Int \be_t(\alp)s_{2,\al}(\alp) \cot\frac{1}{2}(\z(\al) - \z_2(\alp)) \ d\alp \nonumber\\
& \hspace{0.5cm} + \frac{1}{4\pi}\Int \p_{\alp}\left( \frac{\be(\alp)s_{2,\al}(\alp)\Ut(\al)}{\z_{2,\al}(\alp)} \right) \cot\frac{1}{2}(\z(\al) - \z_2(\alp)) \ d\alp. \label{eqn:ZTimeDeriv}
\end{align}
Now, we can clearly see that $\Comp(\Loc_\BR)^*$ is the second integral in equation \eqref{eqn:WTimeDeriv} and $\Comp(\Nonloc_\BR)^*$ is the first integral. It is the same for $\Loc_\Y$, $\Loc_\Z$, $\Nonloc_\Y$ and $\Nonloc_\Z$.

\subsection{Rewriting $F_\BR$}

Given that $F_\BR$ is given by a singular integral, it will be beneficial to decompose it into smaller pieces. This decomposition will additionally give rise to the previously mentioned cancellation. We begin by using the Leibniz rule to rewrite $\Loc_\BR$:
\begin{align*}
\Comp(\Loc_\BR)^* &= \frac{1}{4\pi i}\pv\Int \p_{\alp}\left(\frac{\y(\alp)}{\z_\al(\alp)}\right)(\Ut(\al) - \Ut(\alp))\cot\frac{1}{2}(\z(\al) - \z(\alp)) \ d\alp\\
&\hspace{0.5cm} - \frac{1}{4\pi i}\pv\Int \frac{\y(\alp)}{\z_\al(\alp)} \z_{t\al}(\alp)\cot\frac{1}{2}(\z(\al)-\z(\alp)) \ d\alp.
\end{align*}
We want to rewrite $\z_{t\al}$. Utilizing the identity $\z_\al = s_\al e^{i\tta}$ gives
\begin{equation*}
\p_t\z_\al = \p_t(s_\al e^{i\tta}) = s_{\al t}e^{i\tta} + s_\al(i\tta_te^{i\tta}) = \frac{s_{\al t}}{s_\al}\z_\al + i\tta_t\z_\al.
\end{equation*}
We now substitute equation \eqref{eqn:ThetaEvolutionEquation} for $\tta_t$ to obtain
\begin{equation}
\label{eqn:Ut_aEqn}
\z_{t\al} = \frac{s_{\al t}}{s_\al}\z_\al + i\z_\al\left( \frac{1}{2s_\al^2}\HT(\y_\al) + \frac{\tta_\al}{s_\al}\left(V - \Wt \cdot \ut \right) + \frac{1}{s_\al}\Wo_\al \cdot \un + \frac{\m \cdot \un}{s_\al} \right).
\end{equation}

We can now decompose $F_\BR$ into a singular term involving the Hilbert transform and a remainder term involving a smoothing operator $\Kop$. To carry this out, we make use of a similar decomposition of the Birkhoff-Rott integral given above in \eqref{eqn:WToHK}. Decomposing $F_\BR$ similarly yields
\begin{align}
\label{eqn:Loc_WToHK}
\Comp(\Loc_\BR)^* &= \comm{\Ut}{\HT}\left(\frac{1}{\z_\al}\p_\al\left(\frac{\y}{\z_\al}\right)\right) + \comm{\Ut}{\Kop[\z]}\left(\p_\al\left(\frac{\y}{\z_\al}\right)\right) \nonumber\\
&\hspace{0.5cm} - \frac{1}{2i}\HT\left(\frac{\z_{t\al}}{\z_\al}\left(\frac{\y}{\z_\al}\right)\right) - \Kop[\z]\left(\z_{t\al}\left(\frac{\y}{\z_\al}\right)\right).
\end{align}
We will then substitute in equation \eqref{eqn:Ut_aEqn}. After substituting, we will factor some of the terms out of the Hilbert transform, thus picking up some commutators, exploit the identity $\HT^2 = -\id$ and do a bit of rearranging. The result of these operations is
\begin{align}
\label{eqn:FinalF_BR}
\Comp(\Loc_\BR)^* &= \comm{\Ut}{\HT}\left(\frac{1}{\z_\al}\p_\al\left(\frac{\y}{\z_\al}\right)\right) + \comm{\Ut}{\Kop[\z]}\left(\p_\al\left(\frac{\y}{\z_\al}\right)\right) - \frac{s_{\al t}}{2is_\al} \HT\left( \frac{\y}{\z_\al} \right) \nonumber\\
&\hspace{0.5cm} - \frac{s_{\al t}}{s_\al}\Kop[\z]\y + \frac{\y\y_\al}{4s_\al^2\z_\al} - \frac{1}{4s_\al^2}\left[\HT,\frac{\y}{\z_\al}\right](\HT(\y_\al)) - \frac{i}{2s_\al^2}\Kop[\z](\y\HT(\y_\al)) \nonumber\\
&\hspace{0.5cm} - \frac{1}{2s_\al}\HT\left(\frac{\y \m\cdot \un}{\z_\al}\right) - \frac{i}{s_\al}\Kop[\z](\y \m\cdot\un) -\frac{V - \Wt \cdot \ut}{2s_\al\z_\al}\HT(\y\tta_\al) \nonumber\\
&\hspace{0.5cm} - \frac{1}{2s_\al}\left[\HT,\frac{V-\Wt\cdot\ut}{\z_\al}\right](\y\tta_\al) - \frac{i}{s_\al}\Kop[\z](\y\tta_\al(V - \Wt \cdot\ut)) - \frac{1}{2s_\al}\HT\left(\frac{\y \Wo_\al \cdot \un}{\z_\al}\right) \nonumber\\
&\hspace{0.5cm} - \frac{i}{s_\al}\Kop[\z](\y\Wo_\al\cdot\un).
\end{align}
This is the decomposed version of $F_\BR$ which we shall use. We can now see the cancellation that will occur between $F_\BR$ and $(V - \Wt\cdot\ut)\Wt_\al\cdot\ut$.

\subsection{Obtaining the Cancellation}

To obtain the desired cancellation, we begin by considering
\begin{align*}
(V - \Wt \cdot \ut)\Wt_\al \cdot \ut &= (V - \Wt \cdot\ut)(\BR_\al \cdot \ut + \Wo_\al \cdot\ut)\\
&= (V - \Wt \cdot \ut)\left(-\frac{1}{2s_\al}\HT(\y\tta_\al) + \m \cdot \ut + \Wo_\al \cdot\ut\right). 
\end{align*}
We therefore have
\begin{align*}
2(V - \Wt \cdot \ut)\Wt_\al \cdot \ut - 2s_\al\Loc_\BR \cdot \ut &= -\frac{V - \Wt \cdot \ut}{s_\al}\HT(\y\tta_\al) + \frac{V - \Wt \cdot \ut}{s_\al}\HT(\y\tta_\al) \\
&\hspace{0.5cm} +  2(V - \Wt \cdot \ut)(\m \cdot \ut + \Wo_\al \cdot\ut) - 2s_\al \yb_0 \cdot \ut\\
&= 2(V - \Wt \cdot \ut)(\m \cdot \ut + \Wo_\al \cdot\ut) - 2s_\al \yb_0 \cdot \ut,
\end{align*}
where
\begin{equation*}
\Comp(\yb_0)^* \coloneqq \Comp(\Loc_\BR)^* + \frac{V - \Wt \cdot \ut}{2s_\al\z_\al}\HT(\y\tta_\al).
\end{equation*}
Most of the terms in $\yb_0$ shall be routine to estimate, however we do have one transport term which we wish to isolate. As such, we write
\begin{equation*}
\Comp(\yb_0)^* = \frac{\y\y_\al}{4s_\al^2\z_\al} + \Comp(\yb_1)^*,
\end{equation*}
which implies that
\begin{equation*}
2s_\al \yb_0 \cdot \ut = \frac{\y\y_\al}{2s_\al^2} + 2\Rea\Big\{ \Comp(\yb_1)^*\z_\al \Big\}.
\end{equation*}

This prepares us to write down the right-hand side of the $\y_t$ equation (those terms belonging to $\F_2$):
\begin{align}
\label{eqn:yLocEvolutionEqnNew}
\F_2(\Tta) &= \frac{2\tau}{s_\al}\tta_{\al\al} + \frac{\y}{2s_\al^2}\HT(\y\tta_\al) + \frac{\y_\al}{s_\al}\left(V - \Wt \cdot \ut\right) - \frac{\y\y_\al}{s_\al^2}  \nonumber\\ 
&\hspace{0.5cm} + \frac{\y}{s_\al}\left( s_{\al t} - \Wo_\al \cdot \ut - \m \cdot \ut \right) - 2g\eta_\al + 2\left(V - \Wt \cdot \ut\right)(m \cdot \ut + \Wo_\al \cdot \ut) \nonumber\\
&\hspace{0.5cm} - 2s_\al\big[\yb_1 + \Loc_\Y + \Loc_\Z + \chi\p_t(\nab\vphic(\z))\big] \cdot \ut.
\end{align}

\subsection{Writing Down the System $\Tta_t = \F(\Tta)$}

As previously noted, we will first consider the model problem \eqref{eqn:ModelProblem}. In \eqref{eqn:ModelProblem}, the right-hand side $\F(\Tta)$ is given by
\begin{align}
\label{eqn:FullRHS}
\F_1(\Tta) &= \frac{1}{2s_\al^2}\HT(\y_\al) + \frac{\tta_\al}{s_\al}\left(V - \Wt \cdot \ut \right) + \frac{1}{s_\al}\Wo_\al \cdot \un + \frac{\m \cdot \un}{s_\al} \nonumber\\
\F_2(\Tta) &=  \frac{2\tau}{s_\al}\tta_{\al\al} + \frac{\y}{2s_\al^2}\HT(\y\tta_\al) + \frac{\y_\al}{s_\al}\left(V - \Wt \cdot \ut\right) - \frac{\y\y_\al}{s_\al^2}  \nonumber\\ 
&\hspace{0.5cm} + \frac{\y}{s_\al}\left( s_{\al t} - \Wo_\al \cdot \ut - \m \cdot \ut \right) - 2g\eta_\al + 2\left(V - \Wt \cdot \ut\right)(\m \cdot \ut + \Wo_\al \cdot \ut) \nonumber\\
&\hspace{0.5cm} - 2s_\al\big[\yb_1 + \Loc_\Y + \Loc_\Z + \chi\p_t(\nab\vphic(\z))\big] \cdot \ut \nonumber\\
\F_3(\Tta) &= -\frac{1}{\pi} \Int \y(\alp) k_{\FS,t}^1(\al,\alp) \ d\alp \nonumber\\
\F_4(\Tta) &= -\frac{1}{\pi}\Int \y(\alp) k_{\FS,t}^2(\al,\alp) \ d\alp.
\end{align}

Though simpler than \eqref{eqn:WaterWavesSystem}, the system \eqref{eqn:ModelProblem} is still a rather complicated, quasilinear system. In order to handle this, we will utilize an approach which is quite common in the study of quasilinear hyperbolic equations. Namely, we will first work with a regularized version of our system and then pass to the limit as the regularization parameter $\de \to 0^+$ to solve the non-regularized system. The regularization scheme that we shall use is much like the one used in \cite{A1} and the interested reader can consult this paper for further details (see also \cite{MB} or \cite{T1} for more on such regularization schemes).

\subsection{The Regularized Evolution Equations for the System \eqref{eqn:ModelProblem}}

Now, we want to obtain an appropriately regularized version of the system \eqref{eqn:ModelProblem}. We begin by simply writing down the regularized evolution equations, and then we will go back to briefly discuss how the regularized terms are constructed. Beginning with $\tta$, we have
\begin{equation}
\label{eqn:ThetaRegEvolutionEqn}
\tta_t^\de = \frac{1}{2(s_\al^\de)^2}\HT(\J\y_\al^\de) + \frac{1}{s_\al^\de}\J\left(\left(V^\de-\Wt^\de\cdot\ut^\de\right)\J\tta_\al^\de\right) + \frac{1}{s_\al^\de}\Wo_\al^\de \cdot \un^\de + \frac{\m^\de \cdot \un^\de}{s_\al^\de} + \mu^\de.
\end{equation}
Notice that there is no term corresponding to $\mu^\de$ in the non-regularized equation. Its purpose is to enforce the condition that $\z^\de(\al)-\al$ be $2\pi$-periodic and it is given by
\begin{equation}
\label{eqn:muDef}
\mu^\de(t) \coloneqq -\frac{\displaystyle\Int s_{\al t}^\de\z_\al^\de + i U_\al^\de\z_\al^\de + V^\de \z_{\al\al}^\de \ d\al}{is_{\al}^\de\displaystyle\Int \z_\al^\de \ d\al}.
\end{equation}
See \cite{A1} for the derivation of $\mu^\de$ and the proof that it enforces the aforementioned periodicity condition. The same calculations and arguments work in the present setting with the only difference being the terms contained in $U$. We also remark that $\mu^\de$ is entirely distinct from the density $\mu$ of the double layer potential on the free boundary.

We now turn to the $\y_t$ equation:
\begin{equation}
\label{eqn:yRegEvolutionEqn}
\y_t^\de = \frac{2\tau}{s_\al^\de}\J\tta_{\al\al}^\de + \frac{1}{2(s_\al^\de)^2}\HT((\y^\de)^2\J\tta_\al^\de) + \frac{1}{s_\al^\de}\J\left(\left(V^\de - \Wt^\de \cdot \ut^\de\right)\J\y_\al^\de\right) - \frac{\J(\y^\de\J\y_\al^\de)}{(s_\al^\de)^2} + m_\y^\de.
\end{equation} 
The term $m_\y^\de$ is primarily a remainder term, but it does contain one term not appearing in the non-regularized system. Notice that in the regularized evolution equation for $\y$ we have pulled a factor of $\y^\de$ through the Hilbert transform. The cost of doing so is a (smooth) commutator which we also include in $m_\y^\de$. We thus have
\begin{align}
\label{eqn:myDef}
m_\y^\de &= \frac{\y^\de}{s_\al^\de}\left( s_{\al t}^\de - \Wo_\al^\de \cdot \ut^\de - \m^\de \cdot \ut^\de \right) - 2g\eta_\al^\de + 2\J\left((V^\de - \Wt^\de \cdot \ut^\de)\J(\m^\de \cdot \ut^\de + \Wo_\al^\de \cdot \ut^\de)\right) \nonumber\\
&\hspace{0.5cm} - 2s_\al^\de\J\left(\big[\yb_1^\de + \Loc_\Y^\de + \Loc_\Z^\de + \chi\p_t(\nab\vphic(\z^\de))\big] \cdot \ut^\de\right) - \big[\HT,\y^\de\big]\left( \frac{\y^\de\J\tta_\al^\de}{2(s_\al^\de)^2} \right).
\end{align}
For $\w$ and $\be$, we have
\begin{equation}
\label{eqn:wRegEvolutionEqn}
\w_t^\de = -\frac{1}{\pi}\Int \y^\de(\alp)k_{\FS,t}^{1,\de}(\al,\alp) \ d\alp
\end{equation}
and
\begin{equation}
\label{eqn:beRegEvolutionEqn}
\be_t^\de = -\frac{1}{\pi}\Int \y^\de(\alp)k_{\FS,t}^{2,\de}(\al,\alp) \ d\alp.
\end{equation}

The regularized system we consider is then
\begin{equation}
\label{eqn:FullRegSystem}
\begin{dcases}
\tta_t^\de = \frac{1}{2(s_\al^\de)^2}\HT(\J\y_\al^\de) + \frac{1}{s_\al^\de}\J\left(\left(V^\de-\Wt^\de\cdot\ut^\de\right)\J\tta_\al^\de\right) + \frac{1}{s_\al^\de}\Wo_\al^\de \cdot \un^\de + \frac{\m^\de \cdot \un^\de}{s_\al^\de} + \mu^\de \\
\y_t^\de = \frac{2\tau}{s_\al^\de}\J\tta_{\al\al}^\de + \frac{1}{2(s_\al^\de)^2}\HT((\y^\de)^2\J\tta_\al^\de) + \frac{1}{s_\al^\de}\J\left(\left(V^\de - \Wt^\de \cdot \ut^\de\right)\J\y_\al^\de\right) - \frac{\J(\y^\de\J\y_\al^\de)}{(s_\al^\de)^2} + m_\y^\de \\
\w_t^\de = -\frac{1}{\pi}\Int \y^\de(\alp)k_{\FS,t}^{1,\de}(\al,\alp) \ d\alp \\
\be_t^\de = -\frac{1}{\pi}\Int \y^\de(\alp)k_{\FS,t}^{2,\de}(\al,\alp) \ d\alp \\
\tta^\de(t=0) = \tta_0, \ \y^\de(t=0) = \y_0, \ \w^\de(t=0) = \w_0, \ \be^\de(t=0) = \be_0
\end{dcases}.
\end{equation}

We shall now succinctly describe the various terms appearing in the regularized equations, beginning with the family of mollifiers $\J$. For each $\de > 0$, we have a corresponding operator $\J$, which is an approximation of the identity. There are a number of different ways which we can conceptualize these operators. In the spatially periodic setting, a convenient conceptualization, and the one we employ, is the following: the operator $\J$ represents truncation of the Fourier series via zeroing out modes with wavenumber greater than $\de^{-1}$. Alternatively, and equivalently, one might also conceptualize $\J$ as convolution with an approximation of the Dirac mass depending on the parameter $\de$. Most importantly, $\J$ shall be self-adjoint and will commute with derivatives as well as the Hilbert transform. We now state two lemmas regarding the action of $\J$ on Sobolev spaces $H^r$. The first is

\begin{lemma}
\label{JLemma1}
If $\de > 0$ and $u \in H^{r}$ for some $r \in \R$. Then, for any $k \in \N_0$, we have $\J u \in H^{r+k}$ with
\begin{equation*}
\Norm{\J u}{r+k} \lesssim \de^{-k}\Norm{u}{r}.
\end{equation*}
\end{lemma}
\begin{proof}
See Lemma 3.5 in \cite{MB}.
\end{proof}

Lemma \ref{JLemma1} communicates a couple of interesting properties of the mollifiers. First, if we take $k=0$, we see that Lemma \ref{JLemma1} tells us that, for any $\de > 0$, $\J$ is a bounded (and therefore continuous) linear operator on $H^r$ for any $r \in \R$. The second is that we can, loosely speaking, exchange derivatives of $\J u$ for powers of $\de^{-1}$. This is, in fact, a Bernstein-type lemma regarding the action of the derivative on band-limited functions.

The next result we shall need is

\begin{lemma}
\label{JLemma2}
For $f \in H^1$ and $\de, \ \delt > 0$,
\begin{equation*}
\norm{\J f - \Jt f}_\Lp{2} \leq \max(\de, \delt) \Norm{f}{1}.
\end{equation*}
\end{lemma}
\begin{proof}
Again, see Lemma 3.5 in \cite{MB}.
\end{proof}
Let $\set{\de_k}$ be a sequence of real numbers with $\de_k \to 0^+$. Then, Lemma \ref{JLemma2} tells us that $\set{\mathcal{J}_{\de_k}u}$ is a Cauchy sequence in $L^2$ as soon as $u \in H^1$.

\begin{rmk}
\label{NotationRmk}
Here is a good place to introduce some notational conventions which we shall utilize.
\begin{enumerate}
\item We use $A \lesssim B$ to denote $A \leq C B$ for some constant $C > 0$.
\item We take $A \lesssim_{a_1,\ldots,a_k} B$ to mean $A \leq C(a_1,\ldots,a_k)B$.
\item By $A \sim B$ we mean $B \lesssim A \lesssim B$.
\item Finally, for $r \in \R$, $r+$ denotes $r+h$ for some small, positive parameter $h$. For example, by Lemma \ref{SobolevMultiplication}, we have
\begin{equation*}
\norm{uv}_\Lp{2} \lesssim \norm{u}_\Lp{2}\Norm{v}{\half+}.
\end{equation*}
\end{enumerate}
\end{rmk}

Most of the nuance in defining the regularized terms lies in constructing $\z^\de$ and $\BR^\de$. We shall define $\z^\de$ and $\BR^\de$ exactly as in \cite{A1} and the interested reader can find all of the details in that paper. The remaining regularized terms are defined in the same way as the non-regularized ones with $\z$, $\BR$, $\y$, etc\@. replaced with $\z^\de$, $\BR^\de$, $\y^\de$, etc. For example, $\Comp(\un^\de) \coloneqq \frac{i\z_\al^\de}{s_\al^\de}$, where $s_\al^\de \coloneqq \abs{\z_\al^\de}$, $\Tta^\de$ solves \eqref{eqn:FullRegSystem} and
\begin{equation*}
\Comp(\Y^\de)^*(\al) \coloneqq \frac{1}{4\pi}\Int \w^\de(\alp)s_{1,\al}(\alp)\cot\frac{1}{2}(\z^\de(\al)-\z_1(\alp)) \ d\alp.
\end{equation*} 

We now state some useful results regarding the term $\z_d$ and the operator $\Kop$ used in the decomposition \eqref{eqn:WToHK}.
\begin{lemma}
\label{zdNorm}
Let $r \geq 0$. If $\tta \in H^r$, then $\z_d \in H^{r+1}$ with the estimate
\begin{equation}
\Norm{\z_d}{r+1} \lesssim 1+ \Norm{\tta}{r}.
\end{equation}
\end{lemma}
\begin{proof}
We define $\z$ exactly the same as $z$ in \cite{A1}. Ergo, the desired estimate follows directly from Lemma 3.2 in \cite{A1}	.
\end{proof}

We include the following two results regarding mapping properties of $\Kop$ which will be of use to us.
\begin{lemma}
\label{KopEst}
If $\z_d \in H^{r+1}$, $r \in \ZZ$ with $r \geq 3$, then $\Kop[\z]: H^j \to H^{r+j-1}$, for $j \in \set{1,0,-1}$, with the estimate
\begin{equation}
\label{eqn:KopEst}
\Norm{\Kop[\z]f}{r+j-1} \lesssim \Norm{f}{j}(1 + \Norm{\tta}{r})^3.
\end{equation}
\end{lemma}
\begin{proof}
We shall show that $\Kop[\z]: H^{-1} \to H^{r-2}$ with the corresponding estimate; the proofs of the other claims are contained in Lemma 3.5 of \cite{A1}. In proving this mapping property, we follow the proof given in \cite{A1}. We begin by writing $K = K_1 + K_2$, where
\begin{align}
\Kop_1[\z]f(\al) &= \frac{1}{2\pi i} \int_0^{2\pi} f(\alp)\bigg[ \frac{1}{\z_d(\al) - \z_d(\alp)} - \frac{1}{\z_\al(\alp)(\al-\alp)} \bigg] \ d\alp, \label{eqn:Kop1Def}\\
\Kop_2[\z]f(\al) &= \frac{1}{4\pi i}\int_{\al-\pi}^{\al+\pi} f(\alp) \bigg[ g\left(\frac{1}{2}(\z_d(\al) - \z_d(\alp))\right) - \frac{1}{\z_\al(\alp)}g\left(\frac{1}{2}(\al-\alp)\right) \bigg] \ d\alp. \label{eqn:Kop2Def}
\end{align}
In the above definition, $g$ is a function, holomorphic at the origin, such that
\begin{equation*}
\cot z = \frac{1}{z} + g(z).
\end{equation*}
Notice that the choice of limits of integration in the definition of $\Kop_2$ allows us to integrate over one period while avoiding the poles of $g$, which by definition must be the non-zero integer multiples of $2\pi$ -- this choice of limits of integration will force $\abs{\al - \alp} \leq \pi$.

First, consider
\begin{equation*}
\p_\al^{r-2}\Kop_1[\z]f(\al) = \frac{1}{2\pi i} \Int f(\alp) \p_\al^{r-2}\left[ \frac{1}{\z_d(\al) - \z_d(\alp)} - \frac{1}{\z_\al(\alp)(\al-\alp)} \right] \ d\alp.
\end{equation*}
We then apply one of the $r-2$ derivatives to the quantity inside the brackets:
\begin{equation*}
\p_\al^{r-2}\Kop_1[\z]f(\al) = \frac{1}{2\pi i} \Int f(\alp) \p_\al^{r-3}\left[ -\frac{\z_\al(\al)}{(\z_d(\al) - \z_d(\alp))^2} + \frac{1}{\z_\al(\alp)(\al-\alp)^2} \right] \ d\alp.
\end{equation*}
By rearranging the factors of $\z_\al$, we can write the quantity in brackets as a derivative with respect to $\alp$:
\begin{equation*}
\p_\al^{r-2}\Kop_1[\z]f(\al) = \frac{1}{2\pi i} \Int \frac{f(\alp)}{\z_\al(\alp)} \p_\al^{r-3}\p_{\alp}\left[ \frac{1}{\al-\alp} - \frac{\z_\al(\al)}{\z_d(\al) - \z_d(\alp)} \right] \ d\alp.
\end{equation*}
Then, by integrating by parts and recognizing the quantity in brackets as a ratio of divided differences, we can rewrite this expression to obtain
\begin{equation*}
\p_\al^{r-2}\Kop_1[\z]f(\al) = \frac{1}{2\pi i} \Int \p_{\alp}^{-1}\left(\frac{f(\alp)}{\z_\al(\alp)}\right) \p_\al^{r-3}\p_{\alp}^2\left[ \frac{q_2(\al,\alp)}{q_1(\al,\alp)} \right] \ d\alp.
\end{equation*}
We introduced above some notation used in \cite{A1}:
\begin{equation}
\label{eqn:DividedDiffs}
q_1(\al,\alp) \coloneqq \frac{\z_d(\al) - \z_d(\alp)}{\al - \alp}, \ q_2(\al,\alp) \coloneqq \frac{\z_d(\al) - \z_d(\alp) - \z_\al(\al)(\al-\alp)}{(\al-\alp)^2}.
\end{equation}
Regarding the divided differences, we have the following result from Lemma 3.4 of \cite{A1} (also see \cite{BHL1}): If $\z_d \in H^r$, then
\begin{equation}
\label{eqn:DividedDiffsEsts}
\begin{dcases}
q_1 \in H_\al^{r-1} \text{ with } \norm{q_1}_{H_\al^{r-1}} \lesssim \Norm{\z_d}{r}\\
q_1 \in H_{\alp}^{r-1} \text{ with } \norm{q_1}_{H_{\alp}^{r-1}} \lesssim \Norm{\z_d}{r}\\
q_2 \in H_\al^{r-2} \text{ with } \norm{q_2}_{H_\al^{r-2}} \lesssim \Norm{\z_d}{r}\\
q_2 \in H_{\alp}^{r-2} \text{ with } \norm{q_2}_{H_{\alp}^{r-2}} \lesssim \Norm{\z_d}{r}
\end{dcases}.
\end{equation}

From here, we deduce the immediate bound
\begin{equation*}
\abs{\p_\al^{r-2}\Kop_1[\z]f(\al)} \lesssim \Norm{\frac{f}{\z_\al}}{-1}\Norm{\frac{q_2}{q_1}}{r-1}.
\end{equation*}
In particular, notice that since $\frac{q_2}{q_1}$ is in $H^{r-1}$, in both variables, we know that $\frac{q_2}{q_1}$ will be in $W_\al^{r-3,\infty}$ and $H_{\alp}^2$. Lemma \ref{SobolevMultiplication} and the Sobolev algebra property then imply that
\begin{equation*}
\abs{\p_\al^{r-2}\Kop_1[\z]f(\al)} \lesssim \Norm{f}{-1}\Norm{\frac{1}{\z_\al}}{1+}\Norm{q_2}{r-1}\Norm{\frac{1}{q_1}}{r-1}.
\end{equation*}
Finally, we can apply Lemma \ref{CompositionEst} in conjunction with \eqref{eqn:DividedDiffsEsts} to deduce that
\begin{equation}
\label{eqn:KopEst1}
\Norm{\Kop_1[\z]f}{r-2} \lesssim \Norm{f}{-1}(1+\Norm{\tta}{r})^3.
\end{equation}
A similar modification of the argument in \cite{A1} implies that
\begin{equation}
\label{eqn:KopEst2}
\Norm{\Kop_2[\z]f}{r-2} \lesssim \Norm{f}{-1}(1+\Norm{\tta}{r})^2.
\end{equation}
Combining \eqref{eqn:KopEst1} and \eqref{eqn:KopEst2} gives the desired result.
\end{proof}

\begin{lemma}
\label{KopLipschitzEst}
If $\tta,\tilde{\tta} \in H^1$, and the associated $\z, \tilde{\z}$ satisfy equations \eqref{eqn:ChordArc}, \eqref{eqn:LBound} and \eqref{eqn:EBound}, then we have the following Lipschitz estimate for $\Kop$:
\begin{equation}
\Norm{\Kop[\z]f - \Kop[\tilde{\z}]f}{1} \lesssim \Norm{f}{1}\Norm{\tta - \tilde{\tta}}{1}.
\end{equation}
\end{lemma}
\begin{proof}
See Lemma 3.6 in \cite{A1}.
\end{proof}

As noted earlier, the above regularization scheme is common in studying quasilinear PDE. The usual plan of attack in using such a scheme is to prove that solutions to the regularized equations exist and that those solutions satisfy an appropriate uniform (in $\de$) energy estimate. The energy estimate allows one to deduce a common existence time (independent of $\de$) for the regularized solutions. Then, one can show that the limit as $\de \to 0^+$ of the regularized solutions exists and satisfies the non-regularized system. Carrying out the above plan will be the focus of the next two sections. We will begin by defining a suitable energy and then establishing the uniform energy estimate.

\section{The Energy Estimate}

Now that we have the appropriate evolution equations, as well as the above preliminary remarks and results under our belts, we shall begin the process of proving the first main result. The results in the next two sections are all concerning the regularized equations. For the sake of the reader, we shall, for the most part, drop the $\delta$ notation in the regularized equations. The reader should presume all quantities are regularized in the manner discussed above unless and until otherwise stated.

A quantity which shall be of fundamental importance to the analysis in the sequel is the energy for a solution $(\tta,\y,\w,\be)$.

\begin{defn}
\label{EnergyDef}
Inspired by \cite{A1}, we define the energy of a solution to the regularized system as follows
\begin{equation}
\label{eqn:Energy}
\E(t) = \E^0(t) + \E^1(t) + \sum_{j=2}^{s+1} \E^j(t),
\end{equation}
where
\begin{align}
\E^0 &= \frac{1}{2}\left(\norm{\tta}_\Lp{2}^2 + \norm{\y}_\Lp{2}^2 + \norm{\w}_\Lp{2}^2 + \norm{\be}_\Lp{2}^2\right), \label{eqn:E0}\\
\E^1 &= \frac{1}{2}\left( \norm{\p_\al\w}_\Lp{2}^2 + \norm{\p_\al\be}_\Lp{2}^2 \right), \label{eqn:E1}\\
\E^j &= \frac{1}{2}\Int (\p_\al^{j-1}\tta)^2 + \frac{1}{4\tau s_\al}(\p_\al^{j-2}\y)\La(\p_\al^{j-2}\y) + \frac{\y^2}{16\tau^2 s_\al^2}(\p_\al^{j-2}\y)^2 \ d\al \qquad (2 \leq j \leq s+1). \label{eqn:Ej}
\end{align}
We define $\La \coloneqq \HT\p_\al$ and note that $\La$ is a Fourier multiplier: $\La = \abs{D}$. We will write $\E^j = \E_1^j + \E_2^j + \E_3^j$.
\end{defn}

We note that in \cite{A1}, the coefficient of surface tension appeared in the energy implicitly via the Weber number:
\begin{equation*}
\We = \frac{\rho_1 + \rho_2}{2\tau}.
\end{equation*}
In our case (i\@.e\@., the case of water waves), we have $\We = \frac{1}{2\tau}$.

\begin{defn}
\label{EnergySpaceDefns}
For $\E$ as above, we have
\begin{equation}
\label{eqn:EnergySobolev}
\E(t) \sim \Norm{\tta(t)}{s}^2 + \Norm{\y(t)}{\sHalf}^2 + \Norm{\w(t)}{1}^2 + \Norm{\be(t)}{1}^2 = \norm{\Tta(t)}_{H^s\times H^{\sHalf} \times H^1 \times H^1}^2.
\end{equation}
We therefore define the energy space to be $X \coloneqq H^s \times H^\sHalf \times H^1 \times H^1$. We shall let $\X$ denote the subset of $X$ where three conditions are satisfied:
\begin{itemize}
\item the chord-arc condition \eqref{eqn:ChordArc} holds,
\item we have
\begin{equation}
\label{eqn:LBound}
s_\al \geq 1,
\end{equation}
with equality holding in the case $\tta = 0$,
\item and
\begin{equation}
\label{eqn:EBound}
\E < \eb
\end{equation}
for some $0 < \eb < +\infty$.
\end{itemize}
Henceforth, we shall for the most part restrict our attention to $\X$ as this is where we shall seek solutions.
\end{defn}

\begin{rmk}
\label{EnergyRmk}
We shall assume throughout that $s$ is sufficiently large for all computations to make sense; we are not seeking sharp regularity results. Here we simply remark that we shall at least require that $s > \frac{3}{2}$. Notice then that, by Lemma \ref{SobolevEmbedding}, $H^{\sHalf} \hookrightarrow L^\infty$, and therefore $\Tta \in (L^\infty)^4$. Lemma \ref{zdNorm} implies that $\z_d \in H^{s+1}$. We will further have $\psi \coloneqq \varphi\rvert_\FS \in H^{s+\half}$, therefore $\varphi \in H^{s+1}$ and $\vel = \nab\varphi \in H^s$. It follows, again from Lemma \ref{SobolevEmbedding}, that $\z, \vel \in \Lip$. This is in line with the standard regularity requirements for proving local well-posedness by energy methods (see, e\@.g\@., \cite{ABZ1}). Further, the definition of $s_\al$, the definition of the energy and the bound on the energy in \eqref{eqn:EBound} imply that $s_\al \in L^\infty$. Of course, this implies that $L \in L^\infty$ as well.
\end{rmk}

Definition \ref{EnergySpaceDefns} implies that, for $\Tta \in \X$, we have $\norm{\Tta}_X \lesssim 1$. Further, by Remark \ref{EnergyRmk}, we also have $\norm{s_\al}_\Lp{\infty},\norm{L}_\Lp{\infty} \lesssim 1$. Before proceeding to the main energy estimate, we begin by obtaining some a priori estimates for some important quantities appearing in our evolution equations. These estimates will be used repeatedly in the sequel when proving the main energy estimate.

\begin{lemma}
\label{WtEst}
The following estimates hold for $s$ sufficiently large:
\begin{align}
\norm{\BR}_\Lp{2} &\lesssim \sqrt{\E} + \E^2, \label{eqn:WL2Norm}\\
\norm{\Y}_\Lp{2} &\lesssim \sqrt{\E}, \label{eqn:YL2Norm}\\
\norm{\Z}_\Lp{2} &\lesssim \sqrt{\E}, \label{eqn:ZL2Norm}\\
\norm{\Y}_{s+1} &\lesssim \sqrt{\E} + \E, \label{eqn:YHsNorm}\\
\norm{\Z}_{s+1} &\lesssim \sqrt{\E} + \E, \label{eqn:ZHsNorm}\\
\norm{\nab\vphic(\z)}_{s+1} &\lesssim 1+\sqrt{\E} \label{eqn:GradPhiCylHsNorm}.
\end{align}
These estimates hold for both the regularized and non-regularized terms.
\end{lemma}
\begin{proof}
We use the representation \eqref{eqn:WToHK} and Lemma \ref{KopEst} to estimate
\begin{equation*}
\norm{\BR}_\Lp{2} \lesssim \norm{\y}_\Lp{2}\norm{\frac{1}{\z_\al}}_\Lp{\infty} + \norm{\y}_1(1 + \ttanorm)^3.
\end{equation*}
It then follows that
\begin{equation}
\label{eqn:W}
\norm{\BR}_\Lp{2} \lesssim \ynorm(1+\ttanorm)^3.
\end{equation}

To estimate the norm of $\Y$, consider
\begin{equation*}
\abs{\Comp(\Y)^*(\al)} \lesssim \Int \abs{\w(\al^\prime)}\abs{s_{1,\al}(\al^\prime)\cot\frac{1}{2}(\z(\al) - \z_1(\al^\prime))} \ d\al^\prime \lesssim \norm{\w}_\Lp{2}.
\end{equation*}
This implies the estimate \eqref{eqn:YL2Norm}. Next, we consider
\begin{align*}
\abs{\p_\al^{s+1}\Comp(\Y)^*(\al)} &\leq \frac{1}{4\pi} \Int \abs{\w(\al^\prime)s_{1,\al}(\al^\prime)}\abs{\p_\al^{s+1}\cot\frac{1}{2}(\z(\al) - \z_1(\al^\prime))} \ d\al^\prime\\
&\lesssim \norm{\w}_\Lp{2}\Norm{\p_\al\cot\frac{1}{2}(\z(\al) - \z_1(\cdot))}{s}
\end{align*}
It then follows from Lemma \ref{CompositionEst} that
\begin{equation*}
\Norm{\Y}{s+1} \lesssim \wnorm + \wnorm(1 + \Norm{\z_\al}{s}) \lesssim \wnorm(1+\ttanorm).
\end{equation*}

The proofs of \eqref{eqn:ZL2Norm} and \eqref{eqn:ZHsNorm} are nearly identical to that of \eqref{eqn:YL2Norm} and \eqref{eqn:YHsNorm}.
Next, recalling the definition of $\vphic$ in \eqref{eqn:Phi_cyl}, it is easy to see that $\nab\vphic$ is a smooth function and so we can apply Lemma \ref{CompositionEst} to obtain
\begin{equation}
\label{eqn:GradPhiCylEst}
\Norm{\nab\vphic(\z)}{s+1} \lesssim (1 + \Norm{\z}{s+1}) \lesssim 1 + \ttanorm.
\end{equation}
\end{proof}

\begin{lemma}
\label{UnitVecEsts}
We can control the $H^s$ norms of the unit vectors $\un$ and $\ut$ (both regularized and non-regularized) in $\X$ where we have the following estimates:
\begin{align}
\Norm{\un}{s} &\lesssim 1 + \sqrt{\E}. \label{eqn:unHsEst}\\
\Norm{\ut}{s} &\lesssim 1 + \sqrt{\E}, \label{eqn:utHsEst}
\end{align}
\end{lemma}
\begin{proof}
We shall only prove the estimate for $\un$ as the argument for $\ut$ is totally analogous. Upon writing $\Comp(\un) = \frac{i\z_\al}{s_\al}$, Lemma \ref{zdNorm} gives
\begin{equation*}
\Norm{\un}{s} \lesssim \Norm{\z_\al}{s} \leq \Norm{\z_d}{s+1} \lesssim 1 + \ttanorm.
\end{equation*}
\end{proof}

\begin{lemma}
\label{ArclengthEltBounds}
Let $s \in \R$ be sufficiently large. Then, on $\X$, we can bound $s_\al$ above and below by
\begin{equation}
\label{eqn:ArclengthEltBounds}
1 \leq s_\al \lesssim 1 + \sqrt{\mathfrak{e}}.
\end{equation}
This estimate holds for the non-regularized $s_\al$ and the regularized $s_\al^\de$. 
\end{lemma}
\begin{proof}
The lower bound is simply equation \eqref{eqn:LBound} in the definition of $\X$. To obtain the upper bound, we can apply the definition of $s_\al$, Lemma \ref{zdNorm} and Lemma \ref{SobolevEmbedding}. In particular, these results together imply that
\begin{equation*}
s_\al \leq \norm{\z_\al}_\Lp{\infty} \lesssim \Norm{\z_\al}{\half+} \leq \Norm{\z_d}{s+1} \lesssim 1 + \ttanorm \lesssim 1 + \sqrt{\E} < 1 + \sqrt{\mathfrak{e}}.
\end{equation*}
\end{proof}

\begin{lemma}
\label{BasicEnergyEsts}
For $s$ sufficiently large and $(\tta,\y,\w,\be) \in \X$, the following estimates hold:
\begin{align}
\abs{s_{\al t}} &\lesssim  \E + \E^3 + \chi(1+\Vzeronorm)(\sqrt{\E} + \E^\frac{3}{2}), \label{eqn:LtEst}\\
\Norm{\m \cdot \ut}{s} &\lesssim \sqrt{\E} + \E^\frac{9}{2}, \label{eqn:mEst}\\
\norm{V}_\Lp{2} &\lesssim \E + \E^3 + \chi(1+\Vzeronorm)(\sqrt\E+\E^\frac{3}{2}), \label{eqn:VEst}\\
\Norm{\p_\al(V-\Wt\cdot\ut)}{s-1} &\lesssim \sqrt{\E}+\E^\frac{9}{2} + \chi(1+\Vzeronorm)(1+\E^\frac{3}{2}), \label{eqn:VWDerivEst}\\
\abs{\mu} &\lesssim \sqrt{\E}+\E^\frac{9}{2} + \chi(1+\Vzeronorm)(1+\E^2). \label{eqn:muEst}
\end{align}
\noindent The estimate for $\Norm{\m\cdot\un}{s}$ is the same as the estimate given above for $\Norm{\m \cdot \ut}{s}$. Finally, we remark that all of these estimates hold for the regularized and non-regularized terms.
\end{lemma}
\begin{proof}
We have $\abs{L_t} \leq \Norm{\tta}{1}\norm{U}_\Lp{2}$. An application of Lemma \ref{WtEst} yields the desired result.

We recall that $\m$ is composed of two types of terms, a commutator and an integral remainder (see \eqref{eqn:mDef}). Beginning with the commutator, we use Lemma \ref{HTCommutatorEst2} to control the $H^s$ norm:
\begin{equation*}
\Norm{\mathbf{B} \cdot \ut}{s} \lesssim \Norm{\z_\al}{s}^2\Norm{\z_\al^{-2}}{s}\Norm{\y_\al - \frac{\y\z_{\al\al}}{\z_\al}}{s-2}
\end{equation*}
Observing that, $\z_{\al\al} = \p_\al(s_\al e^{i\tta}) = \tta_\al \z_\al$, we use the Sobolev algebra property and Lemma \ref{CompositionEst} to deduce that
\begin{equation}
\label{eqn:B1+B2}
\Norm{\mathbf{B} \cdot \ut}{s} \lesssim \ynorm(1 + \ttanorm)^6.
\end{equation}
On the other hand, we can use Lemma \ref{KopEst} to estimate the $H^s$ norm of $\mathbf{R} \cdot \ut$:
\begin{equation*}
\Norm{\mathbf{R} \cdot \ut}{s} \lesssim \Norm{\z_\al}{s}^2 \left( \Norm{\frac{\y_\al}{\z_\al}}{1} + \Norm{\frac{\y\z_{\al\al}}{\z_\al^2}}{1} \right)(1+\ttanorm)^3.
\end{equation*}
The Sobolev algebra property and the identity $\z_{\al\al} = \tta_\al\z_\al$ imply that
\begin{equation}
\label{eqn:R1+R2}
\Norm{\mathbf{R} \cdot \ut}{s} \lesssim \ynorm(1+\ttanorm)^8.
\end{equation}
Adding \eqref{eqn:B1+B2} and \eqref{eqn:R1+R2} gives the desired estimate for $\Norm{\m \cdot \ut}{s}$.

Moving on, we immediately see that
\begin{equation*}
\norm{V}_\Lp{2} = \norm{\p_\al^{-1}\left( \tta_\al U + s_{\al t} \right)}_\Lp{2} \sim \hnorm{\tta_\al U + s_{\al t}}{-1} \leq \norm{\tta_\al U}_\Lp{2} + \abs{s_{\al t}}.
\end{equation*}
Recalling that $\abs{s_{\al t}} \leq \Norm{\tta}{1}\norm{U}_\Lp{2}$, we deduce from Lemma \ref{SobolevMultiplication} that
\begin{equation*}
\norm{V}_\Lp{2} \lesssim \Norm{\tta}{\sfrac{3}{2}+}\norm{U}_\Lp{2}.
\end{equation*}
From here, Lemma \ref{WtEst} gives the stated estimate for $\norm{V}_\Lp{2}$. Next, recalling equations \eqref{eqn:VWDeriv} and \eqref{eqn:WDeriv}, we have
\begin{align*}
\Norm{\p_\al(V - \Wt \cdot \ut)}{s-1} &\lesssim \abs{s_{\al t}} + \Norm{\HT(\y\tta_\al)}{s-1} + \Norm{\m \cdot \ut}{s-1} + \Norm{\Wo_\al \cdot \ut}{s-1}.
\end{align*}
Lemma \ref{WtEst} allows us to estimate the final term. We can dispose of the Hilbert transform term by applying Lemma \ref{HilbertSobolevNorm} and the Sobolev algebra property. Controlling $\abs{s_{\al t}}$ and $\Norm{\m \cdot \ut}{s-1}$ as in equations \eqref{eqn:LtEst} and \eqref{eqn:mEst} then gives \eqref{eqn:VWDerivEst}.

Now, all that is left is to control $\abs{\mu}$. Just as in \cite{A1}, we can use the chord-arc condition \eqref{eqn:ChordArc} to bound the denominator from below:
\begin{equation}
\label{eqn:muDenomEst}
\abs{is_\al\Int \z_\al \ d\al} \geq \abs{s_\al}\ca \geq \ca.
\end{equation}
The estimate on the first term in the numerator is likewise straightforward:
\begin{equation}
\label{eqn:muNumEst1}
\abs{\Int s_{\al t} \z_\al \ d\al} \leq 2\pi \abs{s_\al}\abs{s_{\al t}}.
\end{equation}
The second term in the numerator will be a bit different. We have
\begin{equation}
\label{eqn:muNumEst2}
\abs{\Int i U_\al \z_\al \ d\al} \leq 2\pi\abs{s_\al}\norm{U_\al}_\Lp{2}.
\end{equation}
We begin by computing $U_\al$:
\begin{align}
\label{eqn:UDeriv}
U_\al &= \BR_\al \cdot \un - \tta_\al \BR \cdot \ut + \Y_\al \cdot \un - \tta_\al \Y \cdot \ut + \Z_\al \cdot \un - \tta_\al \Z \cdot \ut \nonumber\\
&\hspace{0.5cm} + \chi(-\tta_\al\V\cdot\ut + \p_\al(\nab\vphic(\z)) \cdot \un - \tta_\al \nab\vphic(\z) \cdot\ut).
\end{align}
Therefore, applying Lemma \ref{SobolevMultiplication}, we estimate
\begin{align*}
\norm{U_\al}_\Lp{2} &\leq \norm{\BR_\al \cdot \un}_\Lp{2} + \ttanorm\norm{\BR \cdot \ut}_\Lp{2} + \norm{\Y_\al \cdot \un}_\Lp{2} + \ttanorm\norm{\Y \cdot \ut}_\Lp{2} + \norm{\Z_\al \cdot \un}_\Lp{2} + \ttanorm\norm{\Z \cdot \ut}_\Lp{2}\\
&\hspace{0.5cm} +\chi(\Vzeronorm\ttanorm\norm{\ut}_\Lp{2} + \norm{\p_\al(\nab\vphic(\z)) \cdot \un}_\Lp{2} + \ttanorm\norm{\nab\vphic(\z) \cdot \ut}_\Lp{2}).
\end{align*}
We can control the $L^2$ norm of $\ut$ using Lemma \ref{UnitVecEsts}. Then, we can apply \ref{WtEst} and equation \eqref{eqn:WDeriv} yielding
\begin{align*}
\norm{U_\al}_\Lp{2} &\lesssim \norm{\HT(\y_\al)}_\Lp{2} + \norm{\m \cdot \un}_\Lp{2} + \ttanorm\ynorm(1+\ttanorm)^4 + \wnorm(1+\ttanorm)^2\\
&\hspace{0.5cm} + \ttanorm\wnorm(1+\ttanorm) + \benorm(1+\ttanorm)^2 + \ttanorm\benorm(1+\ttanorm)\\
&\hspace{0.5cm} + \chi(\Vzeronorm\ttanorm(1+\ttanorm) + (1+\ttanorm)^2 + \ttanorm(1+\ttanorm)^2).
\end{align*}
Using Lemma \ref{HilbertL2Isometry} as well as the bound on the $H^s$ norm of $\m \cdot \un$ and rearranging a bit gives
\begin{equation}
\label{eqn:UDerivEst}
\norm{U_\al}_\Lp{2} \lesssim (1+\ttanorm)^2\left[\ynorm(1+\ttanorm)^6 + \wnorm + \benorm + \chi(1+\Vzeronorm)(1+\ttanorm)\right].
\end{equation}
At this point, we need only control the final part of the numerator. By writing
\begin{equation}
\label{eqn:z_aFormula}
\z_\al = s_\al e^{i\tta},
\end{equation}
we can rewrite this term and proceed estimating:
\begin{equation}
\abs{\Int V\tta_\al\z_\al \ d\al} \leq \abs{s_\al}\norm{V}_\Lp{2}\pnorm{\tta_\al}{2}.
\end{equation}
Noting our estimate for the $L^2$ norm of $V$ above completes the proof.
\end{proof}

\begin{lemma}
\label{Y1Est}
The $H^{\sHalf}$ norm of $\yb_1$ is controlled by the energy. In particular, we have
\begin{equation}
\label{eqn:Y1Est}
\Norm{\yb_1}{\sHalf} \lesssim \E(1+\sqrt{\E})^{13} + \chi(1+\Vzeronorm)\sqrt{\E}(1+\sqrt{\E})^8.
\end{equation}
\end{lemma}
\begin{proof}
We begin by recalling that
\begin{align*}
\Comp(\yb_1)^* &=  \comm{\Ut}{\HT}\left(\frac{1}{\z_\al}\p_\al\left(\frac{\y}{\z_\al}\right)\right) + \comm{\Ut}{\Kop[\z]}\left(\p_\al\left(\frac{\y}{\z_\al}\right)\right) - \frac{s_{\al t}}{2is_\al} \HT\left( \frac{\y}{\z_\al} \right) \nonumber\\
&\hspace{0.5cm} - \frac{s_{\al t}}{s_\al}\Kop[\z]\y - \frac{1}{4s_\al^2}\left[\HT,\frac{\y}{\z_\al}\right](\HT(\y_\al)) - \frac{i}{2s_\al^2}\Kop[\z](\y\HT(\y_\al)) \nonumber\\
&\hspace{0.5cm} - \frac{1}{2s_\al}\HT\left(\frac{\y \m\cdot \un}{\z_\al}\right) - \frac{i}{s_\al}\Kop[\z](\y \m\cdot\un) - \frac{1}{2s_\al}\left[\HT,\frac{V-\Wt\cdot\ut}{\z_\al}\right](\y\tta_\al) \nonumber\\
&\hspace{0.5cm} - \frac{i}{s_\al}\Kop[\z](\y\tta_\al(V - \Wt \cdot\ut)) - \frac{1}{2s_\al}\HT\left(\frac{\y \Wo_\al \cdot \un}{\z_\al}\right) - \frac{i}{s_\al}\Kop[\z](\y\Wo_\al\cdot\un).
\end{align*}
We will proceed term by term and as such write $\yb_1 = \sum_{j=1}^{12} \yb_{1,j}$. We begin by using Lemma \ref{HTCommutatorEst2} to obtain
\begin{equation*}
\Norm{\left[\HT,\Ut\right]\left(\frac{1}{\z_\al}\p_\al\left(\frac{\y}{\z_\al}\right)\right)}{\sHalf} \lesssim \Norm{\Ut}{\sHalf}\Norm{\frac{1}{\z_\al}\p_\al\left(\frac{\y}{\z_\al}\right)}{s-2}.
\end{equation*}
We observe that
\begin{equation}
\label{eqn:z_at}
\p_\al\z_t = \p_t(s_\al e^{i\tta}) = s_{\al t}e^{i\tta} + is_\al\tta_te^{i\tta} = \frac{s_{\al t}}{s_\al}\z_\al + i\tta_t\z_\al.
\end{equation}
Hence, we estimate
\begin{align*}
\Norm{\Ut}{\sHalf} &\sim \norm{\z_t}_\Lp{2} + \Norm{\p_\al\z_t}{s - \sfrac{3}{2}}\\
&\lesssim \norm{\z_t}_\Lp{2} + \abs{s_{\al t}}\Norm{\z_\al}{s-\sfrac{3}{2}} + \Norm{\tta_t}{s-\sfrac{3}{2}}\Norm{\z_\al}{s-\sfrac{3}{2}}.
\end{align*}
Then, it follows that
\begin{align*}
\Norm{\Ut}{\sHalf} &\lesssim \norm{U}_\Lp{2} + \norm{V}_\Lp{2} + \abs{s_{\al t}}\Norm{\z_\al}{s-\sfrac{3}{2}}\\
&\hspace{0.5cm} + (1+\ttanorm)\Big(\Norm{\HT(\y_\al)}{s-\sfrac{3}{2}} + \Norm{\m \cdot \un}{s-\sfrac{3}{2}} + \Norm{\tta_\al}{s-\sfrac{3}{2}}\Norm{V - \Wt \cdot \ut}{s-\sfrac{3}{2}} + \Norm{\Wo_\al \cdot \un}{s-\sfrac{3}{2}}\Big).
\end{align*}
We can now invoke Lemma \ref{WtEst} and Lemma \ref{BasicEnergyEsts} to conclude that
\begin{equation}
\label{eqn:UtEnergyEst}
\Norm{\z_t}{\sHalf} \lesssim \sqrt{\E}(1+\sqrt{\E})^9 + \chi(1+\Vzeronorm)(1 + \sqrt{\E})^4.
\end{equation}
Then for $\yb_{1,1}$, we have
\begin{equation*}
\Norm{\yb_{1,1}}{\sHalf} \lesssim \Norm{\Ut}{\sHalf}\Norm{\frac{1}{\z_\al}}{s-2}\Norm{\frac{1}{\z_\al}}{s-1}\Norm{\y}{s-1} \lesssim \ynorm(1+\ttanorm)^2\Norm{\Ut}{\sHalf},
\end{equation*}
which implies that
\begin{equation}
\label{eqn:Y11Est}
\Norm{\yb_{1,1}}{\sHalf} \lesssim \E(1+\sqrt{\E})^{11} + \chi(1+\Vzeronorm)\sqrt{\E}(1+\sqrt{\E})^6.
\end{equation}

For $\yb_{1,2}$, we begin by writing
\begin{equation*}
\Norm{\yb_{1,2}}{\sHalf} \lesssim \Norm{\Kop[\z]\left( \Ut\p_\al\left( \frac{\y}{\z_\al} \right) \right)}{\sHalf} + \Norm{\Ut}{\sHalf}\Norm{\Kop[\z]\left( \p_\al\left( \frac{\y}{\z_\al} \right) \right)}{\sHalf}.
\end{equation*}
We can then apply Lemmas \ref{KopEst} and \ref{CompositionEst} along with the Sobolev algebra property to obtain
\begin{equation*}
\Norm{\yb_{1,2}}{\sHalf} \lesssim \ynorm(1+\ttanorm)^4\Norm{\Ut}{\sHalf} + \ynorm(1+\ttanorm)^4\Norm{\Ut}{\sHalf}.
\end{equation*}
It then follows that
\begin{equation}
\label{eqn:Y12Est}
\Norm{\yb_{1,2}}{\sHalf} \lesssim \E(1+\sqrt{\E})^{13} + \chi(1+\Vzeronorm)\sqrt{\E}(1+\sqrt{\E})^8.
\end{equation}

The Sobolev algebra property in conjunction with Lemmas \ref{WtEst}, \ref{BasicEnergyEsts}, \ref{KopEst}, \ref{CompositionEst} imply that
\begin{align}
\Norm{\yb_{1,3}}{\sHalf} &\lesssim \abs{s_{\al t}}\Norm{\HT\left(\frac{\y}{\z_\al}\right)}{\sHalf} \lesssim \E^\frac{3}{2}(1+\sqrt{\E})^5 + \chi(1+\Vzeronorm)\E(1+\sqrt{\E})^3, \label{eqn:Y13Est}\\
\Norm{\yb_{1,4}}{\sHalf} &\lesssim \abs{s_{\al t}}\Norm{\y}{1}(1+\ttanorm)^3 \lesssim \E^\frac{3}{2}(1+\sqrt{\E})^7 + \chi(1+\Vzeronorm)\E(1 + \sqrt{\E})^5, \label{eqn:Y14Est}\\
\Norm{\yb_{1,6}}{\sHalf} &\lesssim \Norm{\y\HT(\y_\al)}{1}(1+\ttanorm)^3 \lesssim \E(1+\sqrt{\E})^3, \label{eqn:Y16Est}\\
\Norm{\yb_{1,8}}{\sHalf} &\lesssim \Norm{\y \m \cdot \un}{1}(1+\ttanorm)^3 \lesssim \ynorm^2(1+\ttanorm) \lesssim \E(1+\sqrt{\E})^{11}, \label{eqn:Y18Est}\\
\Norm{\yb_{1,10}}{\sHalf} &\lesssim \Norm{\y}{1}\Norm{\tta_\al}{1}\Norm{(V - \Wt \cdot \ut)}{1}(1+\ttanorm)^3 \lesssim \E^\frac{3}{2}(1+\sqrt{\E})^{11} + \chi(1+\Vzeronorm)\E(1+\sqrt{\E})^6, \label{eqn:Y110Est}\\
\Norm{\yb_{1,12}}{\sHalf} &\lesssim \Norm{\y}{1}\Norm{\Wo_\al\cdot\un}{1}(1+\ttanorm)^3 \lesssim \E(1+\sqrt{\E})^5 + \chi\sqrt{\E}(1+\sqrt{\E})^5. \label{eqn:Y112Est}
\end{align}
On the other hand, we can use Lemma \ref{HTCommutatorEst2} with Lemmas \ref{CompositionEst}, \ref{WtEst} and \ref{BasicEnergyEsts} to obtain
\begin{align}
\Norm{\yb_{1,5}}{\sHalf} &\lesssim \Norm{\frac{\y}{\z_\al}}{\sHalf}\Norm{\HT(\y_\al)}{s-2} \lesssim \ynorm^2(1+\ttanorm) \lesssim \E + \E^\frac{3}{2}, \label{eqn:Y15Est}\\
\Norm{\yb_{1,9}}{\sHalf} &\lesssim (1+\ttanorm)\ttanorm\ynorm\Norm{V - \Wt \cdot \ut}{\sHalf} \lesssim \E^\frac{3}{2}(1+\sqrt{\E})^9 + \chi(1+\Vzeronorm)\E(1 + \sqrt{\E})^4. \label{eqn:Y19Est}
\end{align}
The final two estimates are rather routine. By Lemmas \ref{WtEst} and \ref{BasicEnergyEsts}, we have
\begin{align}
\Norm{\yb_{1,7}}{\sHalf} &\lesssim \ynorm(1+\ttanorm)\Norm{\m \cdot \un}{\sHalf} \lesssim \E(1+\sqrt{\E})^9, \label{eqn:Y17Est}\\
\Norm{\yb_{1,11}}{\sHalf} &\lesssim \ynorm(1+\ttanorm)\Norm{\Wo_\al \cdot \un}{\sHalf} \lesssim \E(1+\sqrt{\E})^3 + \chi\sqrt{\E}(1 + \sqrt{\E})^3. \label{eqn:Y111Est}
\end{align}

Putting together the estimates \eqref{eqn:Y11Est}-\eqref{eqn:Y111Est}, we deduce that \eqref{eqn:Y1Est} holds.

\end{proof}

\begin{lemma}
\label{m_yLemma}
We have the estimate
\begin{equation}
\Norm{m_\y}{\sHalf} \lesssim \sqrt{\E}(1 + \sqrt{\E})^{17} + \chi(1+\Vzeronorm)(1+\sqrt{\E})^{12}. \label{eqn:m_yEst}
\end{equation}
\end{lemma}
\begin{proof}
We begin by breaking $m_\y$ into smaller parts:
\begin{equation}
\label{eqn:m_yDecomp}
m_\y = m_\y^1 + m_\y^2 + m_\y^3 + m_\y^4,
\end{equation}
where
\begin{align}
m_\y^1 &\coloneqq \frac{\y}{s_\al}\left( s_{\al t} - \Wo_\al \cdot \ut - \m \cdot \ut \right), \ m_\y^2 \coloneqq -2g\eta_\al + 2\J((V - \Wt \cdot \ut)\J(\m \cdot \ut + \Wo_\al \cdot \ut)), \label{eqn:m_y^1+2Def}\\
m_\y^3 &\coloneqq -2s_\al\J([\yb_1 + \Loc_\Y + \Loc_\Z + \chi\p_t(\nab\vphic(\z))] \cdot \ut), \ m_\y^4 \coloneqq - \comm{\HT}{\y}\left( \frac{\y\J\tta_\al}{2s_\al^2} \right). \label{eqn:m_y^3+4Def}\\
\end{align}
Beginning with $m_\y^1$, we have, by Lemma \ref{SobolevMultiplication},
\begin{equation*}
\Norm{m_\y^1}{\sHalf} \lesssim \abs{s_{\al t}}\ynorm + \Norm{\Wo_\al\cdot\ut}{\sHalf}\ynorm + \Norm{\m \cdot \ut}{\sHalf}\ynorm.
\end{equation*}
We can apply Lemma \ref{WtEst} and Lemma \ref{BasicEnergyEsts}:
\begin{equation}
\label{eqn:m_y^1Est}
\Norm{m_\y^1}{\sHalf} \lesssim \E(1+\sqrt{\E})^8 + \chi(1+\Vzeronorm)\sqrt{\E}(1+\sqrt{\E})^3.
\end{equation}

Next, we consider
\begin{equation*}
\Norm{m_\y^2}{\sHalf} \lesssim \Norm{\eta_\al}{\sHalf} + \Norm{\J((V - \Wt \cdot \ut)\J(\m \cdot \ut + \Wo_\al \cdot \ut))}{\sHalf}
\end{equation*}
Using the fact that $\eta_\al = s_\al \sin\tta$ and the Sobolev algebra property, we obtain
\begin{equation*}
\Norm{m_\y^2}{\sHalf} \lesssim \ttanorm + \Norm{V - \Wt \cdot \ut}{s}(\Norm{\m \cdot \ut}{\sHalf} + \Norm{\Wo_\al \cdot \ut}{\sHalf}).
\end{equation*}
It then follows from Lemma \ref{WtEst} and Lemma \ref{BasicEnergyEsts} that
\begin{equation}
\label{eqn:m_y^2Est}
\Norm{m_\y^2}{\sHalf} \lesssim \sqrt{\E}(1+\sqrt{\E})^{17} + \chi(1+\Vzeronorm)(1+\sqrt{\E})^{12}.
\end{equation}

Moving on, we next consider $m_\y^3$:
\begin{equation*}
\Norm{m_\y^3}{\sHalf} \lesssim \Norm{\yb_1}{\sHalf} + \Norm{\Loc_\Y \cdot \ut}{\sHalf} + \Norm{\Loc_\Z \cdot \ut}{_\sHalf} + \chi\Norm{\p_t(\nab\vphic(\z)) \cdot \ut}{_\sHalf}.
\end{equation*}
Lemma \ref{Y1Est} gives control of the first term on the right-hand side. We recall that
\begin{equation*}
(\Loc_\Y \cdot \ut)(\al) = \Rea\bigg\{\frac{\z_\al(\al)}{4\pi s_\al} \Int \p_{\alp}\left( \frac{\w(\alp)s_{1,\al}(\alp)\Ut(\al)}{\z_{1,\al}(\alp)} \right) \cot\frac{1}{2}(\z(\al) - \z_1(\alp)) \ d\alp \bigg\}.
\end{equation*}
We therefore have
\begin{equation*}
\abs{(\Loc_\Y \cdot \ut)(\al)} \lesssim \abs{\z_\al(\al)}\abs{\Ut(\al)}\wnorm\norm{\cot\frac{1}{2}(\z(\al) - \z_1(\cdot))}_\Lp{2}.
\end{equation*}
Hence,
\begin{equation*}
\Norm{\Loc_\Y \cdot \ut}{\sHalf} \lesssim \Norm{\z_\al}{\sHalf} \Norm{\Ut}{\sHalf}\wnorm(1 + \Norm{\z}{\sHalf}) \lesssim \wnorm(1+\ttanorm)^2\Norm{\Ut}{\sHalf}.
\end{equation*}
We can use \eqref{eqn:UtEnergyEst} to obtain
\begin{equation*}
\Norm{\Loc_\Y \cdot \ut}{\sHalf} \lesssim \E(1+\sqrt{\E})^{11} + \chi(1+\Vzeronorm)\sqrt{\E}(1+\sqrt{\E})^6.
\end{equation*}
We can similarly estimate
\begin{equation*}
\Norm{\Loc_\Z \cdot \ut}{\sHalf} \lesssim \E(1+\sqrt{\E})^{11} + \chi(1+\Vzeronorm)\sqrt{\E}(1+\sqrt{\E})^6.
\end{equation*}
Finally, we estimate
\begin{equation*}
\Norm{\p_t(\nab\vphic(\z))}{\sHalf} \lesssim \Norm{\Ut}{\sHalf}(1+\Norm{\z_{0}}{\sHalf}) \lesssim \sqrt{\E}(1+\sqrt{\E})^{10} + \chi(1+\Vzeronorm)(1+\sqrt{\E})^5.
\end{equation*}
We thus conclude that
\begin{equation}
\label{eqn:m_y^3Est}
\Norm{m_\y^3}{\sHalf} \lesssim \sqrt{\E}(1+\sqrt{\E})^{14} + \chi(1+\Vzeronorm)(1+\sqrt{\E})^9.
\end{equation}

For $m_\y^4$, we use Lemma \ref{HTCommutatorEst2} and the Sobolev algebra property to estimate
\begin{equation}
\label{eqn:m_y^4Est}
\Norm{m_\y^4}{\sHalf} \lesssim \ynorm\Norm{\y\J\tta_\al}{s-2} \lesssim \ttanorm\ynorm^2 \lesssim \E^\frac{3}{2}.
\end{equation}
Upon combining estimates \eqref{eqn:m_y^1Est}-\eqref{eqn:m_y^4Est}, it follows that
\begin{equation}
\Norm{m_\y}{\sHalf} \lesssim \sqrt{\E}(1 + \sqrt{\E})^{17} + \chi(1+\Vzeronorm)(1+\sqrt{\E})^{12}.
\end{equation}

\end{proof}

We now arrive at the main energy estimate. Our objective shall be to show that the time derivative of $\E$ is controlled by a suitable polynomial in $\sqrt{\E}$. What will be most important is the lowest order term as this will control the lifespan. We define
\begin{equation}
\label{eqn:PolDef}
\Pol(\E) \coloneqq \E + \E^N + \chi(1 + \Vzeronorm)(\sqrt{\E} + \E^M),
\end{equation}
where $N, M \in 2^{-1}\ZZ$, $N>M$, are taken to be sufficiently large ($M,N \geq 11$ will work).

\begin{thm}
\label{UnifEnergyEst}
For $s$ sufficiently large and for $\Pol(\E)$ given as above, it holds that
\begin{equation*}
\frac{d\E}{dt} \lesssim \Pol(\E).
\end{equation*}
\end{thm}
\begin{proof}
We begin with the $\E^j$'s. We first compute
\begin{equation*}
\frac{d\E_1^j}{dt} = \Int (\p_\al^{j-1}\tta)(\p_\al^{j-1}\tta_t) \ d\al.
\end{equation*}
Substituting the right-hand side of equation \eqref{eqn:ThetaRegEvolutionEqn} for $\tta_t$ above, we write
\begin{align*}
\frac{d\E_1^j}{dt} &= \frac{1}{2s_\al^2}\Int (\p_\al^{j-1}\tta)(\p_\al^{j-1}\HT(\J\y_\al)) \ d\al + \frac{1}{s_\al}\Int (\p_\al^{j-1}\tta)(\p_\al^{j-1}(\m \cdot \un)) \ d\al\\
&\hspace{0.5cm} + \frac{1}{s_\al}\Int (\p_\al^{j-1}\tta)\left(\p_\al^{j-1}\J\left(\left(V - \Wt \cdot \ut\right)\J\tta_\al\right)\right) \ d\al + \frac{1}{s_\al}\Int (\p_\al^{j-1}\tta)(\p_\al^{j-1}(\Wo_\al\cdot\un)) \ d\al\\
&= A_1^j + I + II + III,
\end{align*}
where we have used the fact that $\p_\al \mu = 0$.

In $II$, we want to separate out the term where all of the derivatives land on $\tta_\al$ as it will require more care in analysis. To do this, we rewrite $II$ using the Leibniz rule as follows:
\begin{align*}
II &= \frac{1}{s_\al}\Int (\p_\al^{j-1}\tta)\J\left(\left(V - \Wt \cdot \ut\right)\J\p_\al^j\tta\right) \ d\al + \frac{1}{s_\al}\Int (\p_\al^{j-1}\tta)\left( \sum_{\ell=1}^{j-1} {j-1 \choose \ell} \J\left( \p_\al^\ell\left(V - \Wt \cdot \ut\right)\J\p_\al^{j-\ell}\tta\right) \right) \ d\al\\
&= Z_1^j + R^j_1.
\end{align*}
We have singled out two terms, namely $A_1^j$ and $Z_1^j$. Consideration of $A_1^j$ will be temporarily deferred to exploit some cancellation with terms arising in the sequel, while $Z_1^j$ is a transport term which we will consider in short order. Before examining the transport term, we will estimate terms $I$, $III$ and $R_1^j$.

We begin by considering an arbitrary individual summand from $R^j_1$, which by H\"{o}lder's inequality is bounded above by
\begin{equation*}
\norm{\p_\al^{j-1}\tta}_\Lp{2}\norm{\J\left(\p_\al^\ell\left(V - \Wt \cdot \ut\right)\J\p_\al^{j-\ell}\tta\right)}_\Lp{2}.
\end{equation*}
Clearly, $\norm{\p_\al^{j-1}\tta}_\Lp{2}$ is bounded by the $H^s$ norm of $\tta$ as $j \leq s+1$ and so we focus on bounding the other term. We can use Lemma \ref{JLemma1} to dispense with the outermost instance of $\J$, and then the Sobolev lemma in conjunction with the Sobolev algebra property imply that
\begin{equation*}
\norm{\p_\al^\ell\left(V - \Wt \cdot \ut\right)\J\p_\al^{j-\ell}\tta}_\Lp{2} \lesssim \norm{\p_\al\left(V - \Wt \cdot \ut\right)}_{s-1}\norm{\J\tta}_s
\end{equation*}
as $\ell \leq j-1 \leq s$. Then, another application of Lemmas \ref{JLemma1} and \ref{BasicEnergyEsts} imply that
\begin{equation}
\label{eqn:Ej1IVEst}
R_1^j \lesssim \E(1+\sqrt{\E})^8 + \chi(1+\Vzeronorm)\sqrt{\E}(1+\sqrt{\E})^3 \lesssim \Pol(\E).
\end{equation}

Moving on, we can utilize H\"{o}lder's inequality and Lemma \ref{BasicEnergyEsts} to estimate $I$, while $III$ can be controlled using Lemma \ref{WtEst}:
\begin{equation}
\label{eqn:Ej1I+IIIEst}
I + III \lesssim \Pol(\E).
\end{equation}

We now proceed to consider the transport term $Z_1^j$. If we rewrite $Z_1^j$ exploiting the self-adjointness of $\J$, we can recognize a perfect derivative in the factors of $\tta$ and integrate by parts to obtain
\begin{equation*}
Z_1^j =  -\frac{1}{2s_\al}\Int (\J\p_\al^{j-1}\tta)^2\p_\al\left(V - \Wt \cdot \ut\right) \ d\al.
\end{equation*}
Then, application of Lemmas \ref{SobolevMultiplication}, \ref{JLemma1} and \ref{BasicEnergyEsts} readily give us control of $Z_1^j$:
\begin{equation}
\label{eqn:Ej1Zj1EstE}
Z_1^j \lesssim \E^\frac{3}{2}(1+\sqrt{\E})^8 + \chi(1+\Vzeronorm)\E(1+\sqrt{\E})^3 \lesssim \Pol(\E).
\end{equation}

As noted earlier, we delay estimating $A_1^j$ and so now move on to $\E_2^j$. We begin by computing
\begin{equation*}
\frac{d\E_2^j}{dt} = \frac{1}{4\tau s_\al}\Int (\p_\al^{j-2}\y_t)\La(\p_\al^{j-2}\y) \ d\al - \frac{s_{\al t}}{4\tau s_\al^2} \Int (\p_\al^{j-2}\y)\La(\p_\al^{j-2}\y) \ d\al.
\end{equation*}
As with the estimate for $\frac{d\E_1^j}{dt}$, we substitute the regularized evolution equation \eqref{eqn:yRegEvolutionEqn} for $\y_t$, which yields
\begin{align*}
\frac{d\E_2^j}{dt} &= \frac{1}{2s_\al^2}\Int (\J\p_\al^j\tta)\La(\p_\al^{j-2}\y) \ d\al + \frac{1}{8\tau s_\al^3}\Int (\HT(\y^2\J\p_\al^{j-1}\tta)\La(\p_\al^{j-2}\y) \ d\al\\
&\hspace{0.5cm} + \frac{1}{8\tau s_\al^3}\Int \sum_{\ell=1}^{j-2} {j-2 \choose \ell} \HT(\p_\al^\ell(\y^2)\J\p_\al^{j-\ell-1}\tta)\La(\p_\al^{j-2}\y) \ d\al\\
&\hspace{0.5cm} + \frac{1}{4\tau s_\al^2}\Int \J\left(\left(V - \Wt \cdot\ut\right)\J\p_\al^{j-1}\y\right)\La(\p_\al^{j-2}\y) \ d\al\\
&\hspace{0.5cm} + \frac{1}{4\tau s_\al^2}\Int \sum_{\ell=1}^{j-2} {j-2 \choose \ell} \J\left(\p_\al^\ell\left(V - \Wt \cdot\ut\right)\J\p_\al^{j-\ell-1}\y\right)\La(\p_\al^{j-2}\y) \ d\al\\
&\hspace{0.5cm} - \frac{1}{4\tau s_\al^3}\Int (\p_\al^{j-2}\J(\y\J\y_\al))\La(\p_\al^{j-2}\y) \ d\al + \frac{1}{4\tau s_\al}\Int (\p_\al^{j-2}m_\y)\La(\p_\al^{j-2}\y) \ d\al\\
&\hspace{0.5cm} - \frac{s_{\al t}}{4\tau s_\al^2} \Int (\p_\al^{j-2}\y)\La(\p_\al^{j-2}\y) \ d\al\\
&= A_2^j + S_1^j + I + Z_2^j + II + III + IV + V.
\end{align*}
First, we shall exploit the primary cancellation which we mentioned earlier. In particular, recalling that $\La \coloneqq H\p_\al$, we consider 
\begin{align*}
A_1^j + A_2^j &= \frac{1}{2s_\al^2}\Int(\p_\al^{j-1}\tta)\HT(\J\p_\al^j\y) \ d\al + \frac{1}{2s_\al^2}\Int (\J\p_\al^j\tta)\HT(\p_\al^{j-1}\y) \ d\al.
\end{align*}
Noting that $\J$ is a self-adjoint operator which commutes with spatial differentiation and integrating by parts in the second integral, we obtain
\begin{equation}
\label{eqn:EnergyCancel1}
A_1^j + A_2^j =  \frac{1}{2s_\al^2}\Int(\p_\al^{j-1}\tta)\HT(\J\p_\al^j\y) \ d\al - \frac{1}{2s_\al^2}\Int (\p_\al^{j-1}\tta)\HT(\J\p_\al^j\y) \ d\al = 0.
\end{equation}

Much like the $A$'s, consideration of $S_1^j$ will be delayed to exploit some secondary cancellation. We will first estimate $I-V$ and then consider the second transport term $Z_2^j$. In estimating these terms, we shall repeatedly encounter terms of the form $\int (\p_\al^j f)\La(\p_\al^\ell g) \ d\al$. As such, it will be of use to obtain a preliminary estimate for such terms. By applying Plancherel's theorem and recalling that $\La$ is a Fourier multiplier, we can write
\begin{equation*}
\Int (\p_\al^j f)\La(\p_\al^\ell g) \ d\al = \sum_{k \in \ZZ} \FT(\p_\al^j f)\abs{k}\FT(\p_\al^\ell g) = \sum_{k \in \ZZ} \abs{k}^\frac{1}{2}\FT(\p_\al^j f) \cdot \abs{k}^\frac{1}{2}\FT(\p_\al^\ell g).
\end{equation*}
This immediately implies the estimate
\begin{equation}
\label{eqn:fLa(g)Est}
\Int (\p_\al^k f)\La(\p_\al^\ell g) \ d\al \lesssim \Norm{\p_\al^j f}{\half}\Norm{\p_\al^\ell g}{\half} \leq \Norm{f}{j + \half}\Norm{g}{\ell + \half}.
\end{equation}

Utilizing the estimate \eqref{eqn:fLa(g)Est}, it is straightforward to estimate
\begin{equation}
\label{eqn:Ej2I+II+IV+VEst}
I + II + IV + V \lesssim \Pol(\E).
\end{equation}
For $III$, we want to first use the Leibniz rule to isolate the term where all of the derivatives land on $\y_\al$:
\begin{align}
III &= -\frac{1}{4\tau s_\al^3}\Int \J(\y(\p_\al^{j-2}\J\y_\al))\La(\p_\al^{j-2}\y) \ d\al \nonumber\\
&\hspace{0.5cm} -\frac{1}{4\tau s_\al^3} \sum_{\ell = 1}^{j-2} {j-2 \choose \ell} \Int \J((\p_\al^\ell \y)\J(\p_\al^{j-2-\ell}\y_\al))\La(\p_\al^{j-2}\y) \ d\al \nonumber\\
&= Z_3^j + R_2^j.
\end{align}
$Z_3^j$ is a transport term and we shall consider it alongside the other transport term $Z_2^j$ as we treat them in very similar ways. For $R_2^j$, we begin by applying \eqref{eqn:fLa(g)Est} and Lemma \ref{JLemma1}, to eliminate the outermost instance of $\J$, to an arbitrary summand:
\begin{equation*}
\Int \J((\p_\al^\ell\y)\J(\p_\al^{j-1-\ell}\y))\La(\p_\al^{j-2}\y) \ d\al \lesssim \Norm{(\p_\al^\ell\y)\J(\p_\al^{j-1-\ell}\y)}{\half}\Norm{\p_\al^{j-2}\y}{\half}.
\end{equation*}
We want to apply Lemma \ref{SobolevMultiplication}, but we will need to be careful about which factor we place in the higher regularity space. First, recall that $1 \leq \ell \leq j-2 \leq s-1$. If $\ell = j-2$, then $j-1-\ell = 1$ and, upon applying Lemma \ref{JLemma1} again, we have the estimate
\begin{equation*}
\Norm{(\p_\al^\ell\y)\J(\p_\al^{j-1-\ell}\y)}{\half}\Norm{\p_\al^{j-2}\y}{\half} \lesssim \Norm{\p_\al^{j-2}\y}{\half}\Norm{\p_\al \y}{\half+}\ynorm \lesssim \ynorm^3.
\end{equation*}
On the other hand, if $\ell \leq j-3$, we can put $\p_\al^\ell \y$ in the higher regularity space (again we apply Lemma \ref{JLemma1} twice):
\begin{equation*}
\Norm{(\p_\al^\ell\y)\J(\p_\al^{j-1-\ell}\y)}{\half}\Norm{\p_\al^{j-2}\y}{\half} \lesssim \Norm{\p_\al^\ell \y}{\half+}\Norm{\p_\al^{j-1-\ell}\y}{\half}\ynorm \lesssim \ynorm^3,
\end{equation*}
where we used the fact that $j-1-\ell \leq j-2 \leq s-1$. In either case, we have the estimate
\begin{equation}
\label{eqn:Ej2VIEst}
R_2^j \lesssim \ynorm^3 \lesssim \Pol(\E).
\end{equation}

We now arrive at the $Z^j$ transport terms. We begin by considering a general integral of the form $\int g f_\al \La(f) \ d\al$. Following the procedure outlined in \cite{A1}, we can write this as
\begin{equation}
\label{eqn:TransportID}
\Int g f_\al \La(f) \ d\al = \frac{1}{2}\Int \p_\al [\HT,g](f_\al) f \ d\al.
\end{equation}
Then, Lemma \ref{HTCommutatorEst1} implies that
\begin{equation}
\label{eqn:TransportEst}
\Int g f_\al \La(f) \ d\al \leq \Norm{[\HT,g](f_\al)}{1} \norm{f}_\Lp{2} \lesssim \Norm{f_\al}{-1}\Norm{g}{3}\norm{f}_\Lp{2} \lesssim \norm{f}_\Lp{2}^2\Norm{g}{3}.
\end{equation}
After exploiting the symmetry of $\J$, $Z_2^j$ is of this form and so we have:
\begin{equation*}
\Int (V - \Wt \cdot \ut)(\J\p_\al^{j-2}\y)_\al\La(\J\p_\al^{j-2}\y) \ d\al \lesssim \norm{\p_\al^{j-2}\y}_\Lp{2}^2\Norm{V - \Wt \cdot \ut}{3}.
\end{equation*}
Then, Lemmas \ref{WtEst} and \ref{BasicEnergyEsts} give
\begin{equation}
\label{eqn:Zj2Est}
Z_2^j \lesssim \E^\frac{3}{2}(1+\sqrt{\E})^8 + \chi(1+\Vzeronorm)\E(1+\sqrt{\E})^3 \lesssim \Pol(\E).
\end{equation}
Next, we consider $Z_3^j$, after rewriting by exploiting the self-adjointness of $\J$. We again apply the estimate of equation \eqref{eqn:TransportEst} in conjunction with the fact that $\J$ commutes with $\p_\al$, as well as $\HT$, and Lemma \ref{JLemma1} to obtain
\begin{equation}
\label{eqn:Zj3Est}
Z_3^j = \Int \y(\J\p_\al^{j-2}\y)_\al\La(\J\p_\al^{j-2}\y) \ d\al \lesssim \norm{\p_\al^{j-2}\y}_\Lp{2}^2\Norm{\y}{3} \lesssim \ynorm^3 \lesssim \Pol(\E).
\end{equation}.

We continue and now compute
\begin{align*}
\frac{d\E_3^j}{dt} &= -\frac{s_{\al t}}{8\tau^2s_\al^3}\Int \y^2(\p_\al^{j-2}\y)^2 \ d\al + \frac{1}{16\tau^2s_\al^2}\Int \y\y_t(\p_\al^{j-2}\y)^2 \ d\al + \frac{1}{16\tau^2s_\al^2}\Int \y^2(\p_\al^{j-2}\y)(\p_\al^{j-2}\y_t) \ d\al\\
&= I + II + III.
\end{align*}
It is obvious that
\begin{equation}
\label{eqn:Ej3I+IIEst}
I + II \lesssim \Pol(\E).
\end{equation}

To estimate $III$, we substitute in the right-hand side of the evolution equation for $\y$:
\begin{align*}
III &= \frac{1}{8\tau s_\al^3} \Int \y^2(\p_\al^{j-2}\y)(\p_\al^{j-2}\J\tta_{\al\al}) \ d\al + \frac{1}{32\tau^2s_\al^4} \Int \y^2(\p_\al^{j-2}\y)(\p_\al^{j-2}(\HT(\y^2\J\tta_\al)) \ d\al\\
&\hspace{0.5cm} + \frac{1}{16\tau^2s_\al^3} \Int \y^2(\p_\al^{j-2}\y)(\p_\al^{j-2}\J((V-\Wt\cdot\ut)\J\y_\al)) \ d\al\\
&\hspace{0.5cm} - \frac{1}{16\tau^2s_\al^4}\Int \y^2(\p_\al^{j-2}\y)\p_\al^{j-2}\J(\y\J\y_\al) \ d\al + \frac{1}{16\tau^2s_\al^2} \Int \y^2(\p_\al^{j-2}\y)(\p_\al^{j-2}m_\y) \ d\al\\
&= S_2^j + C_1^j + C_2^j + C_3^j + C_4^j.
\end{align*}
We first examine the sum of $S_1^j$ and $S_2^j$:
\begin{equation*}
S_1^j + S_2^j = \frac{1}{8\tau s_\al^3}\Int \HT(\y^2\J\p_\al^{j-1}\tta)\La(\p_\al^{j-2}\y) \ d\al  + \frac{1}{8\tau s_\al^3} \Int \y^2(\p_\al^{j-2}\y)(\p_\al^{j-2}\J\tta_{\al\al}) \ d\al.
\end{equation*}
We exploit the fact that $\La$ is self-adjoint and that $\La\HT = -\p_\al$ to rewrite this as
\begin{equation*}
\frac{1}{8\tau s_\al^3}\Int -\p_\al(\y^2\J\p_\al^{j-1}\tta)(\p_\al^{j-2}\y) \ d\al + \frac{1}{8\tau s_\al^3}\Int \y^2(\p_\al^{j-2}\y)(\J\p_\al^j\tta) \ d\al.
\end{equation*}
When we expand the derivative in the first integral, we obtain the cancellation (when the derivative lands on $\J\p^{j-1}\tta$) and are left with
\begin{equation}
\label{eqn:SjSum}
S_1^j + S_2^j = -\frac{1}{4\tau s_\al^3}\Int \y\y_\al(\J\p_\al^{j-1}\tta)(\p_\al^{j-2}\y) \ d\al.
\end{equation}
We can then use H\"older's inequality and Lemma \ref{SobolevMultiplication} to obtain
\begin{equation}
\label{eqn:SjEst}
S_1^j + S_2^j \lesssim \norm{\y\y_\al(\p_\al^{j-1}\tta)}_\Lp{2}\norm{\p_\al^{j-2}\y}_\Lp{2} \lesssim \ttanorm\ynorm^3 \lesssim \E^2 \lesssim \Pol(\E).
\end{equation}

There are no surprises in the $C^j$'s; we have
\begin{equation}
\label{eqn:Ej3CjEst}
C_1^j + C_2^j + C_3^j + C_4^j \lesssim \Pol(\E).
\end{equation}
Collecting these estimates, we now deduce that
\begin{align}
\label{eqn:EjEnergyEst}
\frac{d\E^j}{dt} \lesssim \Pol(\E).
\end{align}

We now proceed to examine $\E^1$ and begin by computing
\begin{align}
\frac{d\E^1}{dt} &= \frac{d}{dt}\Bigg\{ \frac{1}{2}\Int (\p_\al\w)^2 \ d\al + \frac{1}{2}\Int (\p_\al\be)^2 \ d\al \Bigg\} \nonumber\\
&= \Int (\p_\al\w)(\p_\al\w_t) \ d\al + \Int (\p_\al\be)(\p_\al\be_t) \ d\al \nonumber\\
&= I + II. \label{eqn:E1TimeDeriv}
\end{align}
Via H\"{o}lder's inequality, we have $I \leq \wnorm\norm{\p_\al\w_t}_\Lp{2}$. Given that
\begin{equation*}
\p_\al\w_t(\al) = -\frac{1}{\pi}\Int \y(\alp) \p_\al k_{\FS,t}^1(\al,\alp) \ d\alp,
\end{equation*}
we compute
\begin{align*}
\p_\al k_{\FS,t}^1(\al,\alp) &= \Rea\Bigg\{ \frac{ s_{1,\al\al}(\al)\z_t(\alp) }{2 s_{1,\al}^2(\al) } \p_\al \cot\frac{1}{2}(\z_1(\al) - \z(\alp))\Bigg\}\\
&\hspace{0.5cm} - \Rea\Bigg\{\frac{\z_t(\alp)}{2s_{1,\al}(\al)}\p_\al^2\cot\frac{1}{2}(\z_1(\al) - \z(\alp)) \Bigg\}.
\end{align*}
Notice that $s_{1,\al\al} = \p_\al\abs{\z_{1,\al}} = \frac{1}{2s_{1,\al}}(\z_{1,\al\al}\z_{1,\al}^* + \z_{1,\al}\z_{1,\al\al}^*)$ and so $\abs{s_{1,\al\al}} \lesssim 1$. Again applying H\"{o}lder's inequality, we deduce that
\begin{equation*}
\abs{\p_\al\w_t(\al)} \lesssim \norm{\y}_\Lp{2}\norm{\p_\al k_{\FS,t}^1(\al,\cdot)}_\Lp{2},
\end{equation*}
so the only task at hand is to control the $L^2$ norm of the derivative of $k_{\FS,t}^1$. From the above computation, we use Lemma \ref{SobolevMultiplication} to estimate
\begin{equation*}
\norm{\p_\al k_{\FS,t}^1(\al,\cdot)}_\Lp{2} \lesssim \norm{\z_t}_\Lp{2}\left( \Norm{\cot\frac{1}{2}(\z_1(\al) - \z(\cdot))}{\sfrac{3}{2}+} + \Norm{\cot\frac{1}{2}(\z_1(\al) - \z(\cdot))}{\sfrac{5}{2}+} \right).
\end{equation*}
Lemma \ref{CompositionEst} and \eqref{eqn:UtEnergyEst} then imply that
\begin{equation*}
\norm{\p_\al\w_t}_\Lp{2} \lesssim \E(1+\sqrt{\E})^{12} + \chi(1+\Vzeronorm)\sqrt{\E}(1+\sqrt{\E})^7.
\end{equation*}
This implies that we have the following estimate for $I$:
\begin{equation}
\label{eqn:E1Est1}
I \lesssim \E^\frac{3}{2}(1+\sqrt{\E})^{12} + \chi(1+\Vzeronorm)\E(1+\sqrt{\E})^7 \lesssim \Pol(\E).
\end{equation}

For the second term, we may once more apply H\"{o}lder's inequality to obtain $II \leq \benorm\norm{\p_\al\be_t}_\Lp{2}$. The estimate for $\norm{\p_\al\be_t}_\Lp{2}$ is very similar to the estimate for $\norm{\p_\al\w_t}_\Lp{2}$. We omit the calculations, but note that we have
\begin{equation}
\label{eqn:E1Est2}
II \lesssim \E^\frac{3}{2}(1+\sqrt{\E})^{12} + \chi(1+\Vzeronorm)\E(1+\sqrt{\E})^7 \lesssim \Pol(\E).
\end{equation}

Putting together equations \eqref{eqn:E1Est1} and \eqref{eqn:E1Est2}, we have the following estimate in terms of the energy for the time derivative of $\E^1$:
\begin{equation}
\label{eqn:E1EnergyEst}
\frac{d\E^1}{dt} \lesssim \E^\frac{3}{2}(1+\sqrt{\E})^{12} + \chi(1+\Vzeronorm)\E(1+\sqrt{\E})^7 \lesssim \Pol(\E).
\end{equation}

We can similarly estimate
\begin{equation}
\label{eqn:E0EnergyEst}
\frac{d\E^0}{dt} \lesssim \Pol(\E).
\end{equation}
At last, upon combining \eqref{eqn:E0EnergyEst}, \eqref{eqn:E1EnergyEst} and \eqref{eqn:EjEnergyEst}, we have now shown that
\begin{align}
\label{eqn:MainEnergyEst}
\frac{d\E}{dt} \lesssim \Pol(\E).
\end{align}

\end{proof}

\section{Existence of Solutions}

We continue in this section to carry out the plan sketched earlier for obtaining solutions to the non-regularized system. Having established the uniform energy estimate in the previous section, our next goal will be to show that solutions to the regularized system exist, at least for a short time.

\begin{thm}
\label{RegExist}
Given initial data $\Tta_0 \in \X$, there exists a unique solution $\Tta^\de \in \X$ which solves the regularized system \eqref{eqn:FullRegSystem}. Further, there exists a time $T^\de > 0$ such that $\Tta^\de \in C^1([0,T^\de]; \X)$. A priori, $T^\de$ may depend upon the regularization parameter $\de$. In addition, $T^\de$ may depend on $\ee$, $\Vzeronorm$, $s$ and $\X$. Notice that the solution belonging to $\X$ implies that the chord-arc condition \eqref{eqn:ChordArc} and the uniform energy bound \eqref{eqn:EBound} are satisfied on $[0,T^\de]$.
\end{thm}

\begin{rmk}
\label{TdRemark}
Though the existence time $T^\de$ obtained from Theorem \ref{RegExist} is allowed to depend upon $\de$, we will prove a result in the sequel showing that there is a uniform (independent of $\de$) time interval $[0,T]$ on which solutions to the regularized system exist for any $\de>0$. This existence time $T$ will, of course, still depend on $\ee$, $\Vzeronorm$, $s$ and $\X$.
\end{rmk}

\begin{proof}[Proof of Theorem \ref{RegExist}]

We define $\F^\de: \R^4 \to \R^4$, $\F^\de = (\F_1^\de,\F_2^\de,\F_3^\de,\F_4^\de)$, by letting $\F_1^\de$ denote the right-hand side of \eqref{eqn:ThetaRegEvolutionEqn}, $\F_2^\de$ the right-hand side of \eqref{eqn:yRegEvolutionEqn}, $\F_3^\de$ the right-hand side of \eqref{eqn:wRegEvolutionEqn} and $\F_4^\de$ the right-hand side of \eqref{eqn:beRegEvolutionEqn}. We shall use the Picard theorem to establish the existence of solutions to the regularized equations. As such, we wish to show that $\F$ satisfies a particular Lipschitz bound on $\X$. In particular, given $\Tta, \Tta^\pr \in \X$, we claim that
\begin{equation}
\label{eqn:LipschitzBound}
\norm{\F^\de(\Tta) - \F^\de(\Tta^\pr)}_X \lesssim_\de \norm{\Tta - \Tta^\pr}_X.
\end{equation}
Notice that in \eqref{eqn:LipschitzBound} the implied constant can depend on the regularization parameter $\de$. This dependence will generally be in the form of negative powers of $\de$ (see Lemma \ref{JLemma1}). We use the triangle inequality to break the left-hand side of \eqref{eqn:LipschitzBound} up into smaller pieces.

We begin with $\F_1^\de = \F_{1,1}^\de + \F_{1,2}^\de + \F_{1,3}^\de + \F_{1,4}^\de + \F_{1,5}^\de$. From Theorem 5.1 of \cite{A1}, we have
\begin{equation}
\label{eqn:F1Est1}
\Norm{\F_{1,1}^\de(\Tta) - \F_{1,1}^\de(\Tta^\pr)}{s} \lesssim_\de \norm{\Tta - \Tta^\pr}_X.
\end{equation}
By applying Lemma \ref{JLemma1}, adding and subtracting, and utilizing the Sobolev algebra property, we can bound the $\F_{1,2}^\de$ difference by
\begin{equation*}
C\left(\Norm{(V-\Wt\cdot\ut)-(V^\pr-\Wt^\pr\cdot\ut^\pr)}{s}\Norm{\J\tta_\al}{s} + \Norm{V^\pr - \Wt^\pr\cdot\ut^\pr}{s}\Norm{\J(\tta_\al - \tta_\al^\pr)}{s}\right).
\end{equation*}
The second term is straightforward; in particular, we apply Lemma \ref{JLemma1} and the uniform energy estimates:
\begin{equation}
\label{eqn:F1Est2,1}
\Norm{V^\pr - \Wt^\pr\cdot\ut^\pr}{s}\Norm{\J(\tta_\al - \tta_\al^\pr)}{s} \lesssim \de^{-1}\Norm{\tta-\tta^\pr}{s} \lesssim_\de \norm{\Tta - \Tta^\pr}_X.
\end{equation}
We can use the energy estimates to easily bound the first term by a constant multiple of
\begin{equation*}
\norm{(V - \Wt\cdot\ut) - (V^\pr-\Wt^\pr\cdot\ut^\pr)}_\Lp{2} + \Norm{\p_\al(V - \Wt\cdot\ut) - \p_\al(V^\pr-\Wt^\pr\cdot\ut^\pr)}{s-1}.
\end{equation*}
For the first piece, we must estimate $\norm{V-V^\pr}_\Lp{2}$ and $\norm{\Wt\cdot\ut - \Wt^\pr\cdot\ut^\pr}_\Lp{2}$. First, it is straightforward to see that
\begin{equation*}
\norm{V - V^\pr}_\Lp{2} \lesssim \norm{U}_\Lp{2}\Norm{\tta_\al - \tta_\al^\pr}{\half+} + \Norm{\tta_\al^\pr}{\half+}\norm{U-U^\pr}_\Lp{2}.
\end{equation*}
It is clear that the first term is controlled by $C\Norm{\tta - \tta^\pr}{s} \lesssim \norm{\Tta - \Tta^\pr}_X$. Hence, we need only control $\norm{U - U^\pr}_\Lp{2}$ by a constant multiple of $\norm{\Tta - \Tta^\pr}_X$.

By definition, we have
\begin{align*}
\norm{U - U^\pr}_\Lp{2} &\leq \norm{\BR\cdot\un - \BR^\pr\cdot\un^\pr}_\Lp{2} + \norm{\Y \cdot\un - \Y^\pr\cdot\un^\pr}_\Lp{2} + \norm{\Z\cdot\un - \Z^\pr\cdot\un^\pr}_\Lp{2}\\
&\hspace{0.5cm} + \chi(\Vzeronorm\norm{n_1 - n_1^\pr}_\Lp{2} + \norm{\nab\vphic(\z)\cdot\un - \nab\vphic(\z^\pr)\cdot\un^\pr}_\Lp{2}).
\end{align*}
That $\norm{\BR\cdot\un - \BR^\pr\cdot\un^\pr}_\Lp{2} \lesssim_\de \norm{\Tta - \Tta^\pr}_X$ follows from Theorem 5.1 of \cite{A1}. Observe that, by adding and subtracting, we have for the second term
\begin{equation*}
\norm{\Y\cdot\un - \Y^\pr\cdot\un^\pr}_\Lp{2} \leq \norm{(\Y - \Y^\pr)\cdot\un}_\Lp{2} + \norm{\Y^\pr\cdot(\un - \un^\pr)}_\Lp{2}.
\end{equation*}
The second term is easily bounded:
\begin{equation}
\label{eqn:YDiffEst1}
\norm{Y^\pr\cdot(\un-\un^\pr)}_\Lp{2} \lesssim \norm{\z - \z^\pr}_\Lp{2} \lesssim \norm{\tta - \tta^\pr}_\Lp{2} \leq \norm{\Tta - \Tta^\pr}_X.
\end{equation}
For the first term, we begin by considering
\begin{align}
&\abs{(\Comp(\Y)^*(\al) - \Comp(\Y^\pr)^*(\al))\Comp(\un)(\al)} \nonumber\\
&=\frac{\abs{i\z_\al(\al)}}{4\pi s_\al}\abs{\int_0^{2\pi} \w(\alp)s_{1,\al}(\alp)\cot\frac{1}{2}(\z(\al) - \z_1(\alp)) \ d\alp \right. \nonumber\\
&\left.\hspace{2cm} - \int_0^{2\pi} \w^\pr(\alp)s_{1,\al}(\alp)\cot\frac{1}{2}(\z^\pr(\al) - \z_1(\alp)) \ d\alp}. \label{eqn:YDiff1}
\end{align}
This is bounded above by a constant multiple of
\begin{equation}
\label{eqn:YDiff2}
\int_0^{2\pi} \abs{\w(\alp) - \w^\pr(\alp)} \ d\alp + \int_0^{2\pi} \abs{\cot\frac{1}{2}(\z(\al) - \z_1(\alp)) - \cot\frac{1}{2}(\z^\pr(\al) - \z_1(\alp))} \ d\alp.
\end{equation}
By H\"older's inequality,
\begin{equation}
\label{eqn:YDiff3}
\int_0^{2\pi} \abs{\w(\alp) - \w^\pr(\alp)} \ d\alp \lesssim \norm{\w - \w^\pr}_\Lp{2}.
\end{equation}
On the other hand,
\begin{equation}
\label{eqn:YDiff4}
\int_0^{2\pi} \abs{\cot\frac{1}{2}(\z(\al) - \z_1(\alp)) - \cot\frac{1}{2}(\z^\pr(\al) - \z_1(\alp))} \ d\alp \lesssim_\h \abs{\z(\al) - \z^\pr(\al)},
\end{equation}
given that $\abs{\z - \z_1}$ is bounded away from zero - in fact, recall that we require $\eta - \eta_1 \geq \h > 0$ - and thus the map $\z \mapsto \cot\frac{1}{2}(\z - \z_1)$ is Lipschitz continuous with the Lipschitz constant depending upon the water depth $\h$. It then follows that
\begin{equation}
\label{eqn:YDiffEst2}
\norm{(\Y - \Y^\pr) \cdot \un}_\Lp{2} \lesssim \norm{\w - \w^\pr}_\Lp{2} + \norm{\z - \z^\pr}_\Lp{2} \lesssim \norm{\tta - \tta^\pr}_\Lp{2} + \norm{\w - \w^\pr}_\Lp{2} \leq \norm{\Tta - \Tta^\pr}_X.
\end{equation}
Therefore, from \eqref{eqn:YDiffEst1} and \eqref{eqn:YDiffEst2}, we conclude that
\begin{equation}
\label{eqn:YDiffEst}
\norm{\Y\cdot\un - \Y^\pr\cdot\un^\pr}_\Lp{2} \lesssim \norm{\Tta - \Tta^\pr}_X.
\end{equation}
The estimate for the third terms is entirely analogous:
\begin{equation}
\label{eqn:ZDiffEst}
\norm{\Z\cdot\un - \Z^\pr\cdot\un^\pr}_\Lp{2} \lesssim \norm{\Tta - \Tta^\pr}_X.
\end{equation}
The remaining terms contain no surprises and upon carrying out these computations we obtain
\begin{equation}
\label{eqn:UDiffEst}
\norm{U - U^\pr}_\Lp{2} \lesssim_\de \norm{\Tta - \Tta^\pr}_X.
\end{equation}
From here, we deduce that
\begin{equation*}
\norm{V - V^\pr}_\Lp{2} \lesssim_\de \norm{\Tta - \Tta^\pr}_X \text{ and } \abs{s_{\al t} - s_{\al t}^\pr} \lesssim_\de \norm{\Tta - \Tta^\pr}_X.
\end{equation*}

Next, we have
\begin{align*}
\norm{\Wt\cdot\ut - \Wt^\pr\cdot\ut^\pr}_\Lp{2} &\leq \norm{\BR \cdot \ut - \BR^\pr \cdot \ut^\pr}_\Lp{2} + \norm{\Y \cdot \ut - \Y^\pr \cdot \ut^\pr}_\Lp{2} + \norm{\Z \cdot \ut - \Z^\pr \cdot \ut^\pr}_\Lp{2}\\
&\hspace{0.5cm} + \chi(\Vzeronorm\norm{t_1 - t_1^\pr}_\Lp{2} + \norm{\nab\vphic(\z)\cdot\ut - \nab\vphic(\z^\pr)\cdot\ut^\pr}_\Lp{2}).
\end{align*}
It is easily observable that this will satisfy the same estimate as $\norm{U - U^\pr}_\Lp{2}$ and thus
\begin{equation}
\label{eqn:WtDiffEst}
\norm{\Wt\cdot\ut - \Wt^\pr\cdot\ut^\pr}_\Lp{2} \lesssim_\de \norm{\Tta - \Tta^\pr}_X.
\end{equation}
It therefore follows that
\begin{equation*}
\norm{(V - \Wt\cdot\ut) - (V^\pr - \Wt^\pr\cdot\ut^\pr)}_\Lp{2} \lesssim_\de \norm{\Tta - \Tta^\pr}_X.
\end{equation*}

Continuing to estimate term-by-term as we have been leads us to conclude that
\begin{equation}
\label{eqn:F12DiffEst}
\Norm{\F_{1,2}^\de(\Tta) - \F_{1,2}^\de(\Tta^\pr)}{s} \lesssim_\de \norm{\Tta - \Tta^\pr}_X.
\end{equation}
Proceeding in this fashion, we arrive at the estimate
\begin{equation}
\label{eqn:F1DiffEst}
\Norm{\F_1^\de(\Tta) - \F_1^\de(\Tta^\pr)}{s} \lesssim \norm{\Tta - \Tta^\pr}_X.
\end{equation}

Moving on to $\F_2^\de$, Theorem 5.1 of \cite{A1} implies that
\begin{align}
\Norm{\F_{2,1}^\de(\Tta) - \F_{2,1}^\de(\Tta^\pr)}{\sHalf} &\lesssim_\de \norm{\Tta - \Tta^\pr}_X, \label{eqn:F21DiffEst}\\
\Norm{\F_{2,2}^\de(\Tta) - \F_{2,2}^\de(\Tta^\pr)}{\sHalf} &\lesssim_\de \norm{\Tta - \Tta^\pr}_X. \label{eqn:F22DiffEst}
\end{align}
Further, using the above estimates derived in estimating $\F_1^\de$, it is easy to obtain the bounds
\begin{align}
\Norm{\F_{2,3}^\de(\Tta) - \F_{2,3}^\de(\Tta^\pr)}{\sHalf} &\lesssim_\de \norm{\Tta - \Tta^\pr}_X, \label{eqn:F23DiffEst}\\
\Norm{\F_{2,4}^\de(\Tta) - \F_{2,4}^\de(\Tta^\pr)}{\sHalf} &\lesssim_\de \norm{\Tta - \Tta^\pr}_X. \label{eqn:F24DiffEst}
\end{align}

For $\F_{2,5}^\de$, we shall utilize the decomposition of $m_\y = m_\y^1 + m_\y^2 + m_\y^3 + m_\y^4$ from Lemma \ref{m_yLemma}. The following estimates are rather simple:
\begin{align}
\Norm{m_\y^1 - (m_\y^1)^\pr}{\sHalf} &\lesssim_\de \norm{\Tta - \Tta^\pr}_X, \label{eqn:my1DiffEst}\\
\Norm{m_\y^2 - (m_\y^2)^\pr}{\sHalf} &\lesssim_\de \norm{\Tta - \Tta^\pr}_X.
\end{align}

For $m_\y^3$, we have
\begin{align*}
\Norm{m_\y^3 - (m_\y^3)^\pr}{\sHalf} &\lesssim \Norm{\J(\yb_1\cdot\ut - \yb_1^\pr\cdot\ut^\pr)}{\sHalf} + \Norm{\J(\Loc_\Y\cdot\ut - \Loc_\Y^\pr\cdot\ut^\pr)}{\sHalf}\\
&\hspace{0.5cm} + \Norm{\J(\Loc_\Z\cdot\ut - \Loc_\Z^\pr\cdot\ut^\pr)}{\sHalf}\\
&\hspace{0.5cm} + \Norm{\chi\J(\p_t(\nab\vphic(\z))\cdot\ut - \p_t(\nab\vphic(\z^\pr))\cdot\ut^\pr)}{\sHalf}.
\end{align*}
For the first term, we add and subtract, and use Lemma \ref{JLemma1}, to obtain the bound
\begin{equation*}
\Norm{\J(\yb_1\cdot\ut - \yb_1^\pr\cdot\ut^\pr)}{\sHalf} \lesssim_\de \norm{(\yb_1 - \yb_1^\pr)\cdot\ut}_\Lp{2} + \norm{\yb_1^\pr\cdot(\ut - \ut^\pr)}_\Lp{2}.
\end{equation*}
Given that $\yb_1$ is bounded in $L^2$, the second term is easily bounded by $C\norm{\Tta - \Tta^\pr}_X$, as we have seen many times before. For the first term, we shall begin by writing $\yb_1 = \sum \yb_{1,j}$. Beginning with $\yb_{1,1}$, we have
\begin{equation*}
\norm{(\yb_{1,1} - \yb_{1,1}^\pr)\cdot\ut}_\Lp{2} \lesssim \norm{\comm{\HT}{\Ut}\left(\frac{1}{\z_\al}\p_\al\left(\frac{\y}{\z_\al}\right)\right) - \comm{\HT}{\Ut^\pr}\left(\frac{1}{\z_\al^\pr}\p_\al\left(\frac{\y^\pr}{\z_\al^\pr}\right)\right)}_\Lp{2}.
\end{equation*}
We now add and subtract:
\begin{align*}
\norm{(\yb_{1,1} - \yb_{1,1}^\pr)\cdot\ut}_\Lp{2} &\lesssim \norm{\comm{\HT}{\Ut}\left(\frac{1}{\z_\al}\p_\al\left(\frac{\y}{\z_\al}\right)\right) - \comm{\HT}{\Ut^\pr}\left(\frac{1}{\z_\al}\p_\al\left(\frac{\y}{\z_\al}\right)\right)}_\Lp{2}\\
&\hspace{0.5cm} + \norm{\comm{\HT}{\Ut^\pr}\left(\frac{1}{\z_\al}\p_\al\left(\frac{\y}{\z_\al}\right) - \frac{1}{\z_\al^\pr}\p_\al\left(\frac{\y^\pr}{\z_\al^\pr}\right)\right)}_\Lp{2}.
\end{align*}
We begin by considering the first term:
\begin{align*}
&\norm{\comm{\HT}{\Ut}\left(\frac{1}{\z_\al}\p_\al\left(\frac{\y}{\z_\al}\right)\right) - \comm{\HT}{\Ut^\pr}\left(\frac{1}{\z_\al}\p_\al\left(\frac{\y}{\z_\al}\right)\right)}_\Lp{2}\\
&\leq \norm{\HT\left( (\z_t - \z_t^\pr)\left(\frac{1}{\z_\al}\p_\al\left(\frac{\y}{\z_\al}\right)\right) \right)}_\Lp{2} + \norm{\HT\left(\frac{1}{\z_\al}\p_\al\left(\frac{\y}{\z_\al}\right)\right)(\z_t - \z_t^\pr)}_\Lp{2}\\
&\lesssim \norm{\z_t - \z_t^\pr}_\Lp{2}.
\end{align*}
Recalling that $\z_t = Un_1 + Vt_1 + i(Un_2 + Vt_2)$, it follows that
\begin{align*}
&\norm{\comm{\HT}{\Ut}\left(\frac{1}{\z_\al}\p_\al\left(\frac{\y}{\z_\al}\right)\right) - \comm{\HT}{\Ut^\pr}\left(\frac{1}{\z_\al}\p_\al\left(\frac{\y}{\z_\al}\right)\right)}_\Lp{2}\\
&\lesssim \norm{U - U^\pr}_\Lp{2} + \norm{V - V^\pr}_\Lp{2}\\
&\lesssim \norm{\Tta - \Tta^\pr}_X.
\end{align*}
We use Lemma \ref{HTCommutatorEst1}, and the fact that $\z_t^\pr$ is bounded in $H^1$, for the second term:
\begin{equation*}
\norm{\comm{\HT}{\Ut^\pr}\left(\frac{1}{\z_\al}\p_\al\left(\frac{\y}{\z_\al}\right) - \frac{1}{\z_\al^\pr}\p_\al\left(\frac{\y^\pr}{\z_\al^\pr}\right)\right)}_\Lp{2} \lesssim \norm{\frac{1}{\z_\al}\p_\al\left(\frac{\y}{\z_\al}\right) - \frac{1}{\z_\al^\pr}\p_\al\left(\frac{\y^\pr}{\z_\al^\pr}\right)}_\Lp{2}.
\end{equation*}
We next add and subtract to obtain
\begin{align*}
\norm{\comm{\HT}{\Ut^\pr}\left(\frac{1}{\z_\al}\p_\al\left(\frac{\y}{\z_\al}\right) - \frac{1}{\z_\al^\pr}\p_\al\left(\frac{\y^\pr}{\z_\al^\pr}\right)\right)}_\Lp{2} &\lesssim \norm{\p_\al\left(\frac{\y}{\z_\al}\right)\left(\frac{1}{\z_\al} - \frac{1}{\z_\al^\pr}\right)}_\Lp{2}\\
&\hspace{0.5cm} + \norm{\frac{1}{\z_\al^\pr}\left( \p_\al\left(\frac{\y}{\z_\al}\right) - \p_\al\left(\frac{\y^\pr}{\z_\al^\pr}\right) \right)}_\Lp{2}
\end{align*}
We have the following bound for the first term:
\begin{equation*}
\norm{\p_\al\left(\frac{\y}{\z_\al}\right)\left(\frac{1}{\z_\al} - \frac{1}{\z_\al^\pr}\right)}_\Lp{2} \lesssim \norm{\z_\al - \z_\al^\pr}_\Lp{2} \lesssim \norm{\tta - \tta^\pr}_\Lp{2} \leq \norm{\Tta - \Tta^\pr}_X.
\end{equation*}
On the other hand, for the second term, we add and subtract:
\begin{equation*}
\norm{\frac{1}{\z_\al^\pr}\left( \p_\al\left(\frac{\y}{\z_\al}\right) - \p_\al\left(\frac{\y^\pr}{\z_\al^\pr}\right) \right)}_\Lp{2} \lesssim \Norm{\frac{\y}{\z_\al} - \frac{\y^\pr}{\z_\al^\pr}}{1} \lesssim \Norm{\tta - \tta^\pr}{1} + \Norm{\y - \y^\pr} \leq \norm{\Tta - \Tta^\pr}_X.
\end{equation*}
We have shown that
\begin{equation}
\label{eqn:Y11DiffEst}
\norm{(\yb_{1,1} - \yb_{1,1}^\pr)\cdot\ut}_\Lp{2} \lesssim \norm{\Tta - \Tta^\pr}_X.
\end{equation}
We again add and subtract in $\yb_{1,2}$:
\begin{align*}
&\norm{\comm{\Kop[\z]}{\z_t}\left(\p_\al\left(\frac{\y}{\z_\al}\right)\right) - \comm{\Kop[\z^\pr]}{\z_t^\pr}\left(\p_\al\left(\frac{\y^\pr}{\z_\al^\pr}\right)\right)}_\Lp{2}\\
&\leq \norm{\Kop[\z]\left( \z_t\p_\al\left(\frac{\y}{\z_\al}\right) \right) - \Kop[\z^\pr]\left( \z_t^\pr\p_\al\left(\frac{\y^\pr}{\z_\al^\pr}\right) \right)}_\Lp{2}\\
&\hspace{0.5cm} + \norm{ \z_t\Kop[\z]\left( \p_\al\left(\frac{\y}{\z_\al}\right) \right) - \z_t^\pr\Kop[\z^\pr]\left( \p_\al\left(\frac{\y^\pr}{\z_\al^\pr}\right) \right) }_\Lp{2}.
\end{align*}
We begin by adding and subtracting in the first term:
\begin{align*}
&\norm{\Kop[\z]\left( \z_t\p_\al\left(\frac{\y}{\z_\al}\right) \right) - \Kop[\z^\pr]\left( \z_t^\pr\p_\al\left(\frac{\y^\pr}{\z_\al^\pr}\right) \right)}_\Lp{2}\\
&\leq \norm{\Kop[\z]\left( \z_t\p_\al\left(\frac{\y}{\z_\al}\right) \right) - \Kop[\z^\pr]\left( \z_t\p_\al\left(\frac{\y}{\z_\al}\right) \right)}_\Lp{2}\\
&\hspace{0.5cm} + \norm{\Kop[\z^\pr]\left( \z_t\p_\al\left(\frac{\y}{\z_\al}\right) \right) - \Kop[\z^\pr]\left( \z_t^\pr\p_\al\left(\frac{\y^\pr}{\z_\al^\pr}\right) \right)}_\Lp{2}.
\end{align*}
We use Lemma \ref{KopLipschitzEst} to estimate the first term
\begin{equation*}
\norm{\Kop[\z]\left( \z_t\p_\al\left(\frac{\y}{\z_\al}\right) \right) - \Kop[\z^\pr]\left( \z_t\p_\al\left(\frac{\y}{\z_\al}\right) \right)}_\Lp{2} \lesssim \Norm{\tta - \tta^\pr}{1} \leq \norm{\Tta - \Tta^\pr}_X.
\end{equation*}
To estimate the second term, we apply Lemma \ref{KopEst}:
\begin{equation*}
\norm{\Kop[\z^\pr]\left( \z_t\p_\al\left(\frac{\y}{\z_\al}\right) \right) - \Kop[\z^\pr]\left( \z_t^\pr\p_\al\left(\frac{\y^\pr}{\z_\al^\pr}\right) \right)}_\Lp{2} \lesssim \norm{\z_t\p_\al\left(\frac{\y}{\z_\al}\right) - \z_t^\pr\p_\al\left(\frac{\y^\pr}{\z_\al^\pr}\right)}_\Lp{2}.
\end{equation*}
By adding and subtracting, we obtain
\begin{equation*}
\norm{\Kop[\z^\pr]\left( \z_t\p_\al\left(\frac{\y}{\z_\al}\right) \right) - \Kop[\z^\pr]\left( \z_t^\pr\p_\al\left(\frac{\y^\pr}{\z_\al^\pr}\right) \right)}_\Lp{2} \lesssim \norm{\z_t - \z_t^\pr}_\Lp{2} + \Norm{\frac{\y}{\z_\al} - \frac{\y^\pr}{\z_\al^\pr}}{1}.
\end{equation*}
The right-hand side is then easily bounded by $C\norm{\Tta - \Tta^\pr}_X$. We therefore have
\begin{equation*}
\norm{\Kop[\z]\left( \z_t\p_\al\left(\frac{\y}{\z_\al}\right) \right) - \Kop[\z^\pr]\left( \z_t^\pr\p_\al\left(\frac{\y^\pr}{\z_\al^\pr}\right) \right)}_\Lp{2} \lesssim \norm{\Tta - \Tta^\pr}_X.
\end{equation*}
As usual, we can add and subtract to obtain the bound
\begin{align*}
&\norm{ \z_t\Kop[\z]\left( \p_\al\left(\frac{\y}{\z_\al}\right) \right) - \z_t^\pr\Kop[\z^\pr]\left( \p_\al\left(\frac{\y^\pr}{\z_\al^\pr}\right) \right) }_\Lp{2}\\ &\lesssim \norm{\z_t - \z_t^\pr}_\Lp{2} + \norm{\Kop[\z]\left( \p_\al\left(\frac{\y}{\z_\al}\right) \right) - \Kop[\z^\pr]\left( \p_\al\left(\frac{\y^\pr}{\z_\al^\pr}\right) \right)}_\Lp{2}.
\end{align*}
We know that the first term is bounded by $C\norm{\Tta - \Tta^\pr}_X$. For the second term, we add and subtract again:
\begin{align*}
&\norm{\Kop[\z]\left( \p_\al\left(\frac{\y}{\z_\al}\right) \right) - \Kop[\z^\pr]\left( \p_\al\left(\frac{\y^\pr}{\z_\al^\pr}\right) \right)}_\Lp{2}\\
&\lesssim \norm{\Kop[\z]\left( \p_\al\left(\frac{\y}{\z_\al}\right) \right) - \Kop[\z^\pr]\left( \p_\al\left(\frac{\y}{\z_\al}\right) \right)}_\Lp{2} + \norm{\Kop[\z^\pr]\left( \p_\al\left(\frac{\y}{\z_\al}\right) - \p_\al\left(\frac{\y^\pr}{\z_\al^\pr}\right) \right)}_\Lp{2}.
\end{align*}
Lemma \ref{KopLipschitzEst} implies that the first term is bounded by $C\norm{\Tta - \Tta^\pr}_X$. On the other hand, we can control the second term via Lemma \ref{KopEst}
\begin{equation*}
\norm{\Kop[\z^\pr]\left( \p_\al\left(\frac{\y}{\z_\al}\right) - \p_\al\left(\frac{\y^\pr}{\z_\al^\pr}\right) \right)}_\Lp{2} \lesssim \Norm{\frac{\y}{\z_\al} - \frac{\y^\pr}{\z_\al^\pr}}{1}.
\end{equation*}
By adding and subtracting again, we can control the right-hand side by $C\norm{\Tta - \Tta^\pr}_X$. We have now shown that
\begin{equation}
\label{eqn:Y12DiffEst}
\norm{(\yb_{1,2} - \yb_{1,2}^\pr)\cdot\ut}_\Lp{2} \lesssim \norm{\Tta - \Tta^\pr}_X.
\end{equation}

The estimates for the remaining $\yb_{1,j}$ terms follow in a similar fashion. We have now shown that
\begin{equation}
\label{eqn:Y1Diff}
\norm{(\yb_1 - \yb_1^\pr)\cdot\ut}_\Lp{2} \lesssim_\de \norm{\Tta - \Tta^\pr}_X
\end{equation}
and therefore
\begin{equation}
\label{eqn:Y1DiffEst}
\norm{\yb_1\cdot\ut - \yb_1^\pr\cdot\ut^\pr}_\Lp{2} \lesssim_\de \norm{\Tta - \Tta^\pr}_X.
\end{equation}
The remaining terms are estimated much like those we have already seen. Ultimately, we obtain
\begin{equation}
\label{eqn:F2DiffEst}
\Norm{\F_2^\de(\Tta) - \F_2^\de(\Tta^\pr)}{\sHalf} \lesssim_\de \norm{\Tta - \Tta^\pr}_X.
\end{equation}

We now consider $\F_3^\de$:
\begin{equation*}
\abs{\F_3^\de(\Tta(\al)) - \F_3^\de(\Tta^\pr(\al))} = \frac{1}{\pi}\abs{\Int \y(\alp)k_{\FS,t}^1(\al,\alp) \ d\alp - \Int \y^\pr(\alp)(k_{\FS,t}^1)^\pr(\al,\alp) \ d\alp},
\end{equation*}
where $k_\FS^1$ is given in \eqref{eqn:SurfaceIntegralKernels}. It thus follows that
\begin{equation*}
k_{\FS,t}^1(\al,\alp) = -\Rea\bigg\{ \frac{\z_t(\alp)}{2s_{1,\al}(\al)}\p_\al\cot\frac{1}{2}(\z_1(\al) - \z(\alp)) \bigg\}.
\end{equation*}
Upon adding and subtracting, we have
\begin{align*}
&\abs{\F_3^\de(\Tta(\al)) - \F_3^\de(\Tta^\pr(\al))}\\
&\lesssim \left( \int \abs{k_{\FS,t}^1(\al,\alp)}\abs{\y(\alp)-\y^\pr(\alp)} \ d\alp + \int \abs{\y^\pr(\alp)}\abs{k_{\FS,t}^1(\al,\alp) - (k_{\FS,t}^1)^\pr(\al,\alp)} \ d\alp \right).
\end{align*}
H\"older's inequality then implies
\begin{equation*}
\abs{\F_3^\de(\Tta(\al)) - \F_3^\de(\Tta^\pr(\al))} \lesssim \norm{\y - \y^\pr}_\Lp{2} + \norm{k_{\FS,t}^1(\al,\cdot) - (k_{\FS,t}^1)^\pr(\al,\cdot)}_\Lp{2}.
\end{equation*}
We are thus left to estimate the second term and we begin by adding and subtracting:
\begin{equation*}
\norm{k_{\FS,t}^1(\al,\cdot) - (k_{\FS,t}^1)^\pr(\al,\cdot)}_\Lp{2} \lesssim \norm{\p_\al\cot\frac{1}{2}(\z_1(\al) - \z(\cdot)) - \p_\al\cot\frac{1}{2}(\z_1(\al) - \z^\pr(\cdot))}_\Lp{2} + \norm{\z_t - \z_t^\pr}_\Lp{2}.
\end{equation*}
Via Lipschitz continuity, we can estimate
\begin{equation*}
\norm{\p_\al\cot\frac{1}{2}(\z_1(\al) - \z(\cdot)) - \p_\al\cot\frac{1}{2}(\z_1(\al) - \z^\pr(\cdot))}_\Lp{2} \lesssim \norm{\z - \z^\pr}_\Lp{2} \lesssim \norm{\Tta - \Tta^\pr}_X.
\end{equation*}
Further, as we have seen already,
\begin{equation*}
\norm{\z_t - \z_t^\pr}_\Lp{2} \lesssim \norm{U - U^\pr}_\Lp{2} + \norm{V - V^\pr}_\Lp{2} \lesssim_\de \norm{\Tta - \Tta^\pr}_X.
\end{equation*}
We have thus shown that
\begin{equation*}
\norm{\F_3^\de(\Tta) - \F_3^\de(\Tta^\pr)}_\Lp{2} \lesssim_\de \norm{\Tta - \Tta^\pr}_X.
\end{equation*}
Similarly, we have
\begin{align*}
\abs{\p_\al\F_3^\de(\Tta(\al)) - \p_\al\F_3^\de(\Tta^\pr(\al))} &\lesssim \Int \abs{\p_\al k_{\FS,t}^1(\al,\alp)}\abs{\y(\alp)-\y^\pr(\alp)} \ d\alp\\
&\hspace{0.5cm} + \Int \abs{\y^\pr(\alp)}\abs{\p_\al k_{\FS,t}^1(\al,\alp) - \p_\al(k_{\FS,t}^1)^\pr(\al,\alp)} \ d\alp.
\end{align*}
Recall that
\begin{equation*}
\p_\al k_{\FS,t}^1(\al,\alp) = \Rea\bigg\{ \frac{s_{1,\al\al}(\al)\z_t(\alp)}{2s_{1,\al}^2(\al)}\p_\al\cot\frac{1}{2}(\z_1(\al) - \z(\alp)) - \frac{\z_t(\alp)}{2s_{1,\al}(\al)} \p_\al^2\cot\frac{1}{2}(\z_1(\al) - \z(\alp)) \bigg\}
\end{equation*}
Then, applying H\"older's inequality, we estimate
\begin{equation*}
\abs{\p_\al\F_3^\de(\Tta(\al)) - \p_\al\F_3^\de(\Tta^\pr(\al))} \lesssim \norm{\y - \y^\pr}_\Lp{2} + \norm{\p_\al k_{\FS,t}^1(\al,\cdot) - \p_\al(k_{\FS,t}^1)^\pr(\al,\cdot)}_\Lp{2}.
\end{equation*}
By adding and subtracting then using Lipschitz estimates, we obtain
\begin{equation*}
\norm{\p_\al k_{\FS,t}^1(\al,\cdot) - \p_\al(k_{\FS,t}^1)^\pr(\al,\cdot)}_\Lp{2} \lesssim \norm{\z - \z^\pr}_\Lp{2} + \norm{\z_t - \z_t^\pr}_\Lp{2}.
\end{equation*}
As we have seen, the right-hand side is controlled by $C(\de)\norm{\Tta - \Tta^\pr}_X$. It then follows that
\begin{equation*}
\norm{\p_\al\F_3^\de(\Tta) - \p_\al\F_3^\de(\Tta^\pr)}_\Lp{2} \lesssim_\de \norm{\Tta - \Tta^\pr}_X.
\end{equation*}
We therefore conclude that
\begin{equation}
\label{eqn:F3DiffEst}
\Norm{\F_3^\de(\Tta) - \F_3^\de(\Tta^\pr)}{1} \lesssim_\de \norm{\Tta - \Tta^\pr}_X.
\end{equation}

Finally, we move on to $\F_4^\de$, where we begin by recalling that
\begin{equation*}
\abs{\F_4^\de(\Tta) - \F_4^\de(\Tta^\pr)} = \frac{1}{\pi}\abs{\Int \y(\alp)k_{\FS,t}^2(\al,\alp) \ d\alp - \Int \y^\pr(\alp)(k_{\FS,t}^2)^\pr(\al,\alp) \ d\alp};
\end{equation*}
note that $k_\FS^2$ is given in \eqref{eqn:SurfaceIntegralKernels}. Virtually the same arguments used to derive the Lipschitz estimate for $\F_3^\de$ then imply that
\begin{equation}
\label{eqn:F4DiffEst}
\Norm{\F_4^\de(\Tta) - \F_4^\de(\Tta^\pr)}{1} \lesssim_\de \norm{\Tta - \Tta^\pr}_X.
\end{equation}

Combining the estimates \eqref{eqn:F1DiffEst}, \eqref{eqn:F2DiffEst}, \eqref{eqn:F3DiffEst} and \eqref{eqn:F4DiffEst} leads us to deduce the Lipschitz continuity of $\F$:
\begin{equation}
\label{eqn:FDiffEst}
\norm{\F^\de(\Tta) - \F^\de(\Tta^\pr)}_X \lesssim_\de \norm{\Tta - \Tta^\pr}_X.
\end{equation}
Therefore, the Picard theorem for ODE on Banach spaces implies that solutions to the regularized system exist, at least for a short time, and have the desired regularity.
\end{proof}

Now that we have proven the existence of solutions to the regularized system, we want to take a limit of the solutions $\set{\Tta^\de}_{\de > 0}$ as $\de \to 0^+$. In order to do that, we will, as previously mentioned, need to prove that the solutions exist on a common time interval and show that $\set{\Tta^\de}_{\de>0}$ converges. We begin by obtaining an existence time independent of $\de$. To that end, we have the following corollary of the uniform energy estimate Theorem \ref{UnifEnergyEst} (and the existence result of Theorem \ref{RegExist}):

\begin{cor}
\label{UnifTime}
If the regularity index $s$ of the energy space $X$ is sufficiently large, then there exists a positive $T=T(\ee,\Vzeronorm,s,\X)$ such that the solution $\Tta^\de$ of the regularized initial value problem is in $C^1([0,T];\X)$. In particular, notice that $T$ is independent of the regularization parameter $\de$.
\end{cor}
\begin{proof}
We want to apply the continuation theorem for ODE on a Banach space. We will be able to continue the solution as long as the solution does not leave $\X$. Hence, we shall show that $\Tta^\de$ cannot leave $\X$ until some time $T$ which is independent of $\de$. We treat each of the conditions defining $\X$ individually. This proof is very similar to the proof of Corollary 5.2 in \cite{A1}.

We begin with the uniform energy bound $\E < \eb$ of \eqref{eqn:EBound}, which we have imposed on the initial data. Then, the uniform energy estimate, which controls the growth of $\E$, will give us a time $T_1>0$, independent of $\de$, such that we will have $\E < \eb$ on $[0,T_1]$. Periodicity implies that \eqref{eqn:LBound} will automatically be satisfied. 

Finally, we must consider the chord-arc condition. Recalling the divided difference $q_1$, which we introduced in the proof of Lemma \ref{KopEst}, we can express the chord-arc condition as
\begin{equation}
\label{eqn:Chord-ArcDividedDiff}
\abs{q_1(\al,\alp)} > \ca \qquad (\forall \al \neq \alp).
\end{equation} 
Since we impose the chord-arc condition on the Cauchy data, uniform-in-$\de$ control of $\abs{\p_t q_1}$ will allow us to propagate the condition forward in time. We will thereby obtain a $T_2 > 0$ (perhaps small, but independent of $\de$) such that the chord-arc condition is verified for $0 \leq t \leq T_2$. In order to do this, we begin by applying Lemma \ref{SobolevEmbedding} and then the Sobolev estimate on the divided difference of equation \eqref{eqn:DividedDiffsEsts}:
\begin{equation}
\label{eqn:q1Est1}
\pnorm{\p_tq_1}{\infty} \lesssim \Norm{\p_tq_1}{\half+} \lesssim \Norm{\p_t\z_d^\de}{\sfrac{3}{2}+}.
\end{equation}
At this point, we invoke the definition of $\z_d^\de: \z_d^\de(\al,t) \coloneqq \int_0^\al s_\al(t)e^{i\tta^\de(\alp,t)} \ d\alp$, which we differentiate with respect to time. By Lemma \ref{BasicEnergyEsts} and Theorem \ref{UnifEnergyEst}, we can control $\abs{\p_t s_\al}$ and $\Norm{\p_t\tta}{r}$ uniformly in $\de$, at least for $r$ small enough ($r = \frac{3}{2}+$ is easily small enough to make this work). This gives us the aforementioned $T_2$. Then, taking $T = \min(T_1,T_2)$ gives the desired uniform time.

\end{proof}

Having obtained a common time interval on which regularized solutions exist, we now move on to establish that we can take a limit as $\de \to 0^+$. That is, we want to show that the sequence $\set{\Tta^\de}_{\de > 0}$ converges. To achieve this, we will demonstrate that $\set{\Tta^\de}_{\de > 0}$ is a Cauchy sequence in $C([0,T]; Y)$, where $Y \supset X$. Here it will be helpful to introduce some notation: 
\begin{equation}
\label{eqn:XrDef}
X_r \coloneqq H^r \times H^{r-\sfrac{1}{2}} \times H^1 \times H^1 \qquad (r \in \R).
\end{equation}
Using this notation, our energy space is given by $X = X_s$ and we further observe that, trivially, $X_r \supset X_t$ whenever $r \leq t$. Our choice will thus be to take $Y = X_1$. We have the following:

\begin{thm}
\label{RegSolCauchySeq}
If $s$ is sufficiently large, then the sequence of solutions $\set{\Tta^\de}_{\de > 0}$ of the regularized IVP \eqref{eqn:FullRegSystem}, indexed by the regularization parameter $\de$, is a Cauchy sequence in $C([0,T]; X_1)$.
\end{thm}
\begin{proof}
Here we will want to estimate the difference of regularized solutions with different regularization parameters. In particular, what we would like to obtain is some estimate of the form
\begin{equation}
\label{eqn:CauchyForm}
\sup_{t \in [0,T]} \norm{\Tta^\de(t) - \Tta^{\delt}(t)}_{X_1} \lesssim f(\de,\tilde{\de}),
\end{equation}
where $f(\de,\tilde{\de}) \to 0$ as $\max(\de,\tilde{\de}) \to 0^+$.

Following \cite{A1}, we introduce an energy for the difference of two regularized solutions with different values of the regularization parameter, which will control $\norm{\Tta^\de - \Tta^{\delt}}_{X_1}^2$. Define $\E_d$ to be given by
\begin{equation}
\label{eqn:CauchyEnergyDecomp}
\E_d \coloneqq \E_d^1 + \E_d^0 + \frac{1}{2}\Norm{\w^\de - \w^{\delt}}{1}^2 + \frac{1}{2}\Norm{\be^\de - \be^{\delt}}{1}^2,
\end{equation}
where
\begin{align}
\E_d^1 &\coloneqq \frac{1}{2}\Int (\p_\al(\tta^\de - \tta^{\delt}))^2 + \frac{1}{4\tau s_\al^\de}(\y^\de - \y^{\delt})\La(\y^\de - \y^{\delt}) + \frac{(\y^\de)^2}{16\tau^2(s_\al^\de)^2}(\y^\de - \y^{\delt})^2 \ d\al, \label{eqn:CauchyE1}\\
\E_d^0 &\coloneqq \frac{1}{2}\Int (\tta^\de - \tta^{\delt})\La(\tta^\de - \tta^{\delt}) + \frac{1}{4\tau s_\al^\de}(\y^\de - \y^{\delt})^2 + (\tta^\de - \tta^{\delt})^2 \ d\al. \label{eqn:CauchyE2}
\end{align}
Noting that the regularized solutions all satisfy the same initial condition, regardless of the value of the regularization parameter $\de$, so we have $\E_d(0) = 0$. Our goal will then be to come up with a suitable bound on the growth of $\E_d$ over time. We begin by computing
\begin{align*}
\frac{d\E_d^1}{dt} &= \Int \p_\al(\tta_t^\de - \tta_t^{\delt})\p_\al(\tta^\de - \tta^{\delt}) \ d\al + \Int \frac{1}{4\tau s_\al^\de} (\y_t^\de - \y_t^{\delt})\La(\y^\de - \y^{\delt}) \ d\al \\
&\hspace{0.5cm} + \Int \frac{(\y^\de)^2}{16\tau^2 (s_\al^\de)^2} (\y_t^\de - \y_t^{\delt})(\y^\de - \y^{\delt}) \ d\al \\
&\hspace{0.5cm} +\Int \p_t\left( \frac{1}{4\tau s_\al^\de} \right)(\y^\de - \y^{\delt})\La(\y^\de - \y^{\delt}) + \p_t\left( \frac{(\y^\de)^2}{16\tau^2(s_\al^\de)^2} \right)(\y^\de - \y^{\delt})^2 \ d\al.\\
&= d_1 + d_2 + d_3 + d_4.
\end{align*}

We begin with $d_1$ and plug in for $\tta_t^\de$ and $\tta_t^{\delt}$ from \eqref{eqn:ThetaRegEvolutionEqn}:
\begin{equation}
\label{eqn:d1Diff}
d_1 = \Int \left(\frac{1}{2(s_\al^\de)^2}\HT(\J\y_{\al\al}^\de) - \frac{1}{2(s_\al^{\delt})^2}\HT(\Jt\y_{\al\al}^{\delt})\right)(\tta_\al^\de - \tta_\al^{\delt}) \ d\al + e_1,
\end{equation}
where $e_1$ is the remainder. We now examine $d_2$, again plugging in for $\y_t^\de$ and $\y_t^{\delt}$ from \eqref{eqn:yRegEvolutionEqn}. These substitutions yield
\begin{align}
\label{eqn:d2Diff}
d_2 &= \Int \frac{1}{4\tau s_\al^\de}\left( \frac{2\tau}{s_\al^\de}\J\tta_{\al\al}^\de - \frac{2\tau}{s_\al^{\delt}} \Jt\tta_{\al\al}^{\delt} \right) \La(\y^\de - \y^{\delt}) \ d\al \nonumber\\
&\hspace{0.5cm} + \Int \frac{1}{4\tau s_\al^\de}\left( \frac{1}{2(s_\al^\de)^2}\HT((\y^\de)^2\J\tta_\al^\de) - \frac{1}{2(s_\al^{\delt})^2}\HT((\y^{\delt})^2\Jt\tta_\al^{\delt}) \right)\La(\y^\de - \y^{\delt}) \ d\al + e_2,
\end{align}
with $e_2$ being the remainder. By adding together $d_1$ and $d_2$, we will obtain a cancellation.

Recall that $s_\al^\de$ is bounded above and below (away from zero), independent of $\de$, by Lemma \ref{ArclengthEltBounds}. Let $w_1$ denote the sum of the integral in \eqref{eqn:d1Diff} and the first integral in \eqref{eqn:d2Diff}. Upon an integration by parts and noting the bounds on $s_\al^\de$, we have
\begin{align}
\label{eqn:w1Diff}
w_1 &= -\Int \left( \frac{1}{2(s_\al^\de)^2}\La(\J\y^\de) - \frac{1}{2(s_\al^{\delt})^2}\La(\Jt\y^{\delt}) \right)(\tta_{\al\al}^\de - \tta_{\al\al}^{\delt}) \ d\al \nonumber\\
&\hspace{0.5cm} + \Int \left( \frac{1}{2(s_\al^\de)^2}\J\tta_{\al\al}^\de - \frac{1}{2(s_\al^{\delt})^2}\Jt\tta_{\al\al}^{\delt} \right)\La(\y^\de - \y^{\delt}) \ d\al \nonumber\\
&\sim -\Int \left( \La(\J\y^\de) - \La(\Jt\y^{\delt}) \right)(\tta_{\al\al}^\de - \tta_{\al\al}^{\delt}) \ d\al \nonumber\\
&\hspace{0.5cm} + \Int \left( \J\tta_{\al\al}^\de - \Jt\tta_{\al\al}^{\delt} \right)\La(\y^\de - \y^{\delt}) \ d\al.
\end{align}
Recall from Remark \ref{NotationRmk} that we use $A \sim B$ to denote $B \lesssim A \lesssim B$. We now add and subtract in each of the two integrals in \eqref{eqn:w1Diff} to obtain
\begin{align*}
w_1 &\sim -\Int \La(\J\y^\de - \Jt\y^\de)(\tta_{\al\al}^\de - \tta_{\al\al}^{\delt}) \ d\al\\
&\hspace{0.5cm} -\Int \La(\Jt(\y^\de - \y^{\delt}))(\tta_{\al\al}^\de - \tta_{\al\al}^{\delt}) \ d\al\\
&\hspace{0.5cm} + \Int \left( \J\tta_{\al\al}^\de - \Jt\tta_{\al\al}^\de \right)\La(\y^\de - \y^{\delt}) \ d\al\\
&\hspace{0.5cm} + \Int \left( \Jt(\tta_{\al\al}^\de - \tta_{\al\al}^{\delt}) \right)\La(\y^\de - \y^{\delt}) \ d\al.
\end{align*}
The second and fourth integrals above will cancel since $\Jt$ is self-adjoint. Turning now to the first integral, we integrate by parts and apply H\"older's inequality:
\begin{equation*}
-\Int \La(\J\y^\de - \Jt\y^\de)(\tta_{\al\al}^\de - \tta_{\al\al}^{\delt}) \ d\al \leq \norm{\HT(\J\y_{\al\al}^\de - \Jt\y_{\al\al}^\de)}_\Lp{2}\norm{\tta_\al^\de - \tta_\al^{\delt}}_\Lp{2}.
\end{equation*}
Using Lemma \ref{HilbertL2Isometry} and Lemma \ref{JLemma2} as well as the uniform energy estimate of Theorem \ref{UnifEnergyEst}, we can control the first norm by
\begin{equation*}
\norm{\HT(\J\y_{\al\al}^\de - \Jt\y_{\al\al}^\de)}_\Lp{2} \leq \max(\de,\delt)\Norm{\y_{\al\al}^\de}{1} \lesssim \max(\de,\delt). 
\end{equation*}
The second norm above is clearly controlled by $\sqrt{\E_d}$. Finally, turning to the third integral, we use the fact that $\La$ is self-adjoint to rewrite it as
\begin{equation*}
\Int \HT( \J\tta_{\al\al\al}^\de - \Jt\tta_{\al\al\al}^\de )(\y^\de - \y^{\delt}) \ d\al \leq \norm{\HT( \J\tta_{\al\al\al}^\de - \Jt\tta_{\al\al\al}^\de )}_\Lp{2}\norm{\y^\de - \y^{\delt}}_\Lp{2}.
\end{equation*}
The second norm is again easily controlled by $\sqrt{\E_d}$, while for the first norm we apply Lemmas \ref{HilbertL2Isometry} and \ref{JLemma2} as well as Theorem \ref{UnifEnergyEst}:
\begin{equation*}
\norm{\HT( \J\tta_{\al\al\al}^\de - \Jt\tta_{\al\al\al}^\de )}_\Lp{2} \leq \max(\de,\delt)\Norm{\tta_{\al\al\al}^\de}{1} \lesssim \max(\de,\delt).
\end{equation*}
We have now shown that
\begin{equation}
\label{eqn:w1DiffEst}
w_1 \lesssim \max(\de,\delt)\sqrt{\E_d}.
\end{equation}

The cancellation we saw in the sum of $d_1$ and $d_2$ corresponds to the primary cancellation from Theorem \ref{UnifEnergyEst}. So, we should expect some further cancellation which corresponds to the secondary cancellation of Theorem \ref{UnifEnergyEst}. Consider $d_3$ and plug in from \eqref{eqn:yRegEvolutionEqn}:
\begin{equation}
\label{eqn:d3Diff}
d_3 = \Int \frac{(\y^\de)^2}{16\tau^2 (s_\al^\de)^2}\left( \frac{2\tau}{s_\al^\de}\J\tta_{\al\al}^\de - \frac{2\tau}{s_\al^{\delt}}\Jt\tta_{\al\al}^{\delt} \right)(\y^\de - \y^{\delt}) \ d\al + e_3,
\end{equation}
where $e_3$ once again denotes the remainder. To obtain the analogue of the secondary cancellation of Theorem \ref{UnifEnergyEst}, we let $w_2$ denote the sum of the second integral in \eqref{eqn:d2Diff} and the integral in \eqref{eqn:d3Diff}. Utilizing the self-adjointness of $\La$ as well as the identity $\HT^2 = -\id$, which implies that $\La\HT = -\p_\al$, we get
\begin{align}
\label{eqn:w2Diff}
w_2 &= -\Int \frac{1}{4\tau s_\al^\de}\left( \frac{1}{2(s_\al^\de)^2}\p_\al((\y^\de)^2\J\tta_\al^\de) - \frac{1}{2(s_\al^{\delt})^2}\p_\al((\y^{\delt})^2\Jt\tta_\al^{\delt}) \right)(\y^\de - \y^{\delt}) \ d\al \nonumber\\
&\hspace{0.5cm} + \Int \frac{(\y^\de)^2}{16\tau^2 (s_\al^\de)^2}\left( \frac{2\tau}{s_\al^\de}\J\tta_{\al\al}^\de - \frac{2\tau}{s_\al^{\delt}}\Jt\tta_{\al\al}^{\delt} \right)(\y^\de - \y^{\delt}) \ d\al.
\end{align}
We now use the Leibniz rule to expand the derivative in the first integral above then add and subtract in the appropriate integral. This process yields
\begin{align*}
w_2 &= -\Int \frac{1}{4\tau s_\al^\de}\left( \frac{1}{(s_\al^\de)^2}\y^\de\y_\al^\de\J\tta_\al^\de - \frac{1}{(s_\al^{\delt})^2}\y^{\delt}\y_\al^{\delt}\Jt\tta_\al^{\delt} \right)(\y^\de - \y^{\delt}) \ d\al\\
&\hspace{0.5cm} - \Int \frac{1}{8\tau s_\al^\de}\left( \frac{1}{(s_\al^\de)^2}(\y^\de)^2\J\tta_{\al\al}^\de - \frac{1}{s_\al^\de s_\al^{\delt}}(\y^\de)^2\Jt\tta_{\al\al}^{\delt} \right)(\y^\de - \y^{\delt}) \ d\al\\
&\hspace{0.5cm} - \Int \frac{1}{8\tau s_\al^\de}\left( \frac{1}{s_\al^\de s_\al^{\delt}}(\y^\de)^2\Jt\tta_{\al\al}^{\delt} - \frac{1}{(s_\al^{\delt})^2}(\y^{\delt})^2\Jt\tta_{\al\al}^{\delt} \right)(\y^\de - \y^{\delt}) \ d\al\\
&\hspace{0.5cm} + \Int \frac{(\y^\de)^2}{8\tau (s_\al^\de)^2}\left( \frac{1}{s_\al^\de}\J\tta_{\al\al}^\de - \frac{1}{s_\al^{\delt}}\Jt\tta_{\al\al}^{\delt} \right)(\y^\de - \y^{\delt}) \ d\al.
\end{align*}
Observe that the second and fourth integrals above cancel.

We now turn to estimating the remaining integrals which did not cancel. Let $w_{2,1}$ and $w_{2,3}$ denote the first integral above and the third integral above, respectively; these are the remaining integrals which need to be controlled. We will again make use of Lemma \ref{ArclengthEltBounds} to bound $s_\al^\de$ below (away from zero) and above for any $\de > 0$. We begin with $w_{2,1}$, where, after adding and subtracting, we have
\begin{align*}
w_{2,1} &\sim -\Int (\y^\de\y_\al^\de\J\tta_\al^\de - \y^{\delt}\y_\al^\de\J\tta_\al^\de)(\y^\de - \y^{\delt}) \ d\al\\
&\hspace{0.5cm} - \Int (\y^{\delt}\y_\al^\de\J\tta_\al^\de - \y^{\delt}\y_\al^{\delt}\J\tta_\al^\de)(\y^\de - \y^{\delt}) \ d\al\\
&\hspace{0.5cm} - \Int (\y^{\delt}\y_\al^{\delt}\J\tta_\al^\de - \y^{\delt}\y_\al^{\delt}\Jt\tta_\al^\de)(\y^\de - \y^{\delt}) \ d\al\\
&\hspace{0.5cm} - \Int (\y^{\delt}\y_\al^{\delt}\Jt\tta_\al^\de - \y^{\delt}\y_\al^{\delt}\Jt\tta_\al^{\delt})(\y^\de - \y^{\delt}) \ d\al.
\end{align*}
Utilizing H\"older's inequality, Lemma \ref{SobolevMultiplication} and the uniform energy estimate of Theorem \ref{UnifEnergyEst}, we estimate first integral above from $w_{2,1}$:
\begin{equation*}
-\Int (\y^\de\y_\al^\de\J\tta_\al^\de - \y^{\delt}\y_\al^\de\J\tta_\al^\de)(\y^\de - \y^{\delt}) \ d\al \lesssim \norm{\y^\de - \y^{\delt}}_\Lp{2}^2 \lesssim \E_d.
\end{equation*}
We recognize a perfect derivative and integrate by parts to rewrite the second integral of $w_{2,1}$ from above:
\begin{equation*}
-\Int \y^{\delt}\J\tta_\al^\de(\y_\al^\de - \y_\al^{\delt})(\y^\de - \y^{\delt}) \ d\al = \frac{1}{2}\int \p_\al(\y^{\delt}\J\tta_\al^\de)(\y^\de - \y^{\delt})^2 \ d\al.
\end{equation*}
Then, H\"older's inequality, Lemma \ref{SobolevMultiplication} and the uniform energy estimate imply that
\begin{equation*}
\frac{1}{2}\Int \p_\al(\y^{\delt}\J\tta_\al^\de)(\y^\de - \y^{\delt})^2 \ d\al \leq \norm{\p_\al(\y^{\delt}\J\tta_\al^\de)(\y^\de - \y^{\delt})}_\Lp{2}\norm{\y^\de - \y^{\delt}}_\Lp{2} \lesssim \norm{\y^\de - \y^{\delt}}_\Lp{2}^2 \lesssim \E_d.
\end{equation*}
For the third integral of $w_{2,1}$, we use H\"older's inequality to obtain the bound:
\begin{equation*}
- \Int (\y^{\delt}\y_\al^{\delt}\J\tta_\al^\de - \y^{\delt}\y_\al^{\delt}\Jt\tta_\al^\de)(\y^\de - \y^{\delt}) \ d\al \leq \norm{\y^{\delt}\y_\al^{\delt}(\J\tta_\al^\de - \Jt\tta_\al^\de)}_\Lp{2}\norm{\y^\de - \y^{\delt}}_\Lp{2}.
\end{equation*}
Then, by Lemma \ref{SobolevMultiplication}, Lemma \ref{JLemma2} and Theorem \ref{UnifEnergyEst}, we get
\begin{equation*}
\norm{\y^{\delt}\y_\al^{\delt}(\J\tta_\al^\de - \Jt\tta_\al^\de)}_\Lp{2}\norm{\y^\de - \y^{\delt}}_\Lp{2} \lesssim \max(\de,\delt)\norm{\y^\de - \y^{\delt}}_\Lp{2} \lesssim \max(\de,\delt)\sqrt{\E_d}.
\end{equation*}
Considering the final integral in $w_{2,1}$, H\"older's inequality, Lemma \ref{SobolevMultiplication}, the uniform energy estimate of Theorem \ref{UnifEnergyEst} and \ref{JLemma1} yield
\begin{equation*}
- \Int (\y^{\delt}\y_\al^{\delt}\Jt\tta_\al^\de - \y^{\delt}\y_\al^{\delt}\Jt\tta_\al^{\delt})(\y^\de - \y^{\delt}) \ d\al \lesssim \Norm{\tta^\de - \tta^{\delt}}{1}\norm{\y^\de - \y^{\delt}}_\Lp{2} \lesssim \E_d.
\end{equation*}
It then follows that
\begin{equation}
\label{eqn:w21DiffEst}
w_{2,1} \lesssim \E_d + \max(\de,\delt)\sqrt{\E_d}.
\end{equation}

We now proceed to examine $w_{2,3}$, which we can rewrite as
\begin{align*}
w_{2,3} &\sim - \Int \Jt\tta_{\al\al}^{\delt}(\y^\de + \y^{\delt})(\y^\de - \y^{\delt})^2 \ d\al.
\end{align*}
Then, by H\"older's inequality, we have
\begin{equation*}
w_{2,3} \lesssim \norm{\Jt\tta_{\al\al}^{\delt}(\y^\de + \y^{\delt})(\y^\de - \y^{\delt})}_\Lp{2}\norm{\y^\de - \y^{\delt}}_\Lp{2}.
\end{equation*}
We conclude by applying Lemma \ref{SobolevMultiplication}, Lemma \ref{JLemma1} and Theorem \ref{UnifEnergyEst} to control the right-hand side above:
\begin{equation}
\label{eqn:w23DiffEst}
w_{2,3} \lesssim \norm{\y^\de - \y^{\delt}}_\Lp{2}^2 \lesssim \E_d.
\end{equation}
Combining \eqref{eqn:w21DiffEst} with \eqref{eqn:w23DiffEst} and recalling the secondary cancellation, it therefore holds that
\begin{equation}
\label{eqn:w2DiffEst}
w_2 \lesssim \E_d + \max(\de,\delt)\sqrt{\E_d}.
\end{equation}

We are now left to estimate $d_4$ as well as the remainder terms: $e_1$, $e_2$ and $e_3$. There are no surprises here. We have
\begin{align}
d_4 &\lesssim \E_d, \label{eqn:d4DiffEst}\\
e_1 &\lesssim \E_d + \max(\de,\delt)\sqrt{\E_d}, \label{eqn:e1CauchyEst}\\
e_2 &\lesssim \E_d + \max(\de,\delt)\sqrt{\E_d}, \label{eqn:e2CauchyEst}\\
e_3 &\lesssim \E_d + \max(\de,\delt)\sqrt{\E_d}. \label{eqn:e3CauchyEst}
\end{align}
We have found that
\begin{equation}
\label{eqn:CauchyEnergyEst1}
\frac{d\E_d^1}{dt} \lesssim \E_d + \max(\de,\delt)\sqrt{\E_d}.
\end{equation}

We can estimate the remaining terms in a similar way. Doing so gives
\begin{equation}
\label{eqn:CauchyEnergyEstimate}
\frac{d\E_d}{dt} \lesssim \E_d + \max(\de,\delt)\sqrt{\E_d},
\end{equation}
which we can rewrite as 
\begin{equation}
\label{eqn:CauchyEnergyEstimateRedux}
\frac{d\sqrt{\E_d}}{dt} \lesssim \sqrt{\E_d} + \max(\de,\delt).
\end{equation}
Upon solving the differential inequality in \eqref{eqn:CauchyEnergyEstimateRedux}, recalling that $\E_d(0)=0$, we find that
\begin{equation}
\label{eqn:CauchyEnergyBound}
\sqrt{\E_d(t)} \leq \max(\de,\delt)(e^{ct}-1),
\end{equation}
where $c$ is the implied constant in \eqref{eqn:CauchyEnergyEstimateRedux}. Now, we recall that, by the definition of $\E_d$, we have
\begin{equation}
\label{eqn:CauchyEnergyNormControl}
\norm{(\tta^\de - \tta^{\tilde{\de}}, \y^\de - \y^{\tilde{\de}}, \w^\de - \w^{\tilde{\de}}, \be^\de - \be^{\tilde{\de}})}_{X_1} \lesssim \sqrt{\E_d}.
\end{equation}
Finally, we take the supremum and utilize \eqref{eqn:CauchyEnergyBound}:
\begin{equation}
\label{eqn:CauchyEst}
\sup_{t \in [0,T]} \norm{\Tta^\de(t) - \Tta^{\delt}(t)}_{X_1} \lesssim \sup_{t \in [0,T]} \sqrt{\E_d(t)} \lesssim \max(\de,\delt)(e^{cT}-1).
\end{equation}
This is of the desired form \eqref{eqn:CauchyForm} and so we see that $\set{\Tta^\de}_{\de>0}$ is indeed a Cauchy sequence in $C([0,T]; X_1)$.

\end{proof}

We are now able to take a limit of the regularized solutions as $\de \to 0^+$. The next step is, of course, to show that this limit solves the non-regularized system. We defer doing this until the next section as we shall prove a preliminary regularity result at the same time.

\section{Regularity of Solutions}

At this point we know that the sequence of solutions to the regularized system $\set{(\tta^\de,\y^\de,\w^\de,\be^\de)}_{\de>0}$ converges to a limit as $\de \to 0^+$.

\begin{rmk}
\label{LimitRmk}
Set 
\begin{equation}
\label{eqn:LimitSoln}
\Tta = (\tta,\y,\w,\be) \coloneqq \lim_{\de \to 0^+} \Tta^\de = \lim_{\de \to 0^+} (\tta^\de,\y^\de,\w^\de,\be^\de).
\end{equation}
Recalling that $\Tta^\de \in C([0,T];X_1)$ and that the convergence of $\Tta^\de \to \Tta$, in $X_1$, is uniform in time, we are able to deduce that $\Tta \in C([0,T];X_1)$.
\end{rmk}

In this section, we will show that $\Tta$ solves the non-regularized system (\eqref{eqn:ModelProblem} with right-hand side given by \eqref{eqn:FullRHS}), that this solution is unique and that it lies in $C([0,T];\X)$. We shall begin by first showing that $\Tta$ solves the non-regularized system and is continuous, with respect to time, in the weak topology.

\begin{thm}
\label{PrelimReg}
Let $\Tta = (\tta,\y,\w,\be)$ be as in Remark \ref{LimitRmk} (i.e., the limit as $\de \to 0^+$ of the sequence of solutions $\set{\Tta^\de}_{\de>0}$ to the regularized system \eqref{eqn:FullRegSystem}). Then, $\Tta$ solves the non-regularized system \eqref{eqn:ModelProblem} with right-hand side given by \eqref{eqn:FullRHS}. Additionally, $(\tta,\y) \in C_{\mathrm{W}}([0,T]; H^s \times H^\sHalf)$, where $C_{\mathrm{W}}([0,T];H^r)$ denotes the space of weakly continuous function from $[0,T]$ into $H^r$. Finally, $(\tta,\y)$ is additionally in $C([0,T]; H^r \times H^t)$ for $1 \leq r < s$ and $\frac{1}{2} \leq t < s - \frac{1}{2}$. Finally, we have $\Tta \in \X$.
\end{thm}
\begin{proof}
Notice that we do not make any preliminary regularity statements about $\w$ and $\be$. This is because we already have the top-level regularity result for these terms: $\w,\be \in C([0,T];H^1)$. The proof of the preliminary regularity results for $(\tta,\y)$ is virtually identical to the proof of the corresponding claim in Theorem 5.4 in \cite{A1}. With the preliminary regularity results in hand, the proof that $\Tta$ solves the non-regularized system is exactly analogous to the proof given in \cite{A1}. So, omit many details and paint in broad strokes in some places. 

We begin by recalling what we know about the limit $\Tta$, namely that $\Tta \in C([0,T];X_1)$. Since the unit ball of $X$ is compact in the weak topology, the uniform energy estimate of Theorem \ref{UnifEnergyEst} implies that $\Tta^\de \rightharpoonup \Tta \in X$ as $X \subset X_1$. Moreover, since $\Tta^\de$ satisfies the chord-arc condition \eqref{eqn:ChordArc} as well as the estimates \eqref{eqn:LBound} and \eqref{eqn:EBound} for every $\de > 0$, it must be the case that $\Tta$ satisfies the chord-arc condition \eqref{eqn:ChordArc} as well as \eqref{eqn:LBound} and \eqref{eqn:EBound}. We conclude that $\Tta \in \X$.

The claim that $(\tta,\y)\in C([0,T];H^r\times H^t)$ for any $1 \leq r < s$ and $\frac{1}{2} \leq t < s - \frac{1}{2}$ is obtained via applying the interpolation estimate $\Norm{u}{r} \lesssim \pnorm{u}{2}^{1-\vartheta_r}\Norm{u}{\sigma}^{\vartheta_r}$, $\vartheta_r \coloneqq \frac{r}{\sigma}$, to the differences $\tta - \tta^\de$ and $\y - \y^\de$.

We can show that $\tta \in C_{\mathrm{W}}([0,T];H^s)$ and $\y \in C([0,T];H^{\sHalf})$ from the definition of weak continuity. We focus on $\tta$, but the argument for $\y$ is totally analogous. We begin by letting $h>0$ be given and taking $u \in H^{-s}$ to be arbitrary. For arbitrary $1 \leq r < s$, we take $v \in H^{-r}$ to satisfy $\Norm{u-v}{-s} \leq \frac{h}{3}$. By writing
\begin{equation*}
\langle u, \tta - \tta^\de \rangle = \langle u - v, \tta \rangle + \langle v, \tta - \tta^\de \rangle + \langle v-u, \tta^\de \rangle,
\end{equation*}
we conclude that, for $\de > 0$ sufficiently small, $\abs{\langle u,\tta - \tta^\de \rangle} \leq h$. Observe that this estimate is uniform in time.

Finally, we show that $\Tta$ solves the initial value problem for the non-regularized system. Again, we focus on $\tta$, but the arguments for $\y$, $\w$ and $\be$ are no different. We begin by writing
\begin{equation*}
\tta^\de(\al,t) = \tta_0(\al) + \int_0^t \F_1^\de(\Tta^\de(\al,t^\pr)) \ dt^\pr.
\end{equation*}
We now have enough regularity to pass to the limit in the above equation:
\begin{equation*}
\tta(\al,t) = \tta_0(\al) + \int_0^t \F_1(\Tta(\al,t^\pr)) \ dt^\pr.
\end{equation*}
Observing that the quantity on the right-hand side is differentiable with respect to time, we take the derivative to obtain $\tta_t = \F_1(\Tta)$, which is what we wanted.
\end{proof}

Before proceeding to the top-level regularity result for solutions to the non-regularized system, we want to prove that the initial value problem for the non-regularized system is stable under small perturbations of the Cauchy data. This stability result will immediately imply the uniqueness of solutions to the non-regularized initial value problem. We have the following theorem on the dependence of the solutions on the initial data:

\begin{thm}
\label{StabilityThm}
If the regularity index $s$ of $X$ is sufficiently large and $\Tta, \Tta^\pr \in X$ are solutions of the initial value problem for the non-regularized system (again, this is the system \eqref{eqn:ModelProblem} with right-hand side given by \eqref{eqn:FullRHS}) on the time interval $[0,T]$, with corresponding initial data $\Tta_0, \Tta_0^\pr \in \X$, then it holds that
\begin{equation}
\label{eqn:StabilityEst}
\sup_{t \in [0,T]} \norm{\Tta - \Tta^\pr}_{X_1} \lesssim_T \norm{\Tta_0 - \Tta_0^\pr}_{X_1}.
\end{equation}
\end{thm}
\begin{proof}
As in the proof of Theorem \ref{RegSolCauchySeq}, we begin by defining an appropriate energy. In this case, it is the energy $\Ef$ of the difference of two solutions with different Cauchy data: 
\begin{equation}
\label{eqn:StabilityEnergyDecomp}
\Ef \coloneqq \Ef^1 + \Ef^0 + \frac{1}{2}\Norm{\w - \w^\pr}{1}^2 + \frac{1}{2}\Norm{\be - \be^\pr}{1}^2,
\end{equation}
where
\begin{align}
\Ef^1 &\coloneqq \frac{1}{2}\Int (\p_\al(\tta - \tta^\pr))^2 + \frac{1}{4\tau s_\al}(\y - \y^\pr)\La(\y - \y^\pr) + \frac{\y^2}{16\tau^2 s_\al^2}(\y - \y^\pr)^2 \ d\al, \label{eqn:StabilityE1}\\
\Ef^0 &\coloneqq \frac{1}{2}\Int (\tta - \tta^\pr)\La(\tta - \tta^\pr) + \frac{1}{4\tau s_\al}(\y - \y^\pr)^2 + (\tta - \tta^\pr)^2 \ d\al. \label{eqn:StabilityE2}
\end{align}
We denote this energy $\Ef$ as it controls the continuity of the flow map (in $X_1 = H^1 \times H^{\half} \times H^1 \times H^1$). We note that, since $\Tta$ and $\Tta^\pr$ satisfy different initial conditions, $\Ef(0)$ will not vanish as was the case in Theorem \ref{RegSolCauchySeq}, however $\Ef(0) \sim \norm{\Tta_0 - \Tta_0^\pr}_{X_1}$.

We want to estimate $\frac{d\Ef}{dt}$. The calculations are very similar to those in the proofs of Theorem \ref{RegSolCauchySeq} and Theorem \ref{UnifEnergyEst}, so we omit them. In summary, we obtain
\begin{equation}
\label{eqn:EDEst}
\frac{d\Ef}{dt} \lesssim \Ef.
\end{equation}
We then have
\begin{equation}
\label{eqn:EDSolution}
\Ef(t) \leq \Ef(0)e^{ct},
\end{equation}
where $c$ is the implied constant in \eqref{eqn:EDEst}. Therefore, it follows that
\begin{equation}
\label{eqn:FinalStabilityEst}
\sup_{t \in [0,T]} \norm{\Tta(t) - \Tta^\pr(t)}_{X_1}^2 \lesssim \sup_{t \in [0,T]} \Ef(t) \leq \Ef(0)e^{cT} \lesssim e^{cT}\norm{\Tta_0 - \Tta_0^\pr}_{X_1}^2. 
\end{equation}
This is what we wanted to show.

\end{proof}

\begin{thm}
\label{FullReg}
Solutions of the non-regularized initial value problem \eqref{eqn:ModelProblem} (where the right-hand side is given by \eqref{eqn:FullRHS}) are in $C([0,T]; X)$.
\end{thm}
\begin{proof}
We already have $\w,\be \in C([0,T];H^1)$, so all that remains is to show that $\tta \in C([0,T];H^s)$ and $\y \in C([0,T];H^{\sHalf})$. The proof of this is virtually the same as the proof of Theorem 5.6 in \cite{A1}, so we omit details and only give a sketch. The proof proceeds by deducing continuity of various components of the energy and from there deducing the continuity of $(\tta,\y)$. By comparing with the energy, one first establishes that $(\tta,\y)$ is right-continuous at $t=0$. Then, one picks an arbitrary $t_0 \in [0,T]$. By running the Picard existence argument of Theorem \ref{RegExist}, we can obtain a solution in a small neighborhood of $t_0$. However, Theorem \ref{StabilityThm} implies that this solution coincides with the solution starting at $t=0$. Then, the running the argument that gives right-continuity at $t=0$ will give right continuity at $t=t_0$. Now, we have right-continuity of solutions. We can then just reverse time in the arguments for right-continuity to obtain left-continuity of solutions at $t=t_0$. This gives the desired regularity result: $\Tta \in C([0,T];X)$.
\end{proof}

\section{Proof of Theorem \ref{MainTheorem1}}

In this section, we will prove the first main theorem, Theorem \ref{MainTheorem1}. In the previous sections, we have shown that the model problem \eqref{eqn:ModelProblem} is well-posed locally in time and that solutions are continuous from $[0,T]$ into $X$. What remains is to show that these results can be extended to the full water waves system \eqref{eqn:WaterWavesSystem} and then to prove the lifespan results. We shall begin by discussing how to extend the previous local well-posedness and regularity results on the model problem to the full water waves system. Then, we will prove the desired lifespan results as a corollary of the main energy estimate Theorem \ref{UnifEnergyEst}.

\subsection{Extending the Results on the Model Problem to the Full Water Waves System}

In order to extend the well-posedness and regularity results for the model problem to the full water waves system \eqref{eqn:WaterWavesSystem}, we will, following the plan outlined in Remark \ref{NonlocalRmk}, utilize mapping properties of the operator $(\id + \K[\Tta])^{-1}$. In section 5 of \cite{AMEA}, it is shown that the operator $\id + \K$ is an invertible Fredholm operator (see Appendix B below for an alternative presentation on the solvability of the integral equations). We then obtain the following:

\begin{lemma}
\label{MappingProps1}
The inverse operator $(\id + \K[\Tta])^{-1}: X \to X$ is continuous.
\end{lemma}
\begin{proof}
At this point, we know that $\id + \K[\Tta]$ is an invertible Fredholm operator. The desired result then follows from standard Fredholm theory. In particular, we can apply the Fredholm alternative. It is shown in \cite{AMEA} that $\id + \K[\Tta]$ is a Fredholm operator with trivial kernel and so, by the Fredholm alternative, $\id + \K[\Tta]$ is also a surjection. Hence, $\id + \K[\Tta]$ is a bounded, bijective linear operator on a Hilbert space and so has a bounded inverse by the bounded inverse theorem.
\end{proof}

Lemma \ref{MappingProps1} is the desired mapping property and it gives us the following local-in-time well-posedness theorem, recalling that $B$ is defined in \eqref{eqn:BDef}:
\begin{thm}
\label{WaterWavesSystemLWP}
Let $s$ be sufficiently large. The system \eqref{eqn:WaterWavesSystemForm} is locally well-posed. Namely, there exists a unique solution $\Tta \in C([0,T(B,\Vzeronorm)]; \X)$ to the system \eqref{eqn:WaterWavesSystemForm} and the flow map $\Tta_0 \mapsto \Tta$ is continuous.
\end{thm}
\begin{proof}
The solvability result of \cite{AMEA} (or, alternatively, Appendix B) and Lemma \ref{MappingProps1} imply that Theorem \ref{UnifEnergyEst}, Theorem \ref{RegExist}, Corollary \ref{UnifTime}, Theorem \ref{RegSolCauchySeq}, Theorem \ref{PrelimReg}, Theorem \ref{StabilityThm}, Theorem \ref{FullReg} apply to the system \eqref{eqn:WaterWavesSystem}. This then gives the desired result.
\end{proof}

\subsection{Lifespan of Solutions}

We have now established that the full water waves system \eqref{eqn:WaterWavesSystem} is locally well-posed. We now turn to address the question of how long these solutions persist. The theory of quasilinear hyperbolic equations suggests an $\bigo(\frac{1}{\ee})$ lifespan in the small-data setting, given that our system is quadratically nonlinear \cite{Kato1, Kato2, Maj2}. However, obtaining this existence time requires careful, delicate analysis. Our goal here is to prove that when $V_0 = 0$, we get an existence time of order $\bigo(\log\frac{1}{\ee})$ as this can be done using the energy estimates we have already obtained. On the other hand, when $V_0 \neq 0$, we simply show that the existence time is $\bigo(\frac{1}{(1+\Vzeronorm)^2})$. In a forthcoming paper, we will prove the quadratic $\bigo(\frac{1}{\ee})$ lifespan for small-data solutions when $V_0 = 0$. We first consider the case $V_0 = 0$ and then proceed to consider $V_0 \neq 0$.

\subsubsection{Zero Background Current}

In considering the existence time of solutions, the background current $V_0$ plays a significant role. For example, even in the case of a flat initial free surface, the interaction of the background current with the obstacle may lead to large deviations in the free surface and the formation of splash singularities (see \cite{AMEA} for numerical simulations). Here we shall consider the lifespan of solutions in the special case where $V_0 = 0$. In that case, by Theorem \ref{UnifEnergyEst}, we have the following energy estimate:
\begin{equation}
\label{eqn:EnergyEstReduce}
\frac{d\E}{dt} \lesssim \Pol(\E) = \E + \E^N,
\end{equation}
where $N > 2$; recall that $\chi=0$ when $V_0=0$. Further, the energy estimate \eqref{eqn:EnergyEstReduce} applies to the full water waves system as we discussed in the previous subsection.

As noted earlier, our goal here is to prove a ``short'' existence time using just the uniform energy estimate of Theorem \ref{UnifEnergyEst} and some basic analysis tools. Specifically, we have the following result on the lifespan of solutions:

\begin{lemma}
\label{LogExistenceTime1}
For $s$ sufficiently large, the energy $\E=\E(t)$ of a solution to the full water waves system \eqref{eqn:WaterWavesSystem} with $V_0 = 0$ satisfies equation \eqref{eqn:EnergyEstReduce}. Further assume that the Cauchy data augmenting the system is small: $B = \ee \ll 1$. Then, $\E$ remains bounded on $[0,T(\ee)]$, where
\begin{equation}
\label{eqn:LogExistenceTime1} 
T(\ee) \gtrsim \log\frac{1}{\ee},
\end{equation}
which implies that solutions to the water waves system \eqref{eqn:WaterWavesSystem} persist on a timescale of at least $\bigo(\log\ee^{-1})$.
\end{lemma}
\begin{proof}
We begin by writing $T(\ee) = \frac{1}{2C}\log\ee^{-1}$, where $C>0$ is such that $\dot{\E} \leq C(\E + \E^N)$. We shall proceed by utilizing the bootstrapping principle. Namely, we assume that, for some $0<r<1$, we have $\E(t) \in [0,r]$ for all $0 \leq t \leq T(\ee)$. We will then show that, for $\ee$ sufficiently small, $\E(t)$ is bounded above by $\frac{r}{2}$ for all $0 \leq t \leq T(\ee)$. Via Gr\"{o}nwall's inequality, coupled with $\ee \ll 1$ and $r \in (0,1)$, we obtain
\begin{equation}
\label{eqn:ExistenceTimeEst1}
\E \leq K\ee^{2 - \frac{1}{2}(1 + r^{N-1})} < K\ee \text{ on } [0,T(\ee)],
\end{equation}
where $K > 0$ is such that $\E(0) \leq K\ee^2$. Then, we can take $\ee$ sufficiently small so that
\begin{equation}
\label{eqn:ExistenceTimeEst2}
\E(t) < K\ee < \frac{r}{2} \ \ \forall t \in [0,T(\ee)].
\end{equation}
The bootstrapping principle then gives the desired result.
\end{proof}

\begin{rmk}
There is nothing special about $\frac{1}{2}$ and $\frac{r}{2}$ in the proof of Lemma \ref{LogExistenceTime1}. In fact, we can write $T(\ee) = \frac{h}{C}\log\ee^{-1}$ for $h > 0$ and, as long as $h < 1$, we can take $\ee$ sufficiently small so that
\begin{equation*}
\E < K\ee^{2(1-h)} < \varrho < r.
\end{equation*}
However, given that the lifespan we obtain in Lemma \ref{LogExistenceTime1} is already far from sharp, we are not overly concerned with optimizing these constants.
\end{rmk}

In addition to the small-data result of Theorem \ref{LogExistenceTime1}, we also want to deduce a simple $\bigo(1)$ lifespan in the case of large data when $V_0=0$. We do so presently.
\begin{lemma}
\label{LargeDataExistence}
Consider the energy of a solution to \eqref{eqn:WaterWavesSystem}, where we still take $V_0=0$. The energy of such a solution satisfies \eqref{eqn:EnergyEstReduce} as we have noted several times already. Then, $\E$ remains bounded on $[0,T(B)]$, where
\begin{equation}
\label{eqn:LargeDataExistence1}
T(B) \gtrsim \frac{1}{B^{N-1}}.
\end{equation}
In other words, solutions to \eqref{eqn:WaterWavesSystem} with large Cauchy data have at least an $\bigo(\frac{1}{B^{N-1}})$ lifespan. Recall again that $B$ is defined in \eqref{eqn:BDef}.
\end{lemma}
\begin{proof}
Observe that, if $\E \sim 1$, we can rewrite \eqref{eqn:EnergyEstReduce} to obtain
\begin{equation}
\label{eqn:EnergyEstReduceAgain}
\frac{d\E}{dt} \lesssim \E^N.
\end{equation}

Now, write $T(B) = \frac{h}{C}\frac{1}{B^{N-1}}$, where $C>0$ is such that $\dot{\E} \leq C\E$ and $h > 0$ shall be chosen shortly. Recall that, for some $K>0$, we have $\E(0) \leq KB^2$. Assume that we have $\E(t) \in [0,3KB^2]$ for all $0 \leq t \leq T$. Using Gr\"{o}nwall's inequality, we are able to obtain
\begin{equation}
\label{eqn:LargeDataLifespanEst1}
\E(t) \leq KB^2e^{(3K)^{N-1}h}.
\end{equation}
Then, as it is straightforward to see, we can take $h$ sufficiently small so that $0 \leq \E(t) < 2KB^2$ for all $t \in [0,T(B)]$.
\end{proof}

\subsubsection{Non-Zero Background Current}

Here we shall suppose that $V_0 \neq 0$. In that case, our energy estimate is of the form
\begin{equation}
\label{eqn:EnergyEstReduce2}
\frac{d\E}{dt} \lesssim \Pol(\E) = \E + \E^N + (1+\Vzeronorm)(\sqrt{\E} + \E^M) \lesssim (1+\Vzeronorm)\sqrt{\E} + \E^N. 
\end{equation}
We know from numerical experiments (see \cite{AMEA}) that, in this setting, splash singularities can occur in $\bigo(1)$ time and so an $\bigo(1)$ lifespan is the best we can hope to do. As such, we will just consider large data.

\begin{lemma}
\label{Order1ExistenceTime}
When $V_0 \neq 0$, the energy $\E = \E(t)$ of a solution to \eqref{eqn:WaterWavesSystem} satisfies equation \eqref{eqn:EnergyEstReduce2}. Then, $\E$ remains bounded on $[0,T(B,\Vzeronorm)]$ with
\begin{equation}
\label{eqn:Order1ExistenceTime}
T(B,\Vzeronorm) \gtrsim \frac{1}{(1+\Vzeronorm)^2} \wedge \frac{1}{B^{N-1}}.
\end{equation}
So, solutions in this setting have at least an $\bigo(1)$ lifespan.
\end{lemma}
\begin{proof}
We begin by observing that we can rewrite the energy estimate \eqref{eqn:EnergyEstReduce2} as follows:
\begin{equation*}
\frac{d\E}{dt} \lesssim (1+\Vzeronorm)^2 + \E + \E^N \lesssim (1+\Vzeronorm)^2 + \E^N.
\end{equation*}

We shall again proceed by utilizing the bootstrapping principle, supposing that $\E(t) \in [0,4KB^2]$ for all $t \in [0,T(B,\Vzeronorm)]$, where $K > 0$ is such that $\E(0) \leq KB^2$. Write
\begin{equation*}
T = \frac{h}{C}\left( \frac{1}{(1+\Vzeronorm)^2}\wedge\frac{1}{B^{N-1}} \right)
\end{equation*}
with $C>0$ such that $\dot{\E} \leq C((1+\Vzeronorm)^2 + \E^N)$ and $h>0$ to be chosen soon. Then, Gr\"onwall's inequality gives
\begin{equation}
\E(t) \leq \left( KB^2 + hC^{-1} \right)e^{(3K)^{N-1}h}.
\end{equation}
Upon taking $h$ to be sufficiently small, we have $0 \leq \E(t) < 3KB^2$ for all $0 \leq t \leq T(B,\Vzeronorm)$. The bootstrapping principle then gives the desired result.
\end{proof}

\section{The Damped System}

\subsection{Introduction}

We will begin by recalling a bit about the Clamond damper which we introduced in Section 1. We are going to effect the damping via the application of an external pressure to a small portion of the free surface. We shall let $\w \subset [0,2\pi)$ be a connected interval on which we will damp the fluid and let $\chi_{\w}$ be a smooth, non-negative cut-off function supported on $\w$. Then, recall that the Clamond damper is given by
\begin{equation}
\label{eqn:ClamondDamping}
\pex \coloneqq \p_x^{-1}(\chi_{\w}\p_x\vphi) \text{ (modulo a Bernoulli constant).}
\end{equation}
Recall that $\xi(\al) = \al + s_\al \p_\al^{-1}(\cos\tta(\al))$. Given that $x = \xi$ on $\FS$, it follows that we have the following relation at the interface:
\begin{equation}
\label{eqn:x-Deriv}
\p_x = \frac{1}{1 + s_\al\cos\tta(\al)}\p_\al.
\end{equation}
We can then invert $\p_x$ as follows:
\begin{equation}
\label{eqn:x-DerivInverse}
\p_x^{-1}u(\al) = \p_\al^{-1}\big[ (1 + s_\al\cos\tta(\al)) u(\al) \big]
\end{equation}
This allows us to rewrite $\pex$:
\begin{equation}
\label{eqn:NewPext}
\pex = \p_\al^{-1}\big[(1+s_\al\cos\tta)\chi_{\w}(1+s_\al\cos\tta)^{-1}\p_\al\vphi\big] = \p_\al^{-1}(\chi_{\w}\vphi_\al).
\end{equation}
Note that the cut-off function $\chi_{\w}$ acts on $\xi(\al)$, not $\al$, as we want to localize the effects of damping to a region of space.

\subsection{New Evolution Equations}

Given that we will effect the damping via the application of an external pressure, the damping will enter into the evolution equations via a modified pressure at the free surface. Recall from the earlier discussion of the derivation of the evolution equations that the pressure only entered into the $\y_t$ equation. We have, from \cite{AMEA}, that
\begin{equation*}
\y_t = -2p_\al + \frac{\p_\al((V - \Wt \cdot \ut)\y)}{s_\al} - 2s_\al\Wt_t \cdot \ut - \frac{\y\y_\al}{2s_\al^2} - 2g\eta_\al + 2(V - \Wt \cdot \ut)(\Wt_\al \cdot \ut).
\end{equation*}
The modified pressure will be
\begin{equation}
\label{eqn:ModPressure}
p\big\rvert_\FS = -\tau \kappa + \pex = -\frac{\tau}{s_\al}\tta_\al + \p_\al^{-1}\left( \chi_{\w}\vphi_\al \right),
\end{equation}
from which it follows that the damped $\y_t$ equation is
\begin{equation}
\label{eqn:Dampedy_tEquation}
\y_t = \frac{2\tau}{s_\al}\tta_{\al\al} - 2\chi_{\w}\vphi_\al + \frac{\p_\al((V - \Wt \cdot \ut)\y)}{s_\al} - 2s_\al\Wt_t \cdot \ut  - \frac{\y\y_\al}{2s_\al^2} - 2g\eta_\al + 2(V - \Wt \cdot \ut)(\Wt_\al \cdot \ut).
\end{equation}
So, the only difference is that we have picked up a term proportional to $\chi_{\w}\vphi_\al$.

The damped water waves system is then likewise of the form
\begin{equation}
\label{eqn:DampedWaterWavesSystemForm}
\begin{cases} (\id + \K[\Tta])\p_t\Tta = \F_D(\Tta)\\ \Tta(t=0) = \Tta_0  \end{cases},
\end{equation}
where $\F_D$ denote the right-hand side $\F$ with the $\y_t$ equation modified to effect Clamond damping; that is $\F_{D,1} = \F_1$, $\F_{D,2} = \F_2 - 2\chi_{\w}\vphi_\al$, $\F_{D,3} = \F_3$ and $\F_{D,4} = \F_4$. Notice that the compact operator $\K[\Tta]$ is exactly the same as in the undamped water waves system \eqref{eqn:WaterWavesSystemForm}. As before, we will simply prove energy estimates for a model problem and deduce the desired estimates for the full system from the mapping properties of $(\id + \K[\Tta])^{-1}$ (i.e., Lemma \ref{MappingProps1}). Specifically, we consider the following damped model problem:
\begin{equation}
\label{eqn:DampedModelProblem}
\begin{cases} \p_t\Tta = \F_D(\Tta)\\ \Tta(t=0) = \Tta_0  \end{cases}.
\end{equation}

\subsection{Energy Estimates and Analysis}

In this section, we will show that the results obtained for the undamped model problem \eqref{eqn:ModelProblem} also hold for the damped model problem \eqref{eqn:DampedModelProblem}. Given that, as noted above, the only difference in the evolution equations is a term proportional to $\chi_{\w}\vphi_\al$ in the $\y_t$ equation, we only need to ensure that this term does not derail the estimates. We begin by showing that Theorem \ref{UnifEnergyEst} still holds in the presence of Clamond damping. Namely, we have the following theorem:

\begin{thm}
\label{DampedUnifEnergyEst}
We define the energy $\E_{\mathrm{damped}}$ of a solution to \eqref{eqn:DampedModelProblem} in the same way (i.e., via Definition \ref{EnergyDef}). Then, for $s$ sufficiently large, we claim that $\E_{\mathrm{damped}}$ satisfies
\begin{equation}
\label{eqn:DampedUnifEnergyEst}
\frac{d\E_{\mathrm{damped}}}{dt} \lesssim \Pol(\E_{\mathrm{damped}}).
\end{equation}
\end{thm}
\begin{proof}
Notice that $\dot{\E}_{\mathrm{damped}}$ contains the following terms that were not present in the undamped system:
\begin{align}
&-2\Int\y(\chi_{\w}\vphi_\al) \ d\al, \label{eqn:DampedEnergy1}\\
&-\frac{1}{2\tau s_\al}\Int (\p_\al^{j-2}\chi_{\w}\vphi_\al)\La(\p_\al^{j-2}\y) \ d\al, \label{eqn:DampedEnergy2}\\
&-\frac{1}{8\tau^2s_\al^2}\Int \y(\chi_{\w}\vphi_\al)(\p_\al^{j-2}\y)^2 \ d\al, \label{eqn:DampedEnergy3}\\
&-\frac{1}{8\tau^2s_\al^2}\Int \y^2(\p_\al^{j-2}\y)(\p_\al^{j-2}\chi_{\w}\vphi_\al) \ d\al, \label{eqn:DampedEnergy4}
\end{align}
where $2 \leq j \leq s+1$. As we noted above, the only term contributed by the damper is proportional to $\vphi_\al$. The term $\vphi_\al$ may appear unfamiliar, but, in fact, it is a rather routine term. Indeed, we have
\begin{equation}
\label{eqn:PhiAlpha}
\vphi_\al = \p_\al\vphi(\xi(\al),\eta(\al)) = s_\al\nab\vphi\cdot\ut = s_\al\Wt\cdot\ut + \frac{\y}{2}.
\end{equation}
In this way, we see that $\vphi_\al$ is actually a rather familiar term which we have already estimated a number of times, at least at low regularity.

Considering \eqref{eqn:DampedEnergy1}, we see that Lemma \ref{WtEst} in conjunction with the identity \eqref{eqn:PhiAlpha} immediately gives
\begin{equation}
\label{eqn:DampedEnergyEst1}
-2\Int \y(\chiw\vphi_\al) \ d\al \lesssim \pnorm{\y}{2}\pnorm{\vphi_\al}{2} \lesssim \pnorm{\y}{2}\pnorm{\Wt\cdot\ut}{2} + \pnorm{\y}{2}^2 \lesssim \Pol(\Edamp).
\end{equation}
For \eqref{eqn:DampedEnergy2}, we can apply the estimate \eqref{eqn:fLa(g)Est} and Lemma \ref{SobolevMultiplication} to obtain 
\begin{align*}
\frac{1}{2\tau s_\al}\Int (\p_\al^{j-2}\chi_{\w}\vphi_\al)\La(\p_\al^{j-2}\y) \ d\al &\lesssim \Norm{\chi_{\w}\vphi_\al}{\sHalf}\ynorm \lesssim \Norm{\vphi_\al}{\sHalf}\ynorm\\
&\lesssim \ynorm(\Norm{\Wt\cdot\ut}{\sHalf} + \ynorm)
\end{align*}
Notice that Lemma \ref{WtEst} allows us to control all of the terms in $\Norm{\Wt\cdot\ut}{\sHalf}$ except for $\Norm{\BR\cdot\ut}{\sHalf}$. In order to control this term, we represent the Birkhoff-Rott integral using \eqref{eqn:WToHK} and then apply Lemmas \ref{HilbertL2Isometry}, \ref{CompositionEst} and \ref{KopEst}. Doing so gives 
\begin{align}
\label{eqn:BR.tEst}
\Norm{\BR\cdot\ut}{\sHalf} &\lesssim \Norm{\z_\al\HT\left( \frac{\y}{\z_\al} \right)}{\sHalf} + \Norm{\z_\al K[\z]\y}{\sHalf} \nonumber\\
&\lesssim \Norm{\z_a}{\sHalf}\ynorm(1+\Norm{\z_d}{\sHalf}) + \Norm{\z_\al}{\sHalf}\Norm{K[z]\y}{s} \nonumber\\
&\lesssim \ynorm(1+\ttanorm)^2 + \ynorm(1+\ttanorm)^4 \nonumber\\
&\lesssim \Pol(\Edamp).
\end{align}
We can then apply this estimate to finish controlling \eqref{eqn:DampedEnergy2}:
\begin{equation}
\label{eqn:DampedEnergyEst2}
\frac{1}{2\tau s_\al}\Int (\p_\al^{j-2}\chi_{\w}\vphi_\al)\La(\p_\al^{j-2}\y) \ d\al \lesssim \ynorm(\Norm{\Wt\cdot\ut}{\sHalf} + \ynorm) \lesssim \Pol(\Edamp).
\end{equation}

We now apply the estimate \eqref{eqn:BR.tEst} to \eqref{eqn:DampedEnergy3}:
\begin{align}
\label{eqn:DampedEnergyEst3}
\frac{1}{8\tau^2s_\al^2}\Int \y(\chi_{\w}\vphi_\al)(\p_\al^{j-2}\y)^2 \ d\al &\lesssim \norm{\y(\chi_{\w}\vphi_\al)(\p_\al^{j-2}\y)}_\Lp{2}\norm{\p_\al^{j-2}\y}_\Lp{2} \nonumber\\
&\lesssim \ynorm^3\Norm{\vphi_\al}{\sHalf} \nonumber\\
&\lesssim \ynorm^3(\Norm{\Wt\cdot\ut}{\sHalf} + \ynorm) \nonumber\\
&\lesssim \Pol(\Edamp).
\end{align}
Finally, we consider \eqref{eqn:DampedEnergy4} and here we can just use \eqref{eqn:DampedEnergyEst3}. We have
\begin{align}
\label{eqn:DampedEnergyEst4}
\frac{1}{8\tau^2s_\al^2}\Int \y^2(\p_\al^{j-2}\y)(\p_\al^{j-2}\chi_{\w}\vphi_\al) \ d\al &\lesssim \norm{\y(\p_\al^{j-2}\y)}_\Lp{2}\norm{\y(\p_\al^{j-2}\chi_{\w}\vphi_\al)}_\Lp{2} \nonumber\\
&\lesssim \ynorm^3\Norm{\vphi_\al}{\sHalf} \nonumber\\
&\lesssim \Pol(\Edamp).
\end{align}
\end{proof}

\begin{rmk}
\label{MainDampedTheoremProof}
The proofs of Corollary \ref{UnifTime}, Theorem \ref{PrelimReg} and Theorem \ref{FullReg} will either go through in the damped setting exactly as written or require at most minor modifications. Proving damped versions of Theorem \ref{RegExist}, Theorem \ref{RegSolCauchySeq} and Theorem \ref{StabilityThm} require considering energy estimates for the differences. However, as in the above case, the added damping term will cause no problems in these estimates. As such, we omit these calculations. Finally, given that Theorem \ref{UnifEnergyEst} applies to the damped system, all of our results on the lifespan of solutions (Lemma \ref{LogExistenceTime1}, Lemma \ref{LargeDataExistence} and Lemma \ref{Order1ExistenceTime}) also apply to the damped system.
\end{rmk}

Following Remark \ref{MainDampedTheoremProof}, we have the following theorem:

\begin{thm}
\label{DampedModelProblemLWP}
Let $s$ be sufficiently large. The damped model problem \eqref{eqn:DampedWaterWavesSystemForm} is locally-in-time well-posed (in the sense of Hadamard) and the unique solution $\Tta$ is in $C([0,T(B,\Vzeronorm)];\X)$, where $B$ is defined in \eqref{eqn:BDef}. In the context of small Cauchy data $B = \ee \ll 1$, we have
\begin{equation}
\label{eqn:DampedExistenceTime1}
T(\ee) \gtrsim \log\frac{1}{\ee} \text{ for } V_0 = 0.
\end{equation}
For large Cauchy data, we have
\begin{equation}
\label{eqn:DampedExistenceTime2}
T(B,\Vzeronorm) \gtrsim \begin{cases} B^{1-N} & V_0 = 0\\ (1+\Vzeronorm)^{-2} \wedge B^{1-N} & V_0 \neq 0 \end{cases},
\end{equation}
where $N$ is a parameter given in equation \eqref{eqn:PolDef}.
\end{thm}

\begin{rmk}
\label{ProofOfMainThm2}
From Theorem \ref{DampedModelProblemLWP}, we see that the stated claims hold for the damped model problem \eqref{eqn:DampedModelProblem}. By the solvability result of \cite{AMEA} (or of Appendix B) and Lemma \ref{MappingProps1}, we can, exactly as in the undamped case, extend the desired results to the full damped water waves system \eqref{eqn:DampedWaterWavesSystemForm}. This proves Theorem \ref{MainTheorem2}.
\end{rmk}

\appendix

\section{Some Useful Results}

Here we gather some results that are cited repeatedly throughout this paper. In the following, we shall let $\Kb$ denote either $\R$ or $\T$. We shall frequently need to estimate compositions and products of functions. As such, we include two Moser-type inequalities and a product estimate. The product estimate is quite general (it implies the Sobolev algebra property and some other well-known product estimates) and, most importantly, is capable of handling product estimates involving rough functions. We also state a basic Sobolev embedding result, which will be of use.
\begin{lemma}
\label{CompositionEst}
If $F:\R \to \C$ is $C^\infty$ and $u \in H^r \cap L^\infty$ with $r\geq0$, then
\begin{equation}
\norm{F(u)}_{H^r} \lesssim 1+\norm{u}_{H^r};
\end{equation}
the implied constant depends on $\norm{F^{(j)}(u)}_\Lp{\infty}$ for $j$ between $0$ and $r$.
\end{lemma}
\begin{proof}
See section 3.1 in \cite{T2} or Proposition 3.9 in \cite{T1}.
\end{proof}

\begin{lemma}
\label{SobolevEmbedding}
If $r > \frac{1}{2}$, then $H^r \hookrightarrow L^\infty$. In addition, if $r > \frac{3}{2}$, then $H^r \hookrightarrow \mathrm{Lip}$. Further, these embeddings are continuous.
\end{lemma}

\begin{rmk}
Lemma \ref{SobolevEmbedding} implies that Lemma \ref{CompositionEst} applies to any $u \in H^r$ provided $r > \frac{1}{2}$.
\end{rmk}

\begin{lemma}
\label{SobolevMultiplication}
Suppose that $u \in H^r$ and $v \in H^t$ with $r + t > 0$. Then, for all $\sigma$ satisfying $\sigma \leq \min(r,t)$ and $\sigma < r + t - \frac{1}{2}$, we have $uv \in H^\sigma$ with the following estimate:
\begin{equation}
\label{eqn:SobolevMultiplication}
\Norm{uv}{\sigma} \lesssim \Norm{u}{r}\Norm{v}{t}.
\end{equation}
\end{lemma}
\begin{proof}
See Appendix C (Theorem C.10) of \cite{B-GS}.
\end{proof}

Given the prominent role of the Hilbert transform in our analysis, it shall be helpful to establish some of its mapping properties.
\begin{lemma}
\label{HilbertL2Isometry}
The Hilbert transform $\HT$ is an $L^2$-isometry.
\end{lemma}
\begin{proof}
This is a consequence of Plancherel's theorem, combined with the fact that $\HT \coloneqq -i\sgn(D)$.

More specifically, we have the following. Begin with the Hilbert transform $\HT: \Test \to \Test$, defined by $\HT u \coloneqq -i\sgn(D)u$. By Plancherel's theorem, $\HT: \Test \to \Test$ is an isometry. Since $\Test$ is dense in $L^2$, $\HT$ has a unique, densely-defined extension to $L^2$ and, abusing notation mildly, we also denote this extension by $\HT$. Using Plancherel to justify taking the necessary limits, $\HT$ is then an isometry on $L^2$.

\end{proof}

\begin{lemma}
\label{HilbertSobolevNorm}
For $r \in \R$, the Hilbert transform $\HT$ is a continuous (bounded) linear operator on $H^r$; in fact, $\HT$ is an isometry of $H^r$:
\begin{equation*}
\norm{\HT u}_{H^r} = \norm{u}_{H^r}.
\end{equation*}
\end{lemma}
\begin{proof}
Since Fourier multipliers commute, Lemma \ref{HilbertL2Isometry} implies
\begin{equation*}
\Norm{\HT u}{r} = \pnorm{\langle D \rangle^r\HT u}{2} = \pnorm{\HT\langle D \rangle^ru}{2} = \pnorm{\langle D \rangle^ru}{2} = \Norm{u}{r}.
\end{equation*}
\end{proof}

We have the following useful commutator estimates for commutators involving the Hilbert transform:

\begin{lemma}
\label{HTCommutatorEst1}
Let $f \in H^r$ for $r \in \R$. Then, the operator $\comm{\HT}{f}$ is bounded $L^2 \to H^{r-1}$ and $H^{-1} \to H^{r-2}$. Further, for $j=-1,0$, we have
\begin{equation}
\norm{\comm{\HT}{f}(u)}_{H^{r-1+j}} \lesssim \norm{f}_{H^r}\norm{u}_{H^j}.
\end{equation}
\end{lemma}
\begin{proof}
See Lemma 3.7 in \cite{A1}.
\end{proof}

\begin{lemma}
\label{HTCommutatorEst2}
If $f \in H^r$ for $r \geq 3$, then $\comm{\HT}{f}$ is a bounded operator $H^{r-2} \to H^r$. If $f \in H^{r - \half}$ for $r \geq 4$, then $\comm{\HT}{f}$ is a bounded operator mapping $H^{r-2} \to H^{r - \half}$. In addition, for $j \in \big\{0,-\frac{1}{2}\big\}$, we have the estimate
\begin{equation*}
\Norm{\comm{\HT}{f}(u)}{r+j} \lesssim \norm{f}_{H^{r+j}}\norm{u}_{H^{r-2}}.
\end{equation*}
\end{lemma}
\begin{proof}
See Lemma 3.8 in \cite{A1}.
\end{proof}

\section{Invertibility of $\id + \K$}

Our objective in this appendix is to provide a proof of the solvability of the $(\y_t,\w_t,\beta_t)$ system of integral equations in a multiconnected, horizontally-periodic domain with a bottom. Solvability was proved in \cite{AMEA}, but we include this result as it is achieved via alternative means and our approach can be more readily extended to higher dimensions.
In proving that this system is solvable, we follow the work of Schiffer in \cite{Sch}. However, in order to apply these results, we will need to ensure that the periodic Green function defined via the cotangent kernel shares some basic properties with the non-periodic free space Green function. We now turn our attention to this issue.

\subsection{Properties of the Periodic Green Function}

For $x,y \in \R^2$, we denote by $N = N(x,y)$ the fundamental solution to Laplace's equation; that is, $N(x,y) \coloneqq -\frac{1}{2\pi}\log\abs{x-y}$. For $z,w \in \C$, we extend the definition of $N$ in the natural way. Then, we have
\begin{equation}
\label{eqn:NormalDerivFundSoln}
\p_{n_y} N(x,y) = \frac{1}{2\pi} \frac{(x-y)\cdot n_y}{\abs{x-y}^2}
\end{equation}
and subsequently set
\begin{equation}
\label{eqn:Kernels}
k(x,y) \coloneqq \p_{n_y}N(x,y), \ k(z,w) \coloneqq \frac{1}{2\pi}\frac{(z-w)^*n_w^\C}{\abs{z-w}^2},
\end{equation}
where $z = \Comp(x)$, $w = \Comp(y)$ and $n^\C$ satisfies
\begin{equation*}
(a,b) \cdot n_y = \Rea\{(a+ib)^*n_w^\C\}.
\end{equation*}
In this case we have $k(x,y) = \Rea k(z,w)$. Using an identity of Mittag-Leffler, we can transform the integral kernel:
\begin{equation}
\label{eqn:CotKernelDerivation}
\pv\sum_j k(z+2\pi j,w) =  \pv\sum_j \frac{n_w^\C}{2\pi}\frac{1}{z + 2\pi j -w} = \frac{1}{4\pi}n_w^\C\cot\frac{1}{2}(z-w);
\end{equation}
see \cite{AMEA} for details.

For the sake of compactness, we introduce some new notation. Let $\Sigma$ denote the boundary of $\W$; that is, $\Sigma \coloneqq \p\W = \FS \cup \B \cup \Cyl$. As before, $\W$ denotes the fluid domain. Lastly, we make a note regarding the convention we follow with regard to the unit normal since it differs slightly from the convention used until now. In this section, we let $n_P$ denote the inward-pointing normal at $P \in \Sigma$. Hence, for $P \in \FS$, we have $n_P = -\un(\tilde{\al})$, where $\z(\tilde{\al}) = P$.

\begin{lemma}
\label{KernalIntegral}
It holds that
\begin{equation}
\int_\Sigma k(z,P) \ d\sigma(P) = \begin{cases} 1 & z \in \W\\ \frac{1}{2} & z \in \Sigma\\ 0 & z \in \complement\W \end{cases},
\end{equation}
with $\sigma$ denoting surface measure on $\Sigma$.
\end{lemma}
\begin{proof}
We follow the proof in \cite{F} for the non-periodic free space Green function, extending it to the periodic case.\\
$(z \in \complement\W)$ Fix $z \in \complement\W$ and observe that the map $P \mapsto N(z,P)$ is $C^\infty$ in $\overline{\W}$, and harmonic on $\W$. We can therefore apply Green's formula to get
\begin{equation*}
0 = \int_\Sigma \p_{n_P} N(z,P) \ d\sigma(P) = \int_\Sigma k(z,P) \ d\sigma(P),
\end{equation*}
as desired.

\noindent$(z \in \W)$ Fix $z \in \W$, pick $\ee > 0$ such that $B_\ee = B_\ee(z) \Subset \W$, set $\W^\ee = \W - \overline{B}_\ee$ and $S_\ee = S_\ee(z) = \p B_\ee(z)$. Observe that the map $P \mapsto N(z,P)$ satisfies the same conditions as above on $\W^\ee$ as opposed to $\W$. Therefore, following an application of Green's formula, we have
\begin{equation*}
0 = \int_\Sigma k(z,P) \ d\sigma(P) + \int_{S_\ee} k(z,P) \ d\sigma_\ee(P),
\end{equation*}
with $\sigma_\ee$ being the surface measure on $S_\ee$. So, we will need to evaluate $\int_{S_\ee} k(z,\cdot) \ d\sigma_\ee$. First, let us rewrite $k(z,\cdot)$ on $S_\ee$. Notice that $n_P^\C = \ee^{-1}(P - z)$. Write $z-P = \ee e^{i\vartheta}$ for $\vartheta \in [0,2\pi)$ and observe that
\begin{equation*}
k(z,z-\ee e^{i\vartheta}) = -\frac{e^{i\vartheta}}{4\pi}\cot\frac{\ee}{2}e^{i\vartheta} = -\frac{1}{2\pi\ee} + \bigo(\ee),
\end{equation*}
since $\cot z = \frac{1}{z} + \bigo(\abs{z})$. It then follows that
\begin{equation*}
0 = \int_\Sigma k(z,P) \ d\sigma(P) - \frac{\sigma(S_\ee)}{2\pi\ee} + \bigo\left(\int_{S_\ee} \ee \ d\sigma\right) =  \int_\Sigma k(z,P) \ d\sigma(P) - 1 + \bigo(\ee^2).
\end{equation*}
Sending $\ee \to 0^+$ yields the desired equality.

\noindent$(z \in \Sigma)$ Lastly, fix $z \in \Sigma$ and let $\ee > 0$. Set $B_\ee = B_\ee(z)$ and, recalling that $S_\ee = \p B_\ee$, denote $\Sigma^\ee = \Sigma - (\Sigma\cap B_\ee)$, $S_\ee^\pr = S_\ee \cap \W$ and $S_\ee^{\prime\prime} = \{y \in S_\ee : n_z \cdot y < 0 \}$. Again, we observe that the mapping $P \mapsto N(z,P)$ is harmonic in $\W - \overline{B}_\ee$ and $C^\infty$ up to the boundary $\Sigma^\ee \cup S_\ee^\prime$. So,
\begin{equation*}
0 = \int_{\Sigma^\ee} k(z,P) \ d\sigma(P) + \int_{S_\ee^\prime} k(z,P) \ d\sigma_\ee(P).
\end{equation*}
We infer that
\begin{align*}
\int_\Sigma k(z,P) \ d\sigma(P) &= \lim_{\ee \to 0^+} \int_{\Sigma^\ee} k(z,P) \ d\sigma(P) = -\lim_{\ee \to 0^+} \int_{S_\ee^\prime} k(z,P) \ d\sigma_\ee(P) = \lim_{\ee \to 0^+} \Bigg\{ \frac{\sigma_\ee(S_\ee^\prime)}{2\pi\ee} + \bigo\left(\int_{S_\ee^\prime} \ee \ d\sigma_\ee\right) \Bigg\}\\ 
&= \lim_{\ee \to 0^+} \frac{\sigma_\ee(S_\ee^\prime)}{2\pi\ee}.
\end{align*}

So, we need only compute $\sigma_\ee(S_\ee^\prime)$. To this end, we observe that, due to the regularity of the boundary, the symmetric difference of $S_\ee^\prime$ and $S_\ee^{\prime\prime}$ is contained in an ``equatorial strip" with measure $\bigo(\ee^2)$. Whence it follows that $\sigma_\ee(S_\ee^\prime) = \sigma_\ee(S_\ee^{\prime\prime}) + \bigo(\ee^2) = \pi\ee + \bigo(\ee^2)$. Putting this all together, we get
\begin{equation*}
\int_\Sigma k(z,P) \ d\sigma(P) = \lim_{\ee \to 0^+} \bigg\{\frac{\pi\ee + \bigo(\ee^2)}{2\pi\ee}\bigg\} = \frac{1}{2}.
\end{equation*}
This completes the proof.
\end{proof}

For $\phi \in C(\Sigma)$ we may define
\begin{equation*}
u(x) \coloneqq \int_\Sigma k(x,P) \phi(P) \ d\sigma(P).
\end{equation*}
Then, for $h \in \R$ small and nonzero, we define $u_h(P) \coloneqq u(P + hn_P)$ for $P \in \Sigma$ and note that we have $P + hn_P \in \W$ for $h>0$ and $P + hn_P \in \complement\W$ for $h<0$.

\begin{lemma}
\label{BdyJumpRelations}
For $P \in \Sigma$, set 
\begin{equation*}
u_+(P) \coloneqq \lim_{h \to 0^+} u_h(P), \ u_-(P) \coloneqq \lim_{h \to 0^-} u_h(P).
\end{equation*}
Then, we claim that
\begin{equation*}
u_+(P) = -\frac{1}{2}\phi(P) + \int_\Sigma k(P,Q)\phi(Q) \ d\sigma(Q), \ u_-(P) = \frac{1}{2}\phi(P) + \int_\Sigma k(P,Q)\phi(Q) \ d\sigma(Q).
\end{equation*}
\end{lemma}
\begin{proof}
We again follow the proof given in \cite{F} to extend to the periodic case. Fix $P \in \Sigma$ and $h>0$ sufficiently small. Then, as noted above, $P + hn_P \in \W$ and thus
\begin{align*}
u_h(P) &= \phi(P)\int_\Sigma k(P + hn_P,Q) \ d\sigma(Q) + \int_\Sigma k(P+hn_P,Q)(\phi(Q)-\phi(P)) \ d\sigma(Q)\\
&= \int_\Sigma k(P+hn_P,Q)(\phi(Q)-\phi(P)) \ d\sigma(Q).
\end{align*}
Continuity then implies that
\begin{equation*}
\lim_{h \to 0^+} u_h(P) = -\phi(P)\int_\Sigma k(P,Q) \ d\sigma(Q) + \int_\Sigma k(P,Q)\phi(Q) \ d\sigma(Q) = -\frac{1}{2}\phi(P) + \int_\Sigma k(P,Q)\phi(Q) \ d\sigma(Q).
\end{equation*}
On the other hand, for $h<0$, we have
\begin{align*}
u_h(P) &= \phi(P)\int_\Sigma k(P + hn_P,Q) \ d\sigma(Q) + \int_\Sigma k(P+hn_P,Q)(\phi(Q)-\phi(P)) \ d\sigma(Q)\\
&= \phi(P) + \int_\Sigma k(P+hn_P,Q)(\phi(Q)-\phi(P)) \ d\sigma(Q).
\end{align*}
It then follows, again from continuity, that
\begin{equation*}
\lim_{h \to 0^-} u_h(P) = \phi(P) - \phi(P)\int_\Sigma k(P,Q) \ d\sigma(Q) + \int_\Sigma k(P,Q)\phi(Q) \ d\sigma(Q) = \frac{1}{2}\phi(P) + \int_\Sigma k(P,Q)\phi(Q) \ d\sigma(Q).
\end{equation*}
\end{proof}

\subsection{Solvability of the System}

With this machinery in place, we now want to consider the Fredholm eigenvalues of the operator specialized to the water waves problem. We begin by observing that Lemma \ref{BdyJumpRelations} holds in the case of the complexified kernel. That is, if we define $u_+(\wp)$ and $u_-(\wp)$ for complex $\wp \in \Sigma$ in the natural way, then the same jump relations at the boundary given in Lemma \ref{BdyJumpRelations} will hold. We now define the relevant operator $T_k[\cdot]$ by $T_k[\cdot] : \phi \mapsto 2\int_\Sigma k(\cdot,\wp)\phi(\wp) \ d\sigma(\wp)$. We shall let $\phi_\nu$ denote the eigenfunctions of $T_k[\cdot]$ on $S$. In other words, we take the $\phi_\nu$ to solve
\begin{equation}
\label{eqn:EigfxnDef}
\phi_\nu(\cdot) = 2\lambda_\nu\int_\Sigma k(\cdot,\wp)\phi_\nu(\wp) \ d\sigma(\wp) \qquad (\text{on } \Sigma).
\end{equation}
Observe that the $\lambda_\nu$'s aren't exactly the eigenvalues corresponding to the $\phi_\nu$'s, rather the eigenvalues are of the form $\mu_\nu = \lambda_\nu^{-1}$.
Additionally, we define
\begin{equation}
\label{eqn:hDef}
2\lambda_\nu\int_\Sigma k(z,\wp)\phi(\wp) \ d\sigma(\wp) = \begin{cases} h_\nu(z) & z \in \W\\ \tilde{h}_\nu(z) & z \in \complement\W \end{cases}.
\end{equation}
It shall also be worthwhile to consider the complex derivatives of $h_\nu$ and $\tilde{h}_\nu$, which give rise to the dual formulation of the Fredholm eigenvalue problem. Thus, we introduce the holomorphic functions
\begin{equation}
\label{eqn:vDef}
v_\nu(z) = \p_z h_\nu(z), \ \tilde{v}_\nu(z) = \p_z \tilde{h}_\nu(z).
\end{equation}

Then, we can apply Lemma \ref{BdyJumpRelations} to evaluate the limit of the various $h$'s as $z$ tends to the boundary. In particular,
\begin{align}
\label{eqn:hBdyJump}
\lim_{z \to \varpi \in \Sigma} h_\nu(z) &= -\lambda_\nu\phi_\nu(\varpi) + \lambda_\nu\int_\Sigma k(\varpi,\wp)\phi_\nu(\wp) \ d\sigma(\wp) \nonumber\\
&= (1-\lambda_\nu)\phi_\nu(\varpi). \nonumber\\
\lim_{z \to \varpi \in \Sigma} \tilde{h}_\nu(z) &= \lambda_\nu\phi_\nu(\varpi) + \lambda_\nu\int_\Sigma k(\varpi,\wp)\phi_\nu(\wp) \ d\sigma(\wp) \nonumber\\
&= (1+\lambda_\nu)\phi_\nu(\varpi)
\end{align}
Further, it clearly holds that
\begin{equation}
\label{eqn:hBdyDeriv}
\p_nh_\nu^j\big\rvert_\Sigma = \p_n\tilde{h}_\nu\big\rvert_\Sigma.
\end{equation}
If we let $z=z(s)$ parameterize $\Sigma$ by arclength, then we can combine \eqref{eqn:hBdyJump} and \eqref{eqn:hBdyDeriv} into a single equation relating $v_\nu$ and $\tilde{v}_\nu$:
\begin{equation}
\label{eqn:vBdyRelation}
\tilde{v}_\nu(z)\frac{dz}{ds} = \frac{1}{1-\lambda_\nu}v_\nu(z)\frac{dz}{ds} + \frac{\lambda_\nu}{1-\lambda_\nu}\overline{v_\nu(z)\frac{dz}{ds}} \qquad (z=z(s)).
\end{equation}
Utilizing \eqref{eqn:vBdyRelation}, we can formulate a set of integral equations solved by the $v$'s:
\begin{align}
-2\lambda_\nu\int_{\W} \overline{\frac{v_\nu(w)}{(w-z)^2}} \ dm^2(w) &= \begin{cases} v_\nu(z) & z \in \W\\ (1-\lambda_\nu)\tilde{v}_\nu(z) & z \in \complement\W \end{cases}, \label{eqn:vIntEqn}\\
2\lambda_\nu\int_{\complement\W} \overline{\frac{\tilde{v}_\nu(w)}{(w-z)^2}} \ dm^2(w) &= \begin{cases} (1+\lambda_\nu)v_\nu(z) & z \in \W\\ \tilde{v}_\nu(z) & z \in \complement\W \end{cases} \label{eqn:vtIntEqn}
\end{align}
where $m^2$ denotes two-dimensional Lebesgue measure. See \cite{Sch} for further details.

We now see that the periodic $h$ and $v$ functions defined via the cotangent kernel satisfy the same boundary jump relations as those defined via the non-periodic free space Green function. We can utilize the boundary jump relations of \eqref{eqn:hBdyJump} to prove that
\begin{equation}
\label{eqn:Schiffer22}
\int_\W \abs{v_\nu}^2 \ dm^2 = \frac{\lambda_\nu+1}{\lambda_\nu-1}\int_{\complement\W} \abs{\tilde{v}_\nu}^2 \ dm^2;
\end{equation}
see \cite{Sch} for details. As in \cite{Sch}, we deduce from \eqref{eqn:Schiffer22} that $\abs{\lambda_\nu} \geq 1$. What remains then is to show that $\lambda_\nu \neq 1$. Following \cite{Sch} or \cite{F}, we see that, in the non-simply-connected setting, there is a nontrivial kernel corresponding to the integral equations for the $h$ functions. In fact, the kernel is spanned by $\chi_\Cyl$, the characteristic function of the boundary of the obstacle. However, a key point is that in the layer potential formulation of the water waves problem, we are generally more interested in the gradient of the potentials as opposed to the potentials themselves. That is to say, the layer potential formulation of the water waves problem is a ``$v$-problem'' -- what we are really interested in is the kernel corresponding to the $v$'s. Given that the kernel of the $h$ functions is spanned by a constant function, it is clear that the corresponding kernel for the $v$ functions will be trivial. This is exactly as desired for we may now apply the Fredholm alternative to deduce that the inhomogeneous system of integral equations under consideration is solvable (via Neumann series). That is, we have now proved the following theorem.

\begin{thm}
\label{Solvable}
The system of Fredholm integral equations of the second kind for $(\y_t,\w_t,\be_t)$ is solvable.
\end{thm}

\begin{rmk}
\label{CauchyIntegralRmk}
The above analysis also shows that the system arising from the Cauchy integral formulation in \cite{AMEA} is solvable, subject to a minor modification. More particularly, in the case of the Cauchy integral formulation, we still have a Fredholm operator, but this time with a non-trivial kernel and so the operator has a Fredholm pseudoinverse. To see this from the above analysis, we recall that the Cauchy integral formulation is dual to the vortex sheet formulation and corresponds to an ``$h$-problem'', which is dual to the ``$v$-problem''. This implies, as noted above, that the integral equations have a non-trivial kernel, which is spanned by $\chi_\Cyl$, the indicator function of the boundary of the obstacle. Therefore, the system is invertible upon applying a rank-one correction, which projects away from the kernel. This is exactly the process used to invert the system in \cite{AMEA}.
\end{rmk}

\end{document}